\documentclass{amsart}
\usepackage[nohug]{diagrams}
\usepackage[mathscr]{eucal}
\usepackage{graphicx}
\usepackage{amsfonts}
\usepackage{amsmath}
\usepackage{amsthm}
\usepackage{amssymb}
\usepackage{latexsym}
\newfont{\msam}{msam10}

\newtheorem{theorem}[]{Theorem}
\newtheorem{proposition}[]{Proposition}
\newtheorem{corollary}[]{Corollary}
\newtheorem{lemma}[]{Lemma}
\theoremstyle{definition}
\newtheorem{definition}[]{Definition}
\newtheorem{remark}[]{Remark}
\newtheorem{example}[]{Example}
\newtheorem{conjecture}[]{Conjecture}

\let\nc\newcommand

\nc{\la}{\label}

\def\bthm{\begin{theorem}}
\def\ethm{\end{theorem}}
\def\blemma{\begin{lemma}}
\def\elemma{\end{lemma}}
\def\bproof{\begin{proof}}
\def\eproof{\end{proof}}
\def\bprop{\begin{proposition}}
\def\eprop{\end{proposition}}
\def\bcor{\begin{corollary}}
\def\ecor{\end{corollary}}
\def\GG{\mathscr{G}}
\def\AG{\mathscr{A}}
\def\BG{\mathscr{B}}
\def\UG{\mathscr{U}}

\def\Gr{\mbox{\rm{Gr}}^{\mbox{\scriptsize{\rm{ad}}}}}

\def\Z{\mathbb{Z}}

\def\G{\mathcal{G}}
\def\R{\mathcal{R}}

\def\O{\mathcal{O}}
\def\A{\mathbb{A}}

\def\L{\mathfrak{L}}

\def\N{\mathbb{N}}

\def\m{\mathfrak{m}}
\def\c{\mathbb{C}}

\def\CC{\mathcal{C}}

\def\tCC{\tilde{\CC}}

\nc{\Hom}{{\rm{Hom}}}
\nc{\Ext}{{\rm{Ext}}}
\nc{\HOM}{\underline{\rm{Hom}}}
\nc{\EXT}{\underline{\rm{Ext}}}
\nc{\TOR}{\underline{\rm{Tor}}}
\nc{\End}{{\rm{End}}}
\nc{\GL}{{\rm{GL}}}
\nc{\PGL}{{\rm{PGL}}}
\nc{\SL}{{\rm{SL}}}
\nc{\SB}{{\rm{SB}}}
\nc{\sll}{{\mathfrak{sl}}}
\nc{\gl}{{\mathfrak{gl}}}
\nc{\Rep}{{\rm{Rep}}}
\nc{\ad}{{\rm{ad}}}
\nc{\dlim}{\varinjlim}
\nc{\n}{\natural}

\newcommand{\M}{{\mathcal{M}}}

\newcommand{\Spec}{{\rm{Spec}}}
\newcommand{\Span}{{\rm{span}}}
\newcommand{\Stab}{{\rm{Stab}}}
\newcommand{\Pic}{{\rm{Pic}}}
\newcommand{\Aut}{{\rm{Aut}}}

\newcommand{\id}{{\rm{Id}}}
\newcommand{\Der}{{\rm{Der}}}

\newcommand{\rk}{{\rm{rk}}}
\newcommand{\Tr}{{\rm{Tr}}}

\newcommand{\Ker}{{\rm{Ker}}}

\newcommand{\diag}{{\rm{diag}}}

\newcommand{\im}{{\rm{Im}}}

\newcommand{\bG}{{\boldsymbol{\Gamma}}}

\newcommand{\supp}{{\rm{supp}}}
\newcommand{\into}{\,\,\hookrightarrow\,\,}

\newcommand{\onto}{\,\,\twoheadrightarrow\,\,}
\makeatletter

\newcommand{\Rmnum}[1]{\expandafter\@slowromancap\romannumeral #1@}
\makeatother
\begin{document}
\title[Dixmier Groups]{Dixmier Groups and Borel Subgroups}
%
%
\author{Yuri Berest}
\address{Department of Mathematics, Cornell University, Ithaca, NY 14853-4201, USA}
\email{berest@math.cornell.edu}
%
%
\author{Alimjon Eshmatov}
\address{Department of Mathematics, University of Western Ontario,
London, Ontario N6A 5B7, Canada}
\email{aeshmato@uwo.ca}
\author{Farkhod Eshmatov}
\address{Max-Planck-Institut f\"ur Mathematik, P.O.Box: 7280 53072, 
Bonn, Germany}
\curraddr{School of Mathematics, Sichuan University, Chengdu 610064, China}
\email{eshmatov@mpim-bonn.mpg.de}
\maketitle
\section{Introduction and statement of results}
\la{S1} It is well known that many interesting algebraic groups, including classical infinite families of simple groups, arise as the automorphism groups of finite-dimensional simple algebras (see, e.g., \cite{GP}, \cite{KMRT} and references therein).
In this paper, we will examine an infinite-dimensional example of this phenomenon. We
will study a family of ind-algebraic groups associated with algebras Morita equivalent to the  Weyl algebra $ A_1(\c) $. Recall that $ A_1(\c) $ is a simple associative $\c$-algebra isomorphic to the ring
of differential operators on the affine line $ \c^1 $. The algebras Morita equivalent
to $ A_1 $ can be divided into two separate classes: the matrix algebras over $A_1 $
and the rings $ D(X) $ of differential operators on the rational singular curves $X$ with
normalization $ \tilde{X} \cong  \c^1  $ (see \cite{SS}). The matrix algebras $ \M_{k}(A_1) $
are classified, up to isomorphism, by their index (the dimension of matrices).
A remarkable and much less obvious fact\footnote{This fact was first established
in \cite{K} following an earlier work of G.~Letzter and L.~Makar-Limanov (see \cite{LML, Le}). It was rediscovered independently by G.~Wilson and the first author in \cite{BW1}. A conceptual proof and explanations can be found in the survey paper \cite{BW4}.} is that the rings $ D(X) $ are also classified, up to isomorphism, by a single non-negative integer, which is called the differential genus of $ X $
(see \cite{BW4}). In the present paper, we will focus on the automorphism groups of $ D(X) $:
for each $ n \ge 0  $, we choose a curve $ X_n $ of differential genus $ n $,
with $ X_0 = \c^1  $, and write $ G_n $ for the corresponding automorphism
group $ \Aut_{\c}\,D(X_n) $. The group
$ G_0 $ is thus the automorphism group of $ A_1 $ originally studied
by J.~Dixmier \cite{D}. We therefore call $ \{G_n\} $ the {\it Dixmier groups}.
A theorem of Makar-Limanov \cite{ML2} asserts that $ G_0 $ is isomorphic to
the group $ G $ of symplectic (unimodular) automorphisms of the free associative algebra
$ R = \c\langle x,y\rangle $, the isomorphism $ G \stackrel{\sim}{\to} G_0 $ being induced
by the natural projection $ R \onto A_1 $. We will use this isomorphism to
identify $ G_0 $ with $ G $; the groups $ G_n $ for $ n \ge 1 $ can then be naturally
identified with subgroups of $ G $. To explain this in more detail we introduce our main characters: the Calogero-Moser varieties
$$
\CC_n := \{(X,Y) \in \M_n(\c) \times \M_n(\c)\,:\, \rk([X,Y] + I_n) = 1\}/\!/ \PGL_n(\c)\ .
$$
Named after a class of integrable systems in classical mechanics (see \cite{KKS})
these algebraic varieties play an important role in several areas,
especially in geometry and representation theory (see, e.g., \cite{N}, \cite{EG}, \cite{E},
\cite{Go} and references therein). They were studied in detail in \cite{W}, where it was shown (among other things) that the $\CC_n $ are smooth affine irreducible complex symplectic varieties of dimension $2n$.
Furthermore, in \cite{BW},  it was shown that each $ \CC_n $ carries a {\it transitive}
$G$-action, which is obtained, roughly speaking, by thinking of $ \CC_n $ as a
subvariety of $n$-dimensional representations of $R$ (see Section~\ref{S2} for a precise
definition). It turns out that $ G_n $ is isomorphic to the stabilizer of a point for this transitive action: thus, fixing a basepoint in
$ \CC_n $,  we can identify $ G_n $ with a specific subgroup of $ G $.
Our general strategy will be to study $ G_n $ in geometric terms, using
the action of $ G $ on $ \CC_n $.

\subsection*{Main results}
Recall that one of the main theorems of \cite{BW} asserts that
the action of $ G $ on the varieties $ \CC_n $ is transitive for all $ n $.
We extend this result in two ways.
\begin{theorem}
\la{ondouble1}
For each $ n \ge 1 \,$, the action of $ G $ on $ \CC_n $ is doubly transitive.
\end{theorem}
\begin{theorem}
\la{ondouble2}
For any pairwise distinct natural numbers
$ (n_1, n_2, \ldots, n_m) \in {\mathbb N}^m $,
the diagonal action of $ G $ on $ \CC_{n_1} \times \CC_{n_2}
\times \ldots \times \CC_{n_m} $ is transitive.
\end{theorem}
The double transitivity means that $G$ acts transitively on the configuration space
$ \CC_{n}^{[2]} $ of (ordered) pairs of points in $ \CC_n $; in other words, the diagonal
action of $ G $ on $ \CC_n \times \CC_n $ has exactly two orbits: the diagonal
$ \Delta = \{(p,p) \in \CC_n \times \CC_n\} $ and its complement
$ \CC_{n}^{[2]}  = (\CC_n \times \CC_n)\!\setminus \!\Delta \, $. One important consequence of this is that the stabilizer of each point in
$ \CC_n $ (in particular, $ G_n $) is a maximal subgroup of $ G $.
Theorem~\ref{ondouble1} thus strengthens the main results of \cite{W2} and
\cite{KT}, where it is shown that $ G_n $ coincides with its normalizer in $G$.
A notable consequence of Theorem~\ref{ondouble2} is that the restriction of the action of $G$ to $ G_n $ is transitive on $ \CC_k $ provided $ k \ne n $.

We actually expect that Theorem~\ref{ondouble1} and Theorem~\ref{ondouble2} are
part of the much stronger

\vspace{1ex}

\noindent
{\bf Conjecture.} {\it For any pairwise distinct natural numbers
$ (n_1, n_2, \ldots, n_m) \in {\mathbb N}^m $ and for any
$ (k_1, k_2, \ldots, k_m) \in {\mathbb N}^m $, the group
$ G $ acts transitively on}
$$
\CC^{[k_1]}_{n_1} \times \CC^{[k_2]}_{n_2} \times \ldots \times \CC^{[k_m]}_{n_m}\ .
$$
Here $ \CC_{n}^{[k]} $ stands for the configuration space of ordered $k$ points in $\CC_n $.
The above conjecture implies, in particular, that $G$ acts {\it infinitely}
transitively on each $ \CC_n $, which is a well-known fact for $ n = 1 $.
To put this in proper perspective we recall that the varieties $ \CC_n $ are examples of
quiver varieties in the sense of Nakajima \cite{N1}. Using the
formalism of noncommutative symplectic geometry, V.~Ginzburg \cite{G} showed  that the main
theorem of \cite{BW} holds for an arbitrary affine quiver variety
$ \CC_{\boldsymbol{n}}(Q) $ in a weaker form: there is an infinite-dimensional
Lie algebra\footnote{In the Calogero-Moser case, the Lie algebra
$ \L(Q) $ is isomorphic to a central extension of the Lie algebra $ \Der_w(R) $
of symplectic derivations of $ R $.  In Section~\ref{S6.2.2}, following the original
suggestion of \cite{BW}, we will show that $ \Der_w(R) $ can be identified with
the Lie algebra of the group $ G $ equipped with an
appropriate affine ind-scheme structure.}  $ \L(Q) $
(canonically attached to the quiver $Q$) which acts {\it infinitesimally}
transitively on $ \CC_{\boldsymbol{n}}(Q) $; in fact, each  $ \CC_{\boldsymbol{n}}(Q) $
embeds in the dual of  $\, \L(Q) \,$ as a coadjoint orbit. The results of the present
paper suggest that Ginzburg's theorem may admit a natural extension to  higher
configuration spaces $ \CC^{[k]}_{\boldsymbol{n}}(Q) $ and their products.
For the further discussion of the
above conjecture and its implications we refer to Section~\ref{Conj} and Section~\ref{S6.2.2}.

\vspace{1ex}

In the second part of the paper we will study $ \{G_n\} $
as ind-algebraic groups. Recall that the notion of an ind-algebraic group goes back to I.~Shafarevich who called such objects simply
infinite-dimensional groups (see \cite{Sh1, Sh2}). The fundamental example is the group $ \Aut(\c^d) $ of polynomial automorphisms of the affine $d$-space. This group (sometimes called
the affine Cremona group) has been extensively studied, especially for $ d=2 $ (see, e.g.,
\cite{J, vdK, Da, GD, Wr, K1, K2,  FuL, FuM}). It is known \cite{Cz, ML1} that as a {\it discrete} group,
$ \Aut(\c^2) $ is actually isomorphic to the automorphism group of the free algebra
$ \c\langle x,y \rangle $ and hence contains each $ G_n $ as a discrete subgroup. However, the ind-algebraic structure that we put on $ G_n $ is
{\it different} (i.e., not induced) from $ \Aut(\c^2) $. This ind-algebraic structure was
originally proposed by G.~Wilson and the first author in \cite{BW}, but
the details were not worked out in that paper.
It is interesting to note that the ind-algebraic group $G$ can be defined in a simpler
and somewhat more natural way than $ \Aut(\c^2) $ and $ \Aut(A_1)$.
The reason for this is the remarkable fact \cite{Di} that (the analogue of) the Jacobian Conjecture
is known to be true for $ \c\langle x,y \rangle $, while it is
still open for the polynomial ring $ \c[x,y] $ and the Weyl algebra $A_1(\c) $.

\vspace{1ex}

Solvable subgroups play a key role in the theory of classical linear
algebraic groups (see \cite{Bo1}) as well as Kac-Moody groups \cite{Ku}.
It is therefore natural to expect that they should also play a role in
the theory of ind-algebraic groups. In this paper, we will study the
Borel subgroups of $ G_n $: our main result is a complete classification
of such subgroups for all $ n $. To begin with,
we recall that a Borel subgroup of a topological group is a connected solvable subgroup
that is maximal among all connected solvable subgroups. The group $ G $ has an obvious
candidate: the subgroup $ B $ of triangular\footnote{also known as de Jonqui\`eres
transformations
in the commutative case} transformations: $\,(x,y) \mapsto (ax + q(y), a^{-1}y + b)\,$, where
$ q(y) \in \c[y] $, $\,a \in \c^* \,$ and $ b \in \c $. It is not difficult
to prove that $ B $ is indeed a Borel subgroup of $ G $; 
moreover, as in the finite-dimensional case, we have the following theorem.
\begin{theorem}
\la{G-Borel}
Any Borel subgroup of $ G $ is conjugate to $ B $.
\end{theorem}

For $ n > 0 $, the situation is more interesting. Let $ \mathfrak{B}_n $ denote the set of
all Borel subgroups of $ G_n $ on which $ G_n $ acts by conjugation. We will
show that every Borel subgroup of $ G_n $ is conjugate in $ G $ to a subgroup of $ B \,$:
this defines a $G_n$-equivariant map $ \iota:\,\mathfrak{B}_n \to  B\backslash G $, where
$ G_n $ acts on $ B \backslash G $ by right multiplication. It turns out that, at the
quotient level, the map $ \iota $ induces a canonical injection
\begin{equation}
\la{inB}
\mathfrak{B}_n/{\rm Ad}\,G_n\, \into\, \CC_n/B\ .
\end{equation}
Thus, the Borel subgroups of $ G_n $ are classified (up to conjugation) by
orbits in $ \CC_n $ of the Borel subgroup of $ G $. In general (more precisely,
for $ n \ge  2 $), the map \eqref{inB} is not surjective ---
not every $B$-orbit in $ \CC_n $ corresponds to a Borel subgroup of $ G_n $ ---
however, the image of \eqref{inB} has a nice geometric description
in terms of the $\c^*$-action on $ \CC_n $. To be precise, let
$\, T := \{(a x, a^{-1}y) \,:\, a \in \c^*\} \subset B $ denote the group of
scaling automorphisms, which is a maximal torus in $ G $. We will prove

\begin{theorem}
\la{ThB0}
A $B$-orbit $ {\mathcal O} $ in $ \CC_n$ corresponds to a conjugacy
class of Borel subgroups in $ G_n $ if and only if one of the following
conditions holds:
\begin{enumerate}
\item[(A)] $ T $ acts freely on $ \O \,$.
\item[(B)] $ T $ has a fixed point in $ \O \,$.
\end{enumerate}
The orbits of type $ (A) $ correspond precisely to the abelian Borel
subgroups of $ G_n $, while the orbits of type $ (B) $ correspond
to the non-abelian ones.
\end{theorem}

Each of the two possibilities of Theorem~\ref{ThB0} actually occurs:
the orbits of type (A) exist in $ \CC_n $ for $ n \ge 3 $,
while the orbits of type (B) exist for all $n$. Thus, in general,
$ G_n $ has both abelian and non-abelian Borel subgroups. While
the existence of abelian Borel subgroups remains mysterious to us,
we have a fairly good understanding of the non-abelian ones.
It is known (see \cite{W}) that the $T$-fixed points in $ \CC_n $ are represented by
nilpotent matrices $ (X,Y) $ and the latter are classified by the partitions
of $ n $. We will show that the $T$-fixed points actually belong to
{\it distinct} $B$-orbits, which are closed in $ \CC_n $.
Thus Theorem~\ref{ThB0} implies
\begin{theorem}
\la{ThB2}
The conjugacy classes of non-abelian Borel subgroups of $ G_n $ are in
bijection with the partitions of $n$.
In particular, for each $ n \ge 0 $, there are exactly $p(n)$
conjugacy classes of non-abelian Borel subgroups in $ G_n $.
\end{theorem}

The last result that we want to state in the Introduction provides an abstract
group-theoretic characterization of non-abelian Borel subgroups of $G_n$.
\begin{theorem}
\la{ThB1}
An non-abelian subgroup $H$ of $G_n$ is Borel if and only if
\begin{enumerate}
\item[(B1)] $H$ is a maximal solvable subgroup of $G$.
\item[(B2)] $H$ contains no proper subgroups of finite index.
\end{enumerate}
\end{theorem}
Theorem~\ref{ThB1} is an infinite-dimensional generalization of a classical
theorem of R.~Steinberg \cite{Ste} that characterizes (precisely by properties
(B1) and (B2)) the Borel subgroups in reductive affine
algebraic groups. However, unlike in the finite-dimensional case, Steinberg's characterization
does not seem to extend to {\it all}\, Borel subgroups of $ G_n $ (in fact,
even for $n=0$, there exist abelian subgroups that satisfy (B1) and
(B2) but are countable and hence totally disconnected in $G$).

Theorem~\ref{ThB2} and Theorem~\ref{ThB1} combined together imply the
following important

\begin{corollary}
\la{niscor}
The groups $G_n$ are pairwise non-isomorphic $($as abstract groups$)$.
\end{corollary}

In fact, the groups $ G_n $ are distinguished from each other by
the sets of conjugacy classes of their non-abelian Borel subgroups:
by Theorem~\ref{ThB2}, these sets are finite and distinct, while by
Theorem~\ref{ThB1}, they are independent of the algebraic structure.

Although the Borel subgroups of $ G_n $ have geometric origin and their
classification is given in geometric terms, our proofs of Theorem~\ref{ThB0} and
Theorem~\ref{ThB1} are not entirely geometric nor algebraic.
The crucial ingredient is Friedland-Milnor's classificaition of polynomial automorphisms of $ \c^2 $ according to their dynamical properties (see \cite{FM}). This classification was refined by Lamy \cite{L} who extended it to a classification of subgroups of $ \Aut(\c^2) $. We will identify $ G $ as a {\it discrete} group with $ \Aut(\c^2) $ and use Lamy's classification as a main
tool to study the subgroups of $ G $.

In the end, we mention that the original goal of the present paper was to prove
the result of Corollary~\ref{niscor}. Our interest in this result is motivated
by the following generalization of the Dixmier Conjecture (for $ A_1 $)
proposed in \cite{BEE}.

\vspace{1ex}

\noindent
{\bf Conjecture.} {\it For all $\, n, m \ge 0 \,$,
\begin{equation}\la{strdix}
\Hom(D_n,\,D_m) = \left\{
\begin{array}{lll}
\! \! \varnothing & \mbox{if}\quad n \ne m \\*[1ex]
\! \! G_n & \mbox{if} \quad n = m
\end{array}
\right.
\end{equation}
where {\rm `$ \Hom $'} is taken in the category of unital associative $\c$-algebras.}
\noindent

\vspace{1ex}

Corollary~\ref{niscor} implies that the endomorphism monoids $ \Hom(D_n, D_n) $
are pairwise non-isomorphic for different $n$. Still,
we do not know whether the above conjecture is actually stronger than the
original Dixmier Conjecture which is formally the special case of
\eqref{strdix} corresponding to $n=m=0$.

\vspace{0.7ex}

The paper is organized as follows. In Section~\ref{S2}, we introduce notation, review
basic facts about the Calogero-Moser spaces, the Weyl algebra and automorphism groups.
This section contains no new results (except, possibly, for the proof of
Theorem~\ref{T1}, which has not appeared in the literature).

In Section~\ref{S3}, after recalling
elementary facts about doubly transitive actions, we prove
Theorem~\ref{ondouble1} (Section~\ref{S3.4}) and Theorem~\ref{ondouble2}
(Section~\ref{S3.5}). The main consequences of these theorems are discussed in
Section~\ref{maincor} and related conjectures in Section~\ref{Conj}.

In Section~\ref{S4}, we describe the structure of $ G_n $ as a discrete group,
using the Bass-Serre theory of groups acting on graphs. The main result of this
section (Theorem~\ref{T2}) gives an explicit presentation of $ G_n $ in terms of
generalized amalgamated products. This result can be viewed as a generalization
of the classical theorem of Jung and van der Kulk on
the amalgamated structure of $ G $.

In Section~\ref{S6}, we study
$ G_n $ as ind-algebraic groups. After a brief review of ind-varieties
and ind-groups in Section~\ref{S6.1}, we define the structure
of an ind-group on $ G $ in Section~\ref{S6.2} and on $ G_n $ (for $ n \ge 1 $)
in Section~\ref{S6.3}. We show that $G$ is connected (Theorem~\ref{Gconn})
and acts algebraically on $ \CC_n $  (Theorem~\ref{GnAlg}). The connectedness
of $ G_n $ for $ n > 0 $ is a more subtle issue: we prove that $G_n$ is connected
for $ n = 1 $ and $ n = 2 $ (Proposition~\ref{congn}) but leave it as a
conjecture in general. In Section~\ref{S6.2.2}, we define another natural ind-algebraic
structure on $ G $ that makes it an affine ind-group scheme $ \G $.
We show that $ \G \not\cong G $ as ind-schemes (Proposition~\ref{noniso}) and
identify the Lie algebra of $ \G $ in terms of derivations of $ R $,
confirming a suggestion of \cite{BW}.

The main results of the paper are proved in Section~\ref{S7}. Specifically,
Theorem~\ref{G-Borel} and Theorem~\ref{ThB1} (for $n=0$) are proved in
Section~\ref{S7.3}, where we study the Borel subgroups of $ G $.
Theorem~\ref{ThB0} is proved in Section~\ref{S7.4}, while Theorems~\ref{ThB2}
and~\ref{ThB1} (for $n\ge 1$) in Section~\ref{S7.4.1}.
Finally, in Section~\ref{Sadelic}, we give a geometric construction
of Borel subgroups in terms of singular curves and Wilson's adelic
Grassmannian (see Proposition~\ref{lstab} and Corollary~\ref{corbor}). We
also give a complete list of representatives of the conjugacy classes of non-abelian
Borel subgroups of $ G_n $ for $ n = 1,2,3,4\,$.

\subsection*{Acknowledgments}{\footnotesize
We thank P. Etingof and G. Wilson for interesting suggestions, questions and comments.
We also thank S.~Lamy and J.~P.~Furter for answering our questions and
guiding us through the literature.  Yu. B. is grateful to Forschungsinstitut f\"ur Mathematik 
(ETH, Z\"urich) and MSRI (Berkeley) for their hospitality and support during the period
when this work was carried out. F.~E. is grateful to the Max-Planck-Institut f\"ur Mathematik
(Bonn), IH\'ES (Bures-sur-Yvette) and the Mathematics Department of Indiana University (Bloomington). This work was partially supported by NSF grant DMS 09-01570.}

\section{Preliminaries}
\la{S2}
In this section, we fix notation and review basic facts from the literature needed for the present paper.

\subsection{The Calogero-Moser spaces}
\la{S2.1}
For an integer $ n \ge 1 $, let $ \M_n(\c) $ denote the space of complex $ n \times n $ matrices.  Let $ \tCC_n \subseteq \M_n(\c) \times \M_n(\c) $ be the subvariety of pairs of matrices
$ (X,Y) $ satisfying the equation
\begin{equation}
\la{rkone}
{\rm rank}\,([X,\,Y] + I_n) = 1 \ ,
\end{equation}
where $ I_n $ is the identity matrix in $ \M_n(\c) $. It is easy to see that
$ \tCC_n $ is stable under the diagonal action of $ \GL_n(\c) $ on $ \M_n(\c) \times \M_n(\c) $ by
conjugation of matrices, and the induced action of $ \PGL_n(\c) $ on $ \tCC_n $ is free. Following \cite{W}, we define the {\it $n$-th Calogero-Moser space} to be the quotient variety $\, \CC_n := {\tCC}_n/\PGL_n(\c) \,$.  It is shown in \cite{W} that $ \CC_n $ is a smooth irreducible affine variety of dimension $2n$.

It is convenient to make sense of $\, \CC_n  \,$ for $ n = 0 \,$: as in \cite{W}, we simply
assume that $ \CC_0 $ is a point, and with this convention, we set
$$
{\mathcal C} := \bigsqcup_{n \geq 0} {\mathcal C}_n \ .
$$
Abusing notation, we will write $ (X,\,Y) $ for a pair of matrices in $ \tCC_n $
as well as for the corresponding point (conjugacy class) in $ \CC_n $.

The Calogero-Moser spaces can be obtained  by (complex) Hamiltonian reduction
({\it cf.} \cite{KKS}): specifically,
\begin{equation}
\la{iscm}
\CC_n  \cong \mu^{-1}(I_n)/\GL_n(\c)\ ,
\end{equation}
where $\, \mu:\,   T^*(\gl_n  \times \c^n)  \to \gl_n \, , \,  (X,Y,v,w) \mapsto  - [X,Y]  + vw\,$
is the moment map corresponding to the symplectic action of $\, \GL_n\,$ on  the
cotangent bundle $ T^*(\gl_n \times \c^n) \,$. With natural identification $\, T^*(\gl_n \times \c^n) \cong
\M_n(\c) \times \M_n(\c) \times \c^n \times (\c^n)^* $,  this
action is given by
\begin{equation}
\la{gln}
(X, Y, v, w) \mapsto (g X g^{-1}, \,g Y g^{-1},\, g v,\, w g^{-1}) \, , \quad g \in \GL_n(\c)\  .
\end{equation}
It is easy to see that the orbit of $\,(X,\, Y,\, v,\, w) \in  \mu^{-1}(I_n) \,$
under \eqref{gln} is uniquely determined by the conjugacy class of
$ (X,\,Y) \in \tCC_n $; whence the isomorphism \eqref{iscm}.

The above construction shows that the Calogero-Moser spaces carry a natural symplectic structure. In fact, it is known that each $ \CC_n $ is a hyperk\"ahler manifold, and the symplectic structure on $ \CC_n $ is just part of a hyperk\"ahler structure (see \cite[Sect.~3.2]{N} and \cite{W}). In this
paper, we will not use the hyperk\"ahler structure and will regard $ \CC_n $ simply as a complex
variety.

\subsection{The group $ G $ and its action on $ \CC_n $}
\la{S2.2}
Let $  R  = \c\langle x,\,y \rangle $  be the free associative algebra on two generators $x$ and $y$.  Denote by $ \Aut(R) $ the automorphism group of $ R $. Every $\, \sigma \in  \Aut(R) \,$
is determined by its action on $ x $ and $ y\, $:  we will write $ \sigma $ as $\,
(\sigma(x),\, \sigma(y)) \,$,
where $ \sigma(x) $ and  $ \sigma(y) $ are noncommutative polynomials in $ R $ given by the images of $ x $ and $ y $ under $ \sigma $.
A fundamental theorem of Czerniakiewics \cite{Cz} and Makar-Limanov \cite{ML1} states
that $ \Aut(R)  $ is generated by the affine  automorphisms:
\begin{equation*}
(ax+by+e,\,cx+dy+f)\ , \quad a,\,b, \ldots,
f \in \c\  ,
\end{equation*}
and the triangular (Jonqui\`ere) automorphisms:
\begin{equation*}
(ax + q(y),\, by+h)\ , \quad a,\,b \in \c^*,\
h \in \c\ ,\quad  q(y) \in \c[y]\  .
\end{equation*}
This fact is often stated by saying that every automorphism of $ R $ is tame.

In this paper, we will study a certain family $ \{G_0, \,G_1,\, G_2,\,\ldots \} $ of subgroups of $  \Aut(R) $ associated with Calogero-Moser spaces. The first member in this family, which we will often denote simply by $ G $, is the group of  {\it symplectic}  automorphisms of $ R \,$:
\begin{equation}
\la{defG}
G = G_0 := \{\sigma \in  \Aut(R)\, :\, \sigma([x,y]) = [x,y]\}\ .
\end{equation}
The structure of this group is described by the following theorem which is a
simple consequence of the Czerniakiewics-Makar-Limanov Theorem.
\begin{theorem}[\cite{Cz}, \cite{ML1}]
\la{TCML}
The group $ G $ is the amalgamated free product
\begin{equation}
\la{amal}
G =  A *_U B \ ,
\end{equation}
where $\, A \, $ is the subgroup of symplectic affine transformations:
\begin{equation}
\label{A}
(ax+by+e,\,cx+dy+f)\ , \quad a,\,b, \ldots,
f \in \c\ ,\quad  ad-bc=1\ ,
\end{equation}
$\, B \, $ is the subgroup of symplectic triangular transformations:
\begin{equation}
\label{B}
(ax + q(y),\,a^{-1}y+h)\ , \quad a \in \c^*,\
h \in \c\ ,\quad  q(y) \in  \c[y]\ ,
\end{equation}
and $\, U\,$ is the intersection of $ A $ and $ B $ in $ G $:
\begin{equation}
\label{U}
(ax+by+e,\,a^{-1}y+h)\ , \quad a \in \c^*,\
b,\,e,\, h \in \c \ .
\end{equation}
\end{theorem}

Theorem~\ref{TCML} can be deduced from the well-known result of
Jung \cite{J} and van der Kulk \cite{vdK} on the structure of the automorphism
group of the polynomial algebra $ \c[x,y] $ in two variables. The key observation
of \cite{Cz} and \cite{ML1} was the following
\begin{proposition}
\la{PCML}
The natural projection $ R \onto \c[x,y] $ induces an isomorphism of
groups $\, \Aut(R) \stackrel{\sim}{\to}  \Aut\,\c[x,y]  $. Under this isomorphism,
 $ G $ corresponds to the subgroup $ \Aut_{\omega}\,\c[x,y] $ of Poisson
automorphisms (i.e. those with Jacobian $1$).
\end{proposition}

\vspace{1ex}

\begin{remark}
Proposition~\ref{PCML} implies that the natural action of $ G $ on $ \c^2 $
is faithful; this allows one to identify $ G $  with a
subgroup of $ \Aut(\c^2) $ and view the elements of $ G $ as polynomial
automorphisms of $ \c^2 $; we will use this identification in Section~\ref{S7}.
For a detailed proof of the Jung-van der Kulk Theorem as well as Proposition~\ref{PCML}
we refer to \cite{C} (see, {\it loc. cit.}, Theorem~6.8.6 and Theorem~6.9.3, respectively).
A direct proof of Theorem~\ref{TCML} can be found in \cite{Co}.
\end{remark}

\vspace{1ex}

Theorem~\ref{TCML} implies that $ G $ is generated by the automorphisms
\begin{equation}
\label{Phs}
\Phi_p \,:=\, (x,\,y + p(x))\ ,\qquad \Psi_q\,:=\,(x + q(y),\,y)\ ,
\end{equation}
where $\,p(x) \in \c[x]\,$ and $\,q(y) \in \c[y]\,$. We denote the corresponding
subgroups of $ G $ by $\, G_x := \langle \Phi_p \,:\, p \in \c[x]\rangle \,$ and
$\, G_y := \langle \Psi_q\,:\, q \in \c[y] \rangle \,$. These are precisely the
stabilizers of $ x $ and $ y $ under the natural action of $ G $ on $ R $.

Next, following \cite{BW}, we define an action of $G$ on the Calogero-Moser spaces $ \CC_n $.
First, thinking of pairs of matrices $ (X,\,Y) $ as points dual to
the coordinate functions $ (x,y) \in R $, we let $G$ act on $  \M_n(\c) \times \M_n(\c) $ by
\begin{equation}
\la{gact}
(X,\,Y) \mapsto (\sigma^{-1}(X),\, \sigma^{-1}(Y))\ ,
\quad \sigma \in G\ .
\end{equation}
Since $ G $ preserves commutators, this action restricts to the subvariety
$ \tCC_n $ of $  \M_n(\c) \times \M_n(\c) $ defined by \eqref{rkone} and commutes with
the conjugation-action by $\PGL_n(\c)$. Hence \eqref{gact} defines an action of $G$ on
$ \CC_n $. Note that, for $ n = 1 $, the action of $ G $ on $ \CC_1 = \c^2 $ agrees with
the natural one coming from Proposition~\ref{PCML}.

Knowing the structure
of the group $ G $ (more precisely, the fact that $G$ is generated by the triangular automorphisms \eqref{Phs}), it is easy to see that $G$ acts on $ \CC_n $ symplectically and algebraically.
Much less obvious is the following fact.
\begin{theorem}[\cite{BW}]
\la{TBW}
For each $ n \ge 0 $, the action of $ G $ on $ \CC_n $ is transitive.
\end{theorem}

Theorem~\ref{TBW} plays a crucial role in the present paper. First, we use this theorem to define the groups $ G_n $ for $ n \ge 1 $: we let $ G_n $ be the stabilizer of a point in $ \CC_n $ under the action  of $ G $. By transitivity, this determines  $ G_n $ uniquely up to conjugation in $G$. To do computations it will be convenient for us to choose specific representatives in each conjugacy
class $ [G_n] $; to this end we fix a basepoint $ (X_0,\,Y_0) \in \CC_n $ with
\begin{equation}
\label{base}
X_0 = \begin{pmatrix}
0&0&0& \ldots &0 \\*[1ex]
1&0&0& \ldots &0\\
0&1&0& \ddots & \vdots\\
\vdots& \vdots & \ddots& \ddots & 0\\
0&0& \ldots & 1 &0
\end{pmatrix}
\ , \quad
Y_0 = \begin{pmatrix}
0&1-n&0& \ldots &0 \\*[1ex]
0&0&2-n& \ldots &0\\
0&0&0& \ddots & \vdots\\
\vdots& \vdots & \ddots& \ddots & -1\\
0&0& \ldots & 0 &0
\end{pmatrix}\ ,
\end{equation}
and set
\begin{equation}
\label{gn}
G_n := \mbox{\rm Stab}_{G}(X_0, Y_0)\ ,\quad n \ge 1\ .
\end{equation}
\subsection{The Calogero-Moser correspondence}
\la{S2.3}
Next, we recall the connection between the Calogero-Moser spaces
and the Weyl algebra $\, A_1(\c) := R/\langle xy-yx- 1 \rangle \,$ described in \cite{BW}.  In his 1968 paper \cite{D},  Dixmier proved that the automorphism group of $ A_1 $ is generated by the same transformations \eqref{Phs} as the group $ G $. This result was refined by Makar-Limanov \cite{ML2} who showed that the analogue of Proposition~\ref{PCML} also holds for $A_1(\c) $: namely, the natural projection $\, R \onto A_1 \,$ induces an isomorphism of groups
\begin{equation}
\la{MLiso}
G \stackrel{\sim}{\to} \Aut(A_1)\ .
\end{equation}

Identifying $ G = \Aut(A_1) $ via \eqref{MLiso},
we will look at the action of $ G $ on the space of ideals of $ A_1 $. To be precise, let
$ \R = \R(A_1) $ denote the set of isomorphism classes of nonzero right ideals of $A_1 $. The automorphism group of $ A_1 $ acts naturally on
the set of all right ideals (one simply treats an ideal as a subspace of $A_1$), and this action is compatible with isomorphism. Thus, we get an action:
$\,G  \times \R \to \R\,$,$\, (\sigma, [M]) \mapsto [\sigma(M)]\,$.
The following result is another main ingredient of the present paper.
\begin{theorem}[\cite{BW}]
\la{T4}
There is a bijective map $\,\omega:\, \CC \to \R \,$ which is equivariant under the action of $G$.
\end{theorem}
Note that in combination with Theorem~\ref{T4}, Theorem~\ref{TBW} shows that $ \omega(\CC_n) $ are precisely the orbits of $ G $ in $ {\mathcal R} $.
The  map $ \omega $ can be described explicitly as follows ({\it cf.} \cite{BC}). Recall that a point of $ \CC_n $ is represented by a pair of matrices $ (X,Y) $ satisfying
the equation \eqref{rkone}. Factoring $\,[X,\,Y] + I_n = v w  \,$
with $ v \in \c^n $ and $ w \in  (\c^n)^* $, we define the (fractional) right ideal
\begin{equation}
\la{detx}
M(X,Y) = \det(X - x\,I_n)\,A_1\, +\, \chi(X,Y) \cdot \det(Y - y\,I_n)\,A_1\ .
\end{equation}
where $\,\chi(X,Y) := 1 + w\,(X - x\,I_n)^{-1}(Y - y\,I_n)^{-1} v \,$  is an element
of the quotient field of $ A_1 $. Now, the assignment $\,(X,Y) \mapsto M(X,Y)\,$
induces a map from $ \CC_n $ to the set of isomorphism classes of ideals of $ A_1 $;
amalgamating such maps for all $ n $ yields the required bijection $\,\omega: \CC
\stackrel{\sim}{\to} \R\,$. Substituting the matrices \eqref{base} in \eqref{detx},
we find that the basepoint $ (X_0, Y_0) \in  \CC_n $ corresponds to (the class of)
the ideal
\begin{equation}
\la{id00}
M(X_0, Y_0) = x^{n} A_1 + (y+n x^{-1})\,A_1 \ .
\end{equation}
We will denote the ideal \eqref{id00} by $ M_n $ and write $ D_n := \End_{A_1}(M_n) $
for its endomorphism ring. Note that $ D_0 = A_1 $.

\subsection{Automorphism groups}
\la{S2.4}
Theorem~\ref{T4} allows one to translate algebraic questions about $ A_1 $ and
its module category to geometric questions about the Calogero-Moser spaces and
the action of $ G $ on these spaces.  One important application of this theorem is a
classification of algebras (domains) Morita equivalent to $A_1$. Briefly,
by Morita theory, every such algebra can be identified with
the endomorphism ring of a right ideal in $ A_1 $; by a theorem of Stafford
(see \cite{St}), two such endomorphism rings are isomorphic (as algebras) iff
the classes of the corresponding ideals lie in the same orbit of $ \Aut(A_1) $
in $ \R $. Now, using Theorem~\ref{T4}, we can identify the orbits of
$ \Aut(A_1) $ in $ \R $ with the Calogero-Moser spaces $ \CC_n $.
Thus, the domains Morita equivalent to $A_1$ are classified (up to isomorphism) by the single integer $ n \ge 0 \,$: every such domain
is isomorphic to the algebra $ D_n $, and moreover $ D_n \not\cong D_m $ for $ n \ne m $.
This classification was originally established in \cite{K} by a direct calculation;
it has several interpretations and many interesting implications which the reader
may find in \cite{BW4}. We conclude this section by recording a proof of the
following fact which is mentioned in passing in \cite{BW4}.
\begin{theorem}
\la{T1}
Let $ [M] \in \R $ be the ideal class corresponding
to a point $ (X, Y) \in \CC_n $ under the Calogero-Moser map $ \omega $.
Then, there is a natural isomorphism of groups
$$
{\rm Stab}_G(X,Y)  \stackrel{\sim}{\to} \Aut[\End_{A_1}(M)]\ .
$$
In particular, for all $ n \ge 0 \,$,
\begin{equation}
\la{iso_n}
G_n  =                                                                                                                                                               \Aut(D_n)\  .
\end{equation}
\end{theorem}
\bproof
First, we note that $ \Aut[\End_{A_1}(M)] $ can be naturally identified with a
subgroup of $\Aut(A_1) $. To be precise, let
$ \Pic(A) $ denote the Picard group of a $\c$-algebra $A$.
Recall that $ \Pic(A) $ is the group of $\c$-linear Morita equivalences of
the category of $ A$-modules; its elements are
represented by the isomorphism classes of invertible $A$-bimodules $\,P\,$.
There is a natural group homomorphism
$\,\alpha_A:\, \Aut(A) \to \Pic(A) \,$, taking $\, \tau \in \Aut(A) \,$ to the class of
the bimodule $ [{}_1 A_{\tau}] $, and if $\, D \,$ is a ring Morita equivalent to $ A $,
with a progenerator $ M $, then there is a group isomorphism $\, \beta_M:\, \Pic(D)
\stackrel{\sim}{\to} \Pic(A)\,$ given by
$\,
[P] \mapsto [M^* \otimes_{D} P \otimes_{D} M]\,$.
Now, for $\, A := A_1 \,$ and $\, D := \End_A(M) $, we have the following diagram
\begin{equation}
\la{D1}
\begin{diagram}[small, tight]
\Aut(D) & \rTo^{\ \alpha_{D}\ }  & \Pic(D) \\
\dDotsto^{i_M} &                      & \dTo_{\beta_{M}} \\
\Aut(A)   & \rTo^{\ \alpha_{A}\ }  & \Pic(A) \\
\end{diagram}
\end{equation}
where $ \beta_{M} $ is an isomorphism and the two horizontal maps are injective. A theorem of Stafford (see \cite{St}, Theorem~4.7) implies that $ \alpha_{A} $ is actually an isomorphism.
Inverting this isomorphism, we define the
embedding $\, i_M:\,\Aut(D) \into \Aut(A) \,$, which makes \eqref{D1} a commutative diagram.

Now, writing $ H := {\rm Stab}_G(X,Y) $, we have group homomorphisms
\begin{equation*}
\label{ison}
H \into G \stackrel{\sim}{\to} \Aut(A) \stackrel{i_M}{\hookleftarrow} \Aut(D)\ ,
\end{equation*}
where the first map is the canonical inclusion and the second is the Makar-Limanov
isomorphism \eqref{MLiso}. We claim that the image of $ H $ in $\,\Aut(A)\,$ coincides
with the image of $ i_M \,$; this gives the required isomorphism
$\,H \stackrel{\sim}{\to} \Aut(D)\,$. In view of Theorem~\ref{T4}, it suffices
to show that
$$
\im(i_M) = \{\,\tau \in \Aut(A)\ :\ \tau(M) \cong M\,\}\ .
$$
First, we prove the inclusion
$\, \im(i_M) \subseteq \{\tau \in \Aut(A)\, :\, \tau(M) \cong M\,\}\,$.
Given $\, \sigma \in \Aut(D) \,$, $\, i_M(\sigma) \,$ is defined to be
the (unique) automorphism $ \tau \in \Aut(A) $ such that
\begin{equation}
\la{E2}
{}_1A_{\tau} \cong M^* \otimes_D ({}_1D_\sigma) \otimes_D M \quad
(\mbox{as $A$-bimodules})
\end{equation}
The right-hand side of \eqref{E2} can be identified with the subspace
$ M^* \sigma(M) \subseteq Q $ in the quotient field of $A$,
and we denote by $\, f \,$ the corresponding isomorphism
$$
{}_1 A_\tau \stackrel{\sim}{\to} M^* \otimes_D({}_1 D_\sigma)
\otimes_D M \stackrel{\sim}{\to} M^*\sigma(M) \ .
$$
Then, for any $ a \in M \subseteq A $, we have $\,f(a) = f(a.1) =
a\,f(1)\,$. On the other hand, $ f(a) = f(1.a) =
f(1)\sigma(\tau^{-1}(a))$. Thus, writing $ b = f(1) $, we see that
$$
\tau(M) = b \sigma(M) b^{-1} = b M M^*\sigma(M) b^{-1} = b M\ .
$$

Conversely, suppose that $ \tau(M) = bM $ for some $ b \in M^* $.
Then $ \tau(D) = \tau(M M^*) = b M M^* b^{-1} = b D b^{-1} $ in $Q$. If
we let $ \sigma := {\rm Ad}_b \circ \tau \in \Aut(D) $, where
$\,{\rm Ad}_b:\,a \mapsto b^{-1} a b \,$, then it is easy to see
that $\,\tau = i_M(\sigma)\,$.
\eproof
\begin{remark}
The above proof shows that for $ n = 0 $ the isomorphism \eqref{iso_n}
specializes to \eqref{MLiso}. Theorem~\ref{T1} can thus be viewed as an extension of the
Dixmier-Makar-Limanov theorem about $ \Aut(A_1) $  to algebras Morita equivalent to $A_1$.
\end{remark}

\section{Double Transitivity}
\la{S3}
We begin by recalling  basic properties of doubly transitive group actions.
\subsection{Doubly transitive group actions}
\la{S3.1}
Let $ X$ be a set of cardinality $\, |X| \ge 2 \,$.  An action of a group $G$ on $X$  is called
{\it doubly transitive} if for any two pairs $ (x_1, x_2) $ and $ (y_1, y_2) $ of distinct elements
in $X$, there is a $ g  \in G $ such that $g \, x_1 = y_1$ and $g \, x_2 = y_2 $. In other words,
$ G $ acts doubly transitively on $ X $ if the diagonal action of $G$ on $ X \times X $ is transitive
outside the diagonal $ \Delta \subset X \times X $.

Note that a doubly transitive group action is automatically transitive, but the converse
is obviously not true. The next lemma provides some useful characterizations of doubly transitive actions.
\begin{lemma}
\la{Ldt}
Let $ G $ be a group acting on a set $X$ with $ |X| \ge 3 $.
Then the following are equivalent.
\begin{enumerate}
\item[(1)] The action of $G$ on $X$ is doubly transitive.
\item[(2)] For each $ x \in X $, the stabilizer $ \Stab_G(x) $  acts transitively on $ X \setminus \{x\} $.
\item[(3)] $G$ acts transitively on $X$, and there exists $ x_0 \in X $ such that $ \Stab_G(x_0) $
acts transitively on $ X \setminus \{x_0\} $.
\item[(4)] $G$ acts transitively on $X$, and $\, G = H \cup g H g^{-1}\,$, where $ H $ is the
stabilizer of a point in $X$ and $ g \in G \setminus H $.
\end{enumerate}
\end{lemma}
\bproof
We will prove only that (1) $\,\Leftrightarrow \,$ (2) and leave the rest as a (trivial) exercise
to the reader.   Fix $ x \in X $ and choose any $\, y,z \in X \,$ such that
$ x, y, z $ are pairwise distinct. (This is possible since $ |X| \ge 3 $.)
Then, a doubly transitive action admits $ g \in G $ moving
$ y \mapsto z $ while fixing $x$. This proves (1) $\,\Rightarrow \,$ (2).
Conversely, assume that (2) holds. Consider two pairs $(x_1, x_2) $ and
$(y_1, y_2) $ in $ X \times X $ with $ x_1 \ne x_2 $ and $ y_1 \ne y_2 $. If $ x_1 \ne y_2 $, then
we can use elements of $ \Stab_G(x_1) $ and $ \Stab_G(y_2) $  moving
$\, (x_1, x_2) \mapsto (x_1, y_2)  \mapsto (y_1, y_2 )\,$. If $ x_1 = y_2 $, then we choose
$ z  \ne x_1, y_1 $ in $X$ (again, such a $z$ exists since $ |X| \ge 3 $) and use the elements of
$ \Stab_G(x_1) $, $ \Stab_G(z) $ and $ \Stab_G(y_1) $ to move
$\,(x_1, x_2) \mapsto (x_1, z) \mapsto (y_1, z) \mapsto (y_1, y_2) \,$.
%
%
\eproof

\begin{corollary}
\la{Codt}
Suppose  $ G $ acts doubly transitively on a set $X$. Then

$(a)$ the stabilizer of any point of $ X $ is a maximal subgroup of $ G $.

$(b)$ any normal subgroup $ N \lhd G $ acts on $X$ either trivially or transitively.

\end{corollary}
\bproof
$(a)$ Fix $ x \in X $ and let $ H = \Stab_G(x) $. If $\, H \lneqq K \leqq G \,$, then
$ H \cup HgH \subseteq K $ for any $ g \in K \setminus H $. But Proposition~\ref{Ldt}(4)
implies that $ H \cup HgH  = G $. Hence $ K = G $.

$(b)$ Suppose that $ N $ acts nontrivially on $X\,$: i.~e., $\, h\,x \ne x $ for some $x \in X$
and $ h \in N $.  Pick any two distinct elements in $ X $, say $ y $ and $ z $. Then, by
double transitivity, there is $ g \in G $ such that $ y = g\,x $ and $ z = g\,(hx) $.
It follows that $ z = ghg^{-1}(g\,x) = ghg^{-1} y $ and $ ghg^{-1} \in N $, so $ N $
acts transitively on $X$.
\eproof

\begin{remark}
The transitive group actions with maximal stabilizers are called primitive.
The above Corollary shows that any doubly transitive action is primitive.
The converse is not always true: for example, the natural action of
the dihedral group $ D_n $ on the vertices of a regular $n$-gon is
primitive for $ n $ prime but not doubly transitive if $ n \ge 4 $.
\end{remark}
\subsection{Auxiliary results}
\la{S3.3} To prove Theorems~\ref{ondouble1} and~\ref{ondouble2} we will need a few technical
results from the earlier literature. First, following \cite{EG}, we define the map
$$
\Upsilon \, : \, \CC_n \rightarrow \c^n\!/ S_n \times \c^n\!/S_n \ , \quad (X,Y)  \, \mapsto
\,  (\Spec(X),\, \Spec(Y))\ ,
$$
assigning to the matrices $ (X,Y) \in \CC_n $ their eigenvalues. We can write
$ \Upsilon =(\Upsilon_1, \Upsilon_2)$, where $ \Upsilon_1 $ and $ \Upsilon_2$ are
the projections onto the first and second factors, respectively.
The following fact is proved in \cite{EG} (see {\it loc. cit.},
Prop.~4.15 and Theorem~11.16).
\begin{theorem}
\la{EG} The map $\Upsilon$ is surjective.
\end{theorem}

Next, we recall the subgroups $ G_x $ and $ G_y $ of $ G $ generated by the
automorphisms $\, (x,y) \mapsto (x,y+p(x)) \,$ and $\, (x,y) \mapsto (x+q(y),y) \,$
respectively, see \eqref{Phs}. These are precisely the stabilizers of $ x $ and $ y $
under the natural action of $ G $ on $ R = \c\langle x,y\rangle $.
The following simple observation is essentially due to \cite{BW}
(see {\it loc. cit.}, Sect.~10).
\begin{lemma}
\la{W}
Let $ (X,Y) \in \CC_n $.
\begin{enumerate}
\item[(1)]  If $ X $ is diagonalizable, then $ G_x $ acts transitively on
$\Upsilon^{-1}_1(\Spec\,X)$.

\item[(2)]  If $Y$ is diagonalizable, then $G_y$ acts transitively on
$\Upsilon^{-1}_2(\Spec\,Y)$.
\end{enumerate}
\end{lemma}
\begin{proof}
We will only prove $(1)\,$; the proof of $(2)$ is similar.
Assume that $ X $ is diagonal with $ \Spec(X) = \{\lambda_1, \ldots,\lambda_n\} $.
Then, by \cite[(1.14)]{W}, the eigenvalues $ \lambda_i $ are pairwise
distinct, and $ (X,Y) \in \CC_n $ if and only if the matrix
$ Y $ has a standard Calogero-Moser form with off-diagonal entries
$$
Y_{ij} \, = \, (\lambda_i - \lambda_j)^{-1} \quad (i \neq j) \ .
$$
Any two such matrices, say $ Y $ and $ Y' $, may differ only in their
diagonal entries: let $(a_1,\ldots, a_n)$ and $(a_1', \ldots, a_n') $
be these diagonal entries. Then, by Lagrange's Interpolation,
there is $\, p(x) \in \c[x] $  such that $ p(\lambda_i) = a_i - a_i'  $ for
all $i$. The corresponding automorphism $\,(x,\,y+p(x)) \in G_x $ moves
$(X,Y)$ to $(X,Y') $.
\end{proof}

Now, for each $ k \ge 0 $, we introduce the following subgroups of $ G \,$:
\begin{eqnarray}
\la{abel}
G_{k,x}  & := & \{(x, y+ x^k p(x)) \in G \, : \ p(x) \in \, \c[x] \,\}\ ,\\
G_{k,y} & := &\{(x+y^k q(y), y) \in G \,: \ q(y) \in \, \c[y] \, \}\ .
\end{eqnarray}
Note that
$$
G_x = G_{0,x} \supset  G_{1,x} \supset G_{2, x} \supset \ldots \quad ,
\quad
G_y = G_{0,y} \supset  G_{1,y} \supset G_{2, y} \supset \ldots
$$
and in general,
for any $ k \ge 0 $, we have\footnote{However, unlike $ G $, the groups $ G_k $
are not generated by $ G_{k,x} $ and  $ G_{k,y} $ if $ k \ge 2 $. See
Section~\ref{n=2} below.}
\begin{equation}
\la{natin}
G_{k,x} = G_k \,\cap\, G_x \quad \mbox{and} \quad
G_{k,y} = G_k \,\cap\, G_y  \ .
\end{equation}

\begin{lemma}
\la{det}
Let $ (X,Y) \in \CC_n $, and let $ k \ge 0 $ be any integer.
\begin{enumerate}
\item[(1)] If $\det(X) \neq 0 $, then $ G_{k,x} (X,Y) = G_x (X,Y)$.

\item[(2)] If $\det(Y) \neq 0 $, then $ G_{k,y} (X,Y) = G_y (X,Y)$.
\end{enumerate}
\end{lemma}
\begin{proof}
We will only prove $(1)$.
Let $ \chi = \chi(X) $ denote the characteristic polynomial of $X$.
Since $ \det(X)\neq 0 $, we have $ \gcd(x^k , \chi) = 1 $ for all
$ k $. It follows that for each $ k \ge 0 $, there are $ f $ and $ g $
in $\c[x]$ such that $ f\,x^k + g\,\chi = 1 $. Hence,
any $ p \in \c[x] $ can be written in the form $ p =
p \, f \, x^k + p \, g \, \chi $. By the Cayley-Hamilton Theorem,
evaluating $ p $ at $ X $ then yields $ p(X) = p(X)\,f (X)\, X^k $, which
shows that $ G_{x} (X,Y) = G_{k,x} (X,Y) $ for any $ k $.
\end{proof}
In the rest of this section, we will use the following notation.
\begin{eqnarray*}
\CC^*_n & := & \{(X,Y) \in \CC_n \, | \, \det(X) \neq 0 \ \mbox{or} \ \det(Y) \neq 0\} \\
\CC^{*, \mathrm{reg}}_{n,1} & := & \{(X,Y) \in \CC_n \, | \, \det(X)\neq 0
\ \mbox{and}\ X\, \mbox{is diagonalizable} \}\\
\CC^{*,\mathrm{reg}}_{n,2} &:= & \{(X,Y) \in \CC_n \, | \, \det(Y)\neq 0
\ \mbox{and}\ Y \,\mbox{is diagonalizable}\}\\
\CC^{*, \mathrm{reg}}_n  & := &  \CC^{*, \mathrm{reg}}_{n,1} \ \cup \ \CC^{*, \mathrm{reg}}_{n,2}
\end{eqnarray*}
With this notation, Lemma~\ref{W} and Lemma~\ref{det} combined together imply
\begin{corollary}
\la{Wn}
Let $ k \ge 0 $ be any  integer.
\begin{enumerate}
\item[(1)]
If $(X,Y) \in \CC^{*, \mathrm{reg}}_{n,1} $,
then $ G_{k,x} $ acts transitively on
$\Upsilon^{-1}_1(\Spec\,X)$.

\item[(2)] If $(X,Y) \in \CC^{*, \mathrm{reg}}_{n,2} $,  then
$G_{k,y}$ acts transitively on
$\Upsilon^{-1}_2(\Spec\,Y)$.
\end{enumerate}
\end{corollary}

\noindent
Combining Corollary~\ref{Wn} with Theorem \ref{EG}, we get
\begin{corollary}
\la{trfib}
Let $ k \ge 0 $ be any  integer.
\begin{enumerate}
\item[(1)]
If $(X,Y) \in \CC^{*, \mathrm{reg}}_{n,1} $ then there is
$ \sigma \in G_{k,x}$ such that $\sigma (X,Y)\ = \ (X,Y_1)$, where
$Y_1$ is a diagonalizable matrix with eigenvalues   $ (1, \ldots, n)$.
\item[(2)] If $(X,Y) \in \CC^{*, \mathrm{reg}}_{n,2} $ then there is
$\sigma \in G_{k,y}$ such that $\sigma(X,Y)\ = \ (X_1,Y)$, where
$X_1$ is a diagonalizable matrix with eigenvalues   $(1,\ldots, n)$.
\end{enumerate}
\end{corollary}
\begin{proof}
Indeed, by Theorem \ref{EG}, the set $ \Upsilon^{-1}(1,\ldots, n;\, \Spec\,X) $
is nonempty. Since it is a subset of $ \Upsilon^{-1}_1(\Spec\,X) $,
statement $(1) $ follows from Corollary~\ref{Wn}$(1)$. Similarly,
statement $(2)$ is a consequence of Corollary~\ref{Wn}$(2)$.
\end{proof}

The next lemma is a slight modification of an important result due to T.~Shiota.
\begin{lemma}[{\it cf.} \cite{BW}, Lemma~10.3]
\la{Sh}
For any $(X,Y) \in \CC_n$, there exist polynomials $ q \in \c[x]$ and $ r \in \c[y]$
such that $Y+q(X)$ and $X+r(Y)$ are nonsingular diagonalizable matrices: that is,
$$
G_y (X,Y)\,  \cap \, \CC^{*, \mathrm{reg}}_{n,1} \, \neq \varnothing \  \mbox{and } \
G_x (X,Y) \cap \CC^{*, \mathrm{reg}}_{n,2}\, \neq \varnothing \, . $$
\end{lemma}

\vspace{1ex}

\begin{remark}
Shiota's Lemma (as stated in \cite[Lemma~10.3]{BW}) claims the existence of a
polynomial $ r \in \c[y]$ such that $ X + r(Y) $ a diagonalizable matrix.
Adding an appropriate constant to such a polynomial ensures that
$ \det(X+r(Y) + c\,I_n) \neq 0 $.
\end{remark}

\vspace{1ex}

\subsection{Proof of Theorem~\ref{ondouble1}}
\la{S3.4}
In view of Theorem~\ref{TBW} and Lemma~\ref{Ldt}$(3)$, it suffices to prove
that $ G_n $ acts transitively on $ \CC_n \setminus \{(X_0, Y_0)\} $, where
$ (X_0, Y_0) $ is the basepoint of $ \CC_n $ (see \eqref{base}). We will
establish the following more general fact:
\begin{equation}
\la{dcos}
|G_k \backslash \CC_n|  = \begin{cases} 1\ , &
\mbox{if}\ k \neq n \\2\ , & \mbox{if }\ k = n \end{cases}
\end{equation}
In the proof of \eqref{dcos} we may (and will) assume that $ k \ge n $ (indeed,
we have $ G_k \backslash \CC_n = G_k\backslash G/G_n = \CC_k/G_n $, and all the
above statements hold true for the right action of $G_n$ on $\CC_k$). Note also
that for $ k = n = 1 $, the claim \eqref{dcos} is obvious because $ \CC_1 = \c^2 $ and
$ G_1 $ contains $ \SL_2(\c) $ which acts on $ \c^{2} $ linearly as in
its natural (irreducible) representation.

We prove \eqref{dcos} in three steps. First, we show
that $\CC^*_n$ is part of a single orbit of $ G_k $ on $\CC_n$  for any $ k \ge 0$
(see Proposition~\ref{proponreg} below). Second, we show that if $ k > n $ then
$ G_k (X,Y) \ \cap \ \CC_n^{\ast} \neq \varnothing $ for any $ (X,Y) \in \CC_n $
(see Proposition~\ref{singact1}). Finally, for $ k = n $, we show that
$ G_n (X,Y) \ \cap \ \CC_n^{\ast} \neq \varnothing $ for any $ (X,Y) \not= (X_0, Y_0) $
(see Proposition~\ref{singact2}).
\begin{proposition}
\la{proponreg}
$ \CC^*_n $ is in a single orbit of $G_k$ for any $ k \ge 0$.
\end{proposition}
\begin{proof}
By \eqref{natin}, $ G_{k,x} $ and $ G_{k,y} $ are subgroups of $ G_k $ for any $ k $.
We will prove that $\CC^*_n$ lies in a single orbit of the group
generated by these subgroups. To this end, we first show that $\CC^{*,
\mathrm{reg}}_n$ lies in a single orbit of $\langle G_{k,x} , G_{k,y}
\rangle$ and then we prove that the $\langle G_{k,x} , G_{k,y} \rangle$-orbit
of any $(X,Y) \in \CC^*_n$ meets $\CC^{*, \mathrm{reg}}_n$.

Let $ (X_i, Y_i) \in \CC^{*,\mathrm{reg}}_{n,2} $ for $ i=1,2 $. We will show that
these two points can be connected by an element in $\langle G_{k,x} , G_{k,y} \rangle$.
By Corollary \ref{trfib}, there are $ \sigma_i \in G_{k,y} $ such that $\sigma_i(X_i,Y_i) \ =\ (\tilde X, \tilde Y_i)$, where $\tilde X \ =\ \mathtt{Diag}(1,\ldots, n)$ and $ \tilde Y_i$ is
the corresponding Calogero-Moser matrix similar to $Y_i$. Now, since $(\tilde X, \tilde Y_i) \in
 \CC^{*,\mathrm{reg}}_{n,1}$, we may again apply Corollary~\ref{trfib} to get $ \tau \in G_{k,x} $ such that $\tau (\tilde X, \tilde Y_1)\ =\ (\tilde X,\tilde Y_2)$. It follows that
$\sigma^{-1}_2 \, \tau \, \sigma_1 ( X_1,  Y_1)\ =\ ( X_2 , Y_2)$. A similar argument works
for any pair of points in $ \CC^{*,\mathrm{reg}}_{n,1} $.

Now, suppose $ (X_1 , Y_1) \in \CC^{*,\mathrm{reg}}_{n,1}$ and $(X_2
, Y_2) \in \CC^{*,\mathrm{reg}}_{n,2} $. Again, using
Corollary~\ref{trfib}, we may find $\tau \in G_{k,x} $ such that $\tau
(X_1, Y_1)\ =\ (X_1, \tilde Y)$, where $\tilde Y = \mathtt{Diag}(1,\ldots, n)$.
Since both $(X_1, \tilde Y)$ and $(X_2 , Y_2)$ are now in $\CC^{*,\mathrm{reg}}_{n,2} $,
we get back to the previous case.

Finally, let $(X,Y) \in \CC^*_n$ . By Lemma~\ref{det}, we know that
$ G_{k,y} (X,Y) = G_y (X,Y) $ for any $ k $. On the other hand,
by Lemma~\ref{Sh}, the $ G_y$-orbit of $ (X,Y) $ always meets
$ \CC^{*, \mathrm{reg}}_{n,1} $. This completes the proof of Proposition~\ref{proponreg}.
\end{proof}

To move an arbitrary point of $ \CC_n $ to $ \CC_n^* $ we will have to use
automorphisms of $ G_k $ which are composites of elements of $ G_{k,x} $ and $ G_{k,y} $.
To this end, for $ k \ge 2 $ and $ p \in \c[x]$, we define
\begin{equation}
\la{bverc}
\sigma_{k,p}:= \ (x+y^{k-1},y) \, \circ \,  (x, y - p(x))\, \circ \, (x-y^{k-1},y)\ ,
 \end{equation}
and
\begin{equation}
\la{speaut}
\tau_{k, p} \ := \ (x+ y^{k-2} + y^{k-1} ,y) \, \circ \,  (x, y - p(x))\, \circ \,
(x-  y^{k-2} - y^{k-1} ,y)\ .
\end{equation}
Now, let $ (X_0, Y_0) $ be the basepoint of $ \CC_k $ ($k\ge 2$) represented by
\eqref{base}. Write $ \alpha(x) \in \c[x] $ and $ \beta(x) \in \c[x] $ for the minimal
polynomials of the matrices $ X_0 - Y_0^{k-1} $ and $ X_0-Y_0^{k-2}-Y_0^{k-1} $ respectively.
A simple calculation shows
\begin{equation}
\la{minbase}
\alpha(x)\, = \ (-1)^k x^k - (k-1)!
\end{equation}
\begin{equation}
\la{minbase1}
\beta(x) \, = \ \ x^k +(-1)^k \frac{k!}{k-1} x + (-1)^{k-1} (k-1)!
\end{equation}

\begin{lemma}
\la{sitaugk}
For any $ c(x) \in \c[x] $, we have $\,\sigma_{k,\, \alpha(x)c(x)} \in G_k\,$ and $
\,\tau_{k,\, \beta(x) c(x)} \in G_k\,$.
\end{lemma}
\begin{proof}
Write $ X_1 := X_0 - Y_0^{k-1} $ and take any $ p \in \c[x] $ such that $ p(X_1)=0 $.
Then $\sigma_{k,\,p}$ acts on $ (X_0,Y_0) $ by
$$
(X_0, Y_0) \xrightarrow{(x + y^{k-1},y)} (X_1, Y_0)
\xrightarrow{(x, y - p(x))} (X_1,Y_0)
\xrightarrow{(x - y^{k-1},y)} (X_1 + Y_0^{k-1}, Y_0) \,=\, (X_0, Y_0) \,  .
$$
This shows that $\sigma_{k,p} \in G_k $ for $ p = \alpha(x) c(x) $. For $\,\tau_{k,\,p} $,
the calculation is similar.
\end{proof}

We are now ready to prove our next proposition.
\begin{proposition}
\la{singact1}
Let $(X,Y) \in \CC_n $. If $ k > n \,$, then $\,G_k \,
(X,Y) \cap \CC^*_n \neq \varnothing\,$.
\end{proposition}
\begin{proof}
Let $\tilde{G}_{k}:=\langle G_k, G_{k-1,y}\rangle $ be the subgroup of $ G $
generated by $ G_k $ and $ G_{k-1, y} $. First, we prove

\vspace{1ex}

\noindent
\textit{Claim 1}: $\, G_k (X,Y)= \tilde{G}_{k}(X,Y)$.

\vspace{1ex}

It is clear that $ G_k (X,Y) \subseteq \tilde{G}_{k}(X,Y)$, since
$ G_k \subseteq \tilde{G}_{k} $. To prove the opposite inclusion it suffices to
show that $ G_{k-1,y} (X_1,Y_1) \subseteq G_k (X,Y) $ for any $ (X_1,Y_1) \in G_k (X,Y) $.
This is equivalent to showing that, for any $ d(y) \in \c[y] $, the automorphism
$ (x + d(y)y^{k-1}, y) $ maps $ (X_1,Y_1) $ to a point in  $ G_k (X,Y) $.
Now, let $(X_1,Y_1) \in G_k (X,Y) $. Then, for the minimal
polynomial  $ p_1(y) := \mu(Y_1) $, we have $ \gcd(y^k, p_1(y))=y^l $ for some
$l\leq n <k $ (since $ \deg(p_1) \leq n $). Hence we can find
$ a(y), b(y) \in \c[y]$ such that $\, a(y)\, y^k + b(y)\, p_1(y) = y^{l} \,$.
Multiplying by $ d(y)\ y^{ k-l-1} $, we get
$$
d(y) \, a(y)\, y^{2k-l-1} + d(y)\, b(y)\, p_1(y)\, y^{k-l-1} \, =\,d(y) \,  y^{k-1}\ ,
$$
which, in turn, implies
$$
(x + d(y)y^{k-1},\, y) \, (X_1, Y_1) \, = \, (x +  d(y)  a(y) y^{2k-l-1}, \, y) \, (X_1, Y_1)\, .
$$
Now, since $\,2k-l-1\, \ge k\,$,  we see that $\,(x +  d(y)  a(y) y^{2k-l-1}, y) \in G_k\,$.
This finishes the proof of Claim~1.

\vspace{1ex}

\noindent
\textit{Claim 2}: $\,G_x \subset \tilde{G}_k\,$, or equivalently, $\,(x,y + d(x))
\in \tilde{G}_k $ for any $ d(x) \in \c[x]$.

\vspace{1ex}

By Lemma~\ref{sitaugk}, $\, \sigma_{k,\, p} \in G_k  \subset \tilde{G}_{k} \,$
for $ p = \alpha(x)\,c(x) $, where $ \sigma_{k,\, p} $ is the automorphism of $ G $
defined in \eqref{bverc}, $ c(x) \in \c[x] $ is any polynomial
and $ \alpha(x) $ is given by \eqref{minbase}. Since $(x \pm y^{k-1},y)
\in G_{k-1,y} \subset \tilde{G}_k$, we have $(x, y+ \alpha(x)  c(x) ) \in \tilde{G}_k$.
Now, as $ \gcd(\alpha(x),x^k) = 1 \,$, for any $d(x)\in \c[x]$,  we can find
$ a(x), b(x) \in \c[x] $ such that $\, a(x)\, \alpha(x) +  b(x)\, x^k \, = \,d(x)\,$.
This shows that $ (x, y + d(x)) $ is the composition of the  automorphisms
$\,(x, y +  a(x) \alpha(x)) \,$ and $ \,(x, y+  b(x) x^k)\,$, both of which are
in $ \tilde{G}_k $. This proves Claim~2.

Now, combining Claim~1 and Claim~2, we see that $\, G_x (X,Y)  \subseteq  G_k (X,Y) \,$.
On the other hand, $\, G_x (X,Y)  \cap \CC^*_n \neq \varnothing\,$ by Lemma~\ref{Sh}.
Proposition~\ref{singact1} follows.
\end{proof}

With Proposition~\ref{proponreg} and Proposition~\ref{singact1}, the proof of \eqref{dcos} for
$k \neq n$ is complete. To prove \eqref{dcos} for $ k=n $ we will look at the action of $ G_n $ in the complement of $ \CC^*_n $ in $\CC_n$.
Writing $ \mu(X) $ for the minimal polynomial of a matrix $ X $, we define
\begin{eqnarray*}
\CC^0 _n & := & \CC_n \backslash \CC^*_n = \{ (X,Y) \in \CC_n \ | \ \det(X) = \det(Y)=0\}\ , \\
\CC^0 _{n,1} & := &  \{(X,Y) \in \CC^0 _n \ | \ \mu(X) = x^n \, \mbox{ and } \, \mu(Y)=y^n \} \ , \\
 \CC^0 _{n,2} & := &  \{(X,Y) \in \CC^0 _n \ | \ \mu(X) \neq x^n \, \mbox{ and } \, \mu(Y)=y^n \} \ ,\\
\CC^0 _{n,3} & := &  \{(X,Y) \in \CC^0 _n \ | \ \mu(Y) \neq y^n \}\ ,
\end{eqnarray*}
so that
$$
\CC^0 _n = \CC^0 _{n,1} \, \bigsqcup \, \CC^0 _{n,2}\,\bigsqcup \,\CC^0 _{n,3} \ .
$$
Since $\,\CC_{1}^{0} = (0,0) \,$, in the proof of the next proposition
we will assume $ n\ge 2 \,$.
\begin{proposition}
\la{singact2}
Let  $\,(X,Y) \in \CC^0_n \setminus \{(X_0,Y_0)\}\,$. Then $\,G_n \,
(X,Y) \cap \CC^*_n \neq \varnothing\,$ .
\end{proposition}
\begin{proof} We will consider three cases corresponding to the above decomposition.
In case~\Rmnum{1} and case~\Rmnum{2}, we will explicitly produce $\tau \in G_k$ such that
$ \tau(X,Y) \in \CC^*_n$. Then, the last case will be proved by contradiction
assuming the first two cases.

\vspace{1ex}

\noindent
{\it Case\,\Rmnum{1}}\,.\  Let $(X,Y) \in \CC^0 _{n,1}$. By \cite[Proposition 6.8]{W},
$ \CC^0 _{n,1} $ consists of exactly $ n $ points
$ \{(X(n,i) , Y_0)\,:\,  i = 1, \ldots , n\} $, with $i=1$ corresponding to
the base point $ (X_0,Y_0) $. Our goal is to show that there exist
$\, \phi_{i} \in G_n \,$ such that $ \phi_i (X(n,i) , Y_0) \in \CC^{\ast}_n$
for $ i= 2,\ldots, n\,$.

Let $\, X_i' : = X(n,i) - Y^{n-2}_0 -Y_0^{n-1}$, and let $ q_{i}(x) $ be the minimal
polynomials of $ X_i' $ for $ i= 1,\ldots, n\,$. Note that
$ q_1(x) = \beta(x) $ for $ k = n $ (see \eqref{minbase1}), since
$ (X(n,1), Y_0) = (X_0, Y_0) $.  It is easy to compute the polynomials
$ q_{i}(x) $ explicitly:
\begin{equation}
\la{gsdmin}
q_i(x) \, = \, \begin{cases} x^n +  (-1)^{n} \frac{n!}{n-1 }x +(-1)^{n+1} (n-1)!\ , &
\ i=1 \\*[1.5ex]
x^n+ (-1)^{n-i} (i-1)! (n-i)!\ , & \ i=2, \ldots, n-1 \, \\*[1.5ex]
 x^n+   \frac{n!}{n-1} x  + (n-1)!\ , & \ i=n
\end{cases}
\end{equation}
and verify that $ \mathtt{gcd}(q_i(x),\beta(x))=1 $ for all $ i=2,\ldots, n $.
Hence, there are polynomials $ a_{i}(x),\, b_i(x) \in \c[x] $ such that
\begin{equation}
\la{idengcd}
a_i(x) q_i(x) + b_{i}(x) \beta(x)  = 1\ .
\end{equation}
Furthermore, by Lemma~\ref{Sh}, for $ (X_i',Y_0) \in \CC_n $, there are
$ s_i(x)\in \c[x] $ such that
\begin{equation}
\la{detni}
\det(Y_0-s_i(X_i')) \neq 0
\end{equation}
Using these polynomials, we set $\, c_i(x) := s_i(x) \, b_i(x) \, \beta(x) \,$ for each
$ i=2,\ldots ,n $ and consider the corresponding automorphisms
$\,\tau_{n,\,c_i(x)}\,$ defined as in \eqref{speaut}.
By Lemma~\ref{sitaugk},  $\,\tau_{n,\,c_i(x)} \in G_n\,$ since  $ \beta(x) $ divides $ c_i(x) $.
On the other hand, by \eqref{idengcd}, $\,c_i(x) = s_i(x) - s_i(x)\, b_i(x)\,q_i(x)\,$,
which implies $ c_i(X_i') = s_i (X_i') $. Thus, if we apply $\tau_{n,\,c_i} $,
which is an element of  $ G_n$,  to the pair $ (X(n,i),Y_0) $, we get
 $$ \tau_{n,c_i} (X(n,i), Y_0)\,=\,(X_i'+(Y_0-s_i(X_i'))^{n-2}+ (Y_0-s_i(X_i'))^{n-1}, Y_0 -s_i(X_i'))  \,  .$$
Now, using \eqref{detni}, we conclude $\, \sigma_{n,c_i} (X(n,i), Y_0) \in \CC_n^{\ast}\,$.

\vspace{1ex}
\noindent
{\it Case \,\Rmnum{2}}\,.\ Let $ (X,Y) \in \CC^0 _{n,2}$. Then, by \cite[Proposition 6.11]{W}, we may assume
that $ Y = Y_0 $ and
\begin{equation*}
\la{formxys}
 X\, = \, X(n,r) + \sum^{n-1}_{k=1} X_{(k)}  \, ,
\end{equation*}
where $ 2 \leq r \leq n$ and $ X_{(k)}$ is a matrix with only nonzero entries on the
$k$-th diagonal. Let $ \lambda \in \c^* $. Applying a transformation $\, Q_{\lambda} \in G_n \,$, which is a
scaling transformation followed by conjugation by $ {\tt Diag}(1, \lambda, \ldots, \lambda^{n-1})$, see \cite[Eq. (6.5)]{W}, we get
\begin{equation*}
\label{eq22}
Q_{\lambda} (X,Y_0)\ = \ (X_\lambda , Y_0)\ := \
(X(n,r)+\sum^{n-1}_{k=1} \lambda^{k-1}X_{(k)},\ Y_0)\,.
\end{equation*}
If $\det(X_\lambda) \neq 0$ then $(X_\lambda , Y_0) \in \CC_{n}^*$, and we are done. If $\det(X_\lambda) =0$ then either $\mu(X_\lambda)=x^n $ or $\mu(X_\lambda)\neq x^n$.
If $\mu(X_\lambda)=x^n $, then $\,(X_\lambda , Y_0) \in  \CC^0 _{n,1} \,$, which brings us back to
case~\Rmnum{1} established above. If $\mu(X_\lambda) \not= x^n $, then  $ q_{\lambda}(x) := \mu(X_\lambda) $ satisfies $\,\gcd(q_{\lambda}(x),x^n)=x^{l}\,$ for some $ l<n $, and we can find
$ a(x), b(x) \in \c[x]$ such that
\begin{equation*}
\la{eqxl}
b(x)\,x^{2n-l-1} =  x^{n-1} - a(x)\, q_{\lambda}(x)\, x^{n-1-l}\, .
\end{equation*}
If we apply $ (x, y  + b(x)x^{2n-l-1}) \in G_{n,x} $ to $(X_\lambda, Y_0)$, we get
$ (X_\lambda,Y_0 + X_{\lambda}^{n-1})\,$. Now,
$$
\det(Y_0 + X_{\lambda}^{n-1}) =  \det(Y_0 + X(n,r)^{n-1})+ \lambda\,O(\lambda) =
(-1)^{r-1} (r-1)! (n-r)! + \lambda \, O(\lambda)\, .
$$
Thus, if we choose $\lambda $ small enough, this last determinant is nonzero,
and $\,\tau :=  Q_\lambda \circ (x, y + b(x)x^{2n-l-1}) \in G_n \,$
moves $ (X,Y) \in \CC^0 _{n,2} $ to $ (X_\lambda,Y_0 + X_{\lambda}^{n-1}) \in \CC_{n}^* $.

\vspace{1ex}

\noindent
{\it Case\,\Rmnum{3}}\,.\  Let $ (X,Y) \in \CC^0 _{n,3} $. It suffices to prove that $\,
G_n(X,Y) \not\subseteq \CC^0 _{n,3} $. Assume the contrary. Then,
for any $ (X_1,Y_1) \in G_n(X,Y) $, we have $ \gcd(\mu(Y_1),y^{n}) = y^l $,
where $\mu(Y_1)$ is the minimal polynomial of $Y_1$ and $l<n$. Arguing as in
Proposition~\ref{singact1}, we can show that $ G_n (X,Y)= \tilde{G}_{n}(X,Y)$, where
$\tilde{G}_{n} := \langle G_n, G_{n-1,y} \rangle $, and hence $ G_x(X,Y) \subseteq
G_n(X,Y) $. Then, by  Lemma~\ref{Sh}, there is $ r(x) \in \c[x]$
such that $(X, Y+r(X))\in \CC_{n}^{*, \mathrm{reg}} $. This
contradicts our assumption that $\, G_n(X,Y)\subseteq \CC^0 _{n,3}\, $.
\end{proof}

\subsection{Proof of Theorem~\ref{ondouble2}}
\la{S3.5}
We will assume that $\, n_1 < n_2 < \ldots < n_m \,$ and argue inductively in $ m $.
For $ m = 1 \,$, Theorem~\ref{ondouble2} is precisely Theorem~\ref{TBW}. For $ m = 2 \,$,
Theorem~\ref{ondouble2} follows from \eqref{dcos}: indeed, given any point
$\, (P_1, P_2) \in \CC_{n_1} \times \CC_{n_2}\,$, we first use the transitivity of
$ G $ on $ \CC_{n_1} $ to move $ (P_1, P_2) $ to $ (P_0(n_1), P_2') $,
where $ P_0(n_1) = (X_0(n_1), Y_0(n_1)) $ is the basepoint of $ \CC_{n_1} $, and then
use the transitivity of $ G_{n_1} = \Stab_{G}[P_0(n_1)] $ on $ \CC_{n_2} $ to move
$ (P_0(n_1), P_2') $ to $ (P_0(n_1), P_0(n_2)) $.
Now, to extend this argument to any $ m $ we need the following proposition.

Let $ I = \{k_1 < k_2 < \ldots < k_r\} $ be a collection of positive integers written in
increasing order. Let $ (X_0 (k), Y_0 (k)) \in \CC_k $ be the basepoints \eqref{base} of
the spaces $ \CC_k $ corresponding to $k \in I$, and let $\, G_{I} := \cap_{k \in I} \,G_k $
denote the intersection of the stabilizers of these basepoints in $ G $.
\begin{proposition}
\la{singact3}
If $ k > n $ for all $k \in I$, then $G_I$ acts transitively on $\CC_n$.
\end{proposition}
Before proving Proposition~\ref{singact3}, we make one elementary observation.
\begin{lemma}
\la{mult-det}
There is a polynomial $\, p(y) =
\sum_{j \in I} a_{j-1} y^{j-1}\, \in \,\c[y] \, $ such that
$$
\det[\tilde X_0 (k)] \neq 0  \quad \mbox{for all}\quad k \in I\ ,
$$
where $\,\tilde X_0 (k) :=  X_0 (k)+ p(Y_0(k))\,$.
\end{lemma}
\begin{proof}
For a fixed $ k \in I $ and a generic polynomial $ \,p(y) =
\sum^{k-1}_{j=1} a_j y^j\,$, we have
$$
\det(\tilde X_0(k)) = (-1)^{k-1} (k-1)! \, a_{k-1} + f(a_1, a_2, \ldots, a_{k-2})\ ,
$$
where $ f(a_1, a_2, \ldots, a_{k-2}) $ is some polynomial in $ a_1, \ldots , a_{k-2} $.
Applying this to each $ k \in I $ and taking into
account the fact that $ Y_0 (k) $ are nilpotent matrices of order $ k $, we get
$r$ polynomials in $ r $ variables of the form
$$
\det(\tilde X_0 (k_s)) = a_{k_s} + f_{s-1}(a_{k_1}, \ldots , a_{k_{s-1}})\ ,\quad
s = 1, 2, \ldots, r \ .
$$
Such polynomials define an invertible transformation of the affine space $\c^r $,
so we can certainly choose $ a_{k_1}, \ldots , a_{k_r} \in \c$ in such a way
that all $ \det(\tilde X_0 (k_s)) $ take non-zero values.
\end{proof}
As a consequence of Lemma \ref{mult-det}, we get
\begin{corollary}
\la{non-zero}
Let $ p(y)$ and $ \tilde X_0 (k) $ be as in Lemma~\ref{mult-det}, and let
$ \mu_k(x) \in \c[x] $ denote the minimal polynomial of $ \tilde X_0 (k) $.
\begin{enumerate}
\item[(1)] The polynomial $\, \mu(x) := \prod_{k \in I} \mu_k(x) \,$ has a
nonzero constant term.

\item[(2)] The automorphism
$\,(x-p(y),y) \circ (x,y - \mu(x) c(x)) \circ (x+p(y),y) \,$
is in $ G_I $ for any  $ c(x) \in \c[x]$.
\end{enumerate}
\end{corollary}
Now, we can proceed with

\begin{proof}[Proof of Proposition~\ref{singact3}]
Note that $ G_I $ contains two abelian subgroups $ G_{N,x} := \cap_{k \in I} \, G_{k,x} $
and $ G_{N,y} := \cap_{k \in I} \, G_{k,y}$, where $ N = {\rm max} \{k_1, \ldots, k_r\} $.
We will argue as in Proposition~\ref{singact1}. Define $ \tilde{G}_{I}:=\langle G_I, G_{n,y}\rangle $. We will show $\, G_I (X,Y) = \tilde{G}_{I} (X,Y) = G (X,Y) \,$ for any $ (X,Y) \in \CC _{n} $, and the proposition will follow from Theorem~\ref{TBW}.

\vspace{1ex}

\noindent
\textit{Claim 1}: $\ G_I (X,Y)= \tilde{G}_{I}(X,Y)\,$.

\vspace{1ex}
First, $\, G_I (X,Y) \subseteq \tilde{G}_{I}(X,Y)$, since $G_I$ is a subgroup of $\tilde{G}_{I}$.
To prove the opposite inclusion it suffices to show that for any $ d(y) \in \c[y] $, the automorphism $ (x + d(y)y^n, y)$ preserves $ G_I (X,Y) $. Let $ (X_1,Y_1) \in G_I (X,Y)$, and
let $ p_1(y):=\mu(Y_1) $ be the minimal polynomial of $ Y_1 $. Then, since $\deg(p_1) \leq n \,$,
we have $ \gcd(y^N, p_1(y)) = y^l $ for some $ l \leq n < N $. Hence,
we can find $ a(y), b(y) \in \c[y]$ such that $\, a(y) y^{N} + b(y) p_1(y)  = y^{n}\,$.
It follows that for any $ d(y) \in \c[y] $,
$$ (x + d(y)y^{n}, y) \, (X_1, Y_1) \, = \, (x +  d(y)  a(y) y^{N}, y) \, (X_1, Y_1)\, .$$
Thus $(x + d(y)y^{n}, y)$ maps $(X_1,Y_1)$ into $G_{N,y} \, (X_1 ,Y_1) \subseteq G_I \, (X,Y)$.

\vspace{1ex}

\noindent
\textit{Claim 2}: $\tilde G_I = G $.

\vspace{1ex}

It suffices to show that $ G_x \subset \tilde{G}_I $, or equivalently,
$ (x, y+d(x)) \in \tilde{G}_I $ for any $d(x) \in \c[x]$. Indeed,
if $ G_x \subset \tilde{G}_I $, conjugating $G_{N,y}$ by $(x,y+c)$ with an appropriate $c$
we can show that $G_{y}$ is also contained in $\tilde G_{I}$ and hence $ G \subseteq G_{I} $
(since $ G_x $ and $ G_y $ generate $ G $).
Now, let $ p(y) \in \c[y] $ be a polynomial defined in Lemma~\ref{mult-det}.
Then, for any $ c(x) \in \c[x] $, the composition of automorphisms
$\,(x-p(y),y) \circ (x,y + \mu(x)c(x)) \circ (x+p(y),y) \,$
is in $\tilde G_{I}$ (see Corollary~\ref{non-zero}$(2)$).
Since $(x \pm p(y),y)  \in G_{n,y} \subset \tilde{G}_I$, we
have $(x, y+\mu(x)  c(x) ) \in \tilde{G}_I$ for any $c(x) \in \c[x]$.
By Corollary~\ref{non-zero}$(1)$, $\,\gcd(\mu(x), x^N)=1\,$, so for any
$d(x)\in \c[x]$, we can find
$a(x), b(x) \in \c[x]$ such that $\, a(x)\,\mu(x) +  b(x)\, x^N = d(x) \,$.
This shows that $(x, y+d(x)) $ is the composition of
$ (x, y+  a(x) \mu(x)) $ and $ (x, y+  b(x) x^N)$, both of which are in $ \tilde{G}_I$.
It follows that $ G_x \subseteq \tilde{G}_I $. This completes the proof of Claim~2,
Proposition~\ref{singact3} and Theorem~\ref{ondouble2}.
\end{proof}

\subsection{Corollary}
\la{maincor}
Theorem~\ref{ondouble1} and Theorem~\ref{ondouble2} have interesting implications.
Recall that $ G_k \subseteq G $ denotes the stabilizer of a point of $ \CC_k $
under the action of $ G $, see \eqref{gn}.
\begin{corollary}
\la{cor222}
Let $\,k, n \,$ be non-negative integers.
\begin{enumerate}
\item[(1)]  If $ k \ne n $, then $\, G_k $ acts transitively on $ \CC_n $.
In this case, $ G = G_k G_n $.

\item[(2)] For each $ k \ge 0 $,  $\, G_k $ is a maximal subgroup of $G$.

\item[(3)]  The normalizer of $ G_k $ in $ G $ is equal to $ G_k $. Moreover,
if $\,k \neq n \,$, there is no $ g \in G $ such that $ g^{-1} \, G_k \, g
\subseteq G_n $.
\end{enumerate}
\end{corollary}
\bproof
$(1)$ is an easy consequence of Theorem~\ref{ondouble2}. $(2)$ follows from
Theorem~\ref{ondouble1} and Corollary~\ref{Codt}$(a)$. $(3)$ is immediate from
$(2)$.
\eproof

If we reformulate  Corollary~\ref{cor222} in terms of automorphism groups of algebras Morita equivalent
to $ A_1(\c) $ (see Theorem~\ref{T1}), part $(3)$ answers a question of Stafford \cite{St}.
As mentioned in the Introduction, this result was established by different methods in \cite{KT} and \cite{W2}. By Theorem~\ref{TBW}, it is equivalent to the fact that the identity
map is the {\it only} $G$-equivariant map from
$ \CC_k $ to $ \CC_k $, and there are no $G$-equivariant maps $ \CC_k \to \CC_n $
for $ k \ne n $. In this form, Corollary~\ref{cor222}$(3)$ was proven in \cite{W2}.
Theorem~\ref{ondouble1} can thus be viewed as a strengthening of the main theorem
of \cite{W2}.

\subsection{Infinite transitivity}
\la{Conj}
We conclude this section with two conjectures related to
Theorems~\ref{ondouble1} and~\ref{ondouble2}. For a space $X$ and an integer
$ k > 0 $, we denote by $ X^{[k]} $ the configuration space of ordered $k$ points
of $X$, i.e.
$$
X^{[k]} := \{(x_1, x_2, \ldots, x_k) \in X^k\,:\, x_i \ne x_j \}\ .
$$
An action of a group $ G $ on $ X $ is then called {\it $k$-transitive}
if the induced action $\,G \times X^{[k]} \to X^{[k]} \,$ is transitive.
(Clearly, for $k=2$, this definition agrees with the one given in Section~\ref{S3.1}.)
A $G$-action which is $k$-transitive for all $ k $, is called {\it infinitely transitive}.
Since the natural projection $ X^{[k]} \onto X^{[k-1]} $ is $G$-equivariant, the $k$-transitivity implies the $(k-1)$-transitivity: in particular, any $k$-transitive action is transitive.

\vspace{1ex}

\begin{conjecture}
\la{Conj1}
{\it For each $ n \ge 1 $, the action of $G$ on $ \CC_n $ is infinitely transitive.}
\end{conjecture}

\vspace{1ex}

\noindent
{\bf Remarks.}
1.\, Conjecture~\ref{Conj1} is true for $ n = 1 \,$: this follows from
a well-known (and elementary) fact that  $ \Aut(\c^d) $ acts
infinitely transitively on $ \c^d $ for all $ d \ge 1 $.

2.\, The infinite transitivity is an infinite-dimensional phenomenon:
a finite-dimensional Lie group or algebraic group cannot act infinitely
transitively on a variety. Indeed, if (say) an algebraic group $H$
acts $k$-transitively on a variety $ X $, there is a dominant map
$\, H \to X^k \,$, $\,g \mapsto (g.x_1, \ldots, g.x_k)\,$ defined by
a $k$-tuple $ (x_1, \ldots, x_k) $ of pairwise distinct points of $X$;
whence $ \dim(H) \geq k \cdot \dim(X) $ ({\it cf.} \cite{Bo2}).
In Section~\ref{S6}, we will equip $ G $ with the structure of an
infinite-dimensional algebraic group, which
is compatible with the action of $ G $ on $ \CC_n $.
Here, we note that no algebraic subgroup of $ G $ of finite dimension may act
$k$-transitively on $ \CC_n $ if $ k > 2 $. Indeed, by Theorem~\ref{algsub},
any finite-dimensional algebraic subgroup of $ G $ is
conjugate to either a subgroup of $A$ or a subgroup of $B$. In the first case,
assuming that $H$ acts $k$-transitively on $ \CC_n $, we have
$\, 2n k \leq \dim(H) \leq 5 $, hence the $H$-action on $ \CC_n $ cannot be even
transitive if $ n > 2 $ and it is at most $2$-transitive in any case.
In the second case, if $H$ is conjugate to
a subgroup of $ B:\, (x,y) \mapsto (a^{-1}x + q(y),\, a y + b) $, the action may
be at most doubly transitive, since so is the action
of affine transformations $\, y \mapsto ay+b \,$ on $ \c^1 $.

3.\, In connection with Conjecture~\ref{Conj1} we mention an interesting recent
paper \cite{AFKKZ}. For an affine variety $X$, the authors of \cite{AFKKZ}
introduce and study the group $ {\rm SAut}(X) $ of {\it special} automorphisms of $X$. By definition, $ {\rm SAut}(X) $ is generated by all
one-parameter unipotent subgroups of $ \Aut(X) $ (i.e. by the images  in $ \Aut(X) $
of the additive group $ {\mathbb G}_a = (\c,\, +) $ coming from the regular actions
$ {\mathbb G}_a \times X \to X $). The main theorem ({\it cf.} \cite[Theorem~0.1]{AFKKZ})
implies that if $ X $ is smooth, the natural action of $ {\rm SAut}(X) $ on $X$ is transitive
if and only if it is infinitely transitive.
In view of Theorem~\ref{TBW}, this result applies to our Calogero-Moser spaces
$ \CC_n $, since the image of $ G $ in $ \Aut(\CC_n) $ under the action \eqref{gact}
lies in $ {\rm SAut}(\CC_n) $. To prove Conjecture~\ref{Conj1} it would therefore suffice to show that $ G $ generates all of $ {\rm SAut}(\CC_n) $, which is certainly true
for $ n = 1 $ (since $ \CC_1 \cong \c^2 $) but seems unlikely for $ n \ge 2 $. Thus,
Conjecture~\ref{Conj1} may be viewed as a strengthening of the general results of \cite{AFKKZ}
in the special case $ X = \CC_n $.

\vspace{1ex}

The notion of infinite transitivity can be generalized in the following
way ({\it cf.} \cite{AFKKZ}, Sect.~3.1). Let $ X = \bigsqcup_{n} X_n \,$ be a disjoint
union of $G$-sets (e.g., the orbits of an action of $G$ on a space $X$). Then,
for each integer $ k > 0 $, we can stratify
$$
X^{[k]} = \bigsqcup_{k_1 + \ldots + k_m = k}\ \bigsqcup_{n_1 < \ldots < n_m}
X_{n_1}^{[k_1]}  \times \ldots \times X_{n_m}^{[k_m]}\ .
$$
Now, we say that $ G $ acts {\it collectively infinitely transitively} on $X$
if $ G $ acts transitively on each stratum  $\, X_{n_1}^{[k_1]} \times \ldots \times
X_{n_m}^{[k_m]}\,$ of $ X_{n}^{[k]} $ for all $ k > 0 $.
Intuitively, a collective infinite transitivity means the possibility
to move simultaneously (i.e., by the same automorphism) an arbitrary finite
collection of points from different orbits into a given position.

Theorem~\ref{ondouble1}, Theorem~\ref{ondouble2} and Conjecture~\ref{Conj1}
are subsumed by the following general
\begin{conjecture}
\la{Conj2}
{\it The action of $G$ on $ \CC =  \bigsqcup_{n \ge 0}\,\CC_n $ is
collectively infinitely transitive.}
\end{conjecture}
Note that for $ k = 1 $, Conjecture~\ref{Conj2} implies Theorem~\ref{TBW};
for $ k_1 = k $, it implies Conjecture~\ref{Conj1}, and for $m=k$ and
$ k_1 = k_2 = \ldots = k_m = 1 $, it implies Theorem~\ref{ondouble2}.

\section{The Structure of $G_n$ as a Discrete Group}
\la{S4}
In this section, we will use the Bass-Serre theory of graphs of groups
to give an explicit presentation of $ G_n $. We associate to
each space $ \CC_n $ a graph $ \Gamma_n $ consisting of orbits of
certain subgroups of $ G$ and identify $ G_n $ with the fundamental
group $\, \pi_1({\mathbf \Gamma}_n, \ast)\, $ of a graph of groups
$ {\mathbf \Gamma}_n $ defined by the stabilizers of those orbits in
$ \Gamma_n $. The Bass-Serre theory will then provide an explicit formula
for $\, \pi_1({\mathbf \Gamma}_n, \ast)\, $ in terms of generalized
amalgamated products. The results of this section were announced in
\cite{BEE}.

\subsection{Graphs of groups}
\la{S4.1}
To state our results in precise terms we recall the notion of
a graph of groups and its fundamental group (see \cite[Chapter~I, \S 5]{Se}).

A {\it graph of groups} $\, \bG = (\Gamma, \,G)\,$ consists of the following data:

$\,(1)\,$ a connected graph $ \Gamma $ with vertex set $ V = V(\Gamma) $, edge set
$ E = E(\Gamma) $ and incidence maps $\,i,\, t:\, E \to V\, $,

$\,(2)\,$ a group $ G_a $ assigned to each vertex
$ a \in V \,$,

$\,(3)\,$ a group $ G_e $
assigned to each edge $ e \in E \,$,

$\,(4)\,$ injective group homomorphisms $ G_{i(e)} \stackrel{\alpha_e}{\hookleftarrow}
G_e \stackrel{\beta_{e}}{\hookrightarrow} G_{t(e)} $ defined for each $ e \in E $.

\vspace{1ex}

\noindent
Associated to $ \bG $ is the path group $\,\pi(\bG)\,$, which is given by the presentation
\begin{equation*}
\la{pgr}
\pi(\bG) := \frac{(\ast_{a \in V} G_a) \ast \langle E\rangle}
{(e^{-1} \alpha_e(g)\, e = \beta_{e}(g)\ :\
\forall\,e \in E,\, \forall\, g \in G_e)}\ ,
\end{equation*}
where `$\, \ast \,$'  stands for the free product (i.e. coproduct
in the category of groups) and $\, \langle E\rangle \,$ for the
free group with basis set $ E = E(\Gamma) $.
Now, if we fix a maximal tree $ T $ in $ \Gamma $, the {\it fundamental group $ \pi_1(\bG, T) $ of
$\, \bG $ relative to} $ T $
is defined as a quotient of $ \pi(\bG) $ by `contracting the edges of $T$ to a point': precisely,
\begin{equation}
\la{fgr}
\pi_1(\bG, T) := \pi(\bG)/(e = 1\ :\ \forall\,e \in E(T))\ .
\end{equation}
For different maximal trees $ T \subseteq \Gamma $,
the groups $ \pi_1(\bG,\, T) $ are isomorphic. Moreover,
if $ \bG $ is trivial (i.~e. $\,G_a = \{1\}\,$ for all
$\,a \in V$), then $\, \pi_1(\bG,\, T) $ is isomorphic to
the usual fundamental group $ \pi_1(\Gamma,\,a_0) $ of the graph
$\,\Gamma\,$ viewed as a CW-complex. In general,
$\, \pi_1(\bG,\, T)\,$ can be also defined in a topological fashion
by introducing an appropriate notion of path and homotopy equivalence
of paths in $ \bG $ ({\it cf.} \cite{B1}, Sect.~1.6).

When the underlying graph of $ \bG $ is a tree (i.e., $\Gamma = T$),
$\, \bG $ can be viewed as a directed system of groups indexed by $ T $.
In this case, formula \eqref{fgr} shows that
$ \pi_1(\bG, T) $ is just the inductive limit
$\,\dlim\,\bG\,$, which is called the {\it tree product} of
groups $\,\{G_a\}\,$ amalgamated by $ \{G_e\} $ along $ T $.
For example, if $ T $ is a segment with
$\,V(T) = \{0,\,1\}\,$ and $\, E(T) = \{e\}\,$, the tree product is
the usual amalgamated free product $\,G \ast_{\,G_e} G_1 \,$. In
general, abusing notation, we will denote the tree product
by
$$
G_{a_1}\ast_{\,G_{e_1}}  G_{a_2} \ast_{\,G_{e_2}}
G_{a_3} \ast_{\,G_{e_3}}\, \ldots
$$

\subsection{$G_n$ as a fundamental group}
\la{S4.2}
To define the graph $ \Gamma_n $
we take the subgroups $A$, $B$ and $U$ of $G$
given by the transformations \eqref{A}, \eqref{B} and
\eqref{U}, respectively.
Restricting the action of $ G $ on $ \CC_n $
to these subgroups, we let $ \Gamma_n $ be the
oriented bipartite graph, with vertex and edge sets
\begin{equation}
\label{gamman}
V(\Gamma_{n}) := (\mathcal{C}_n/A)\,\bigsqcup\,(\mathcal{C}_n/B) \ ,\quad
E(\Gamma_{n}) := \mathcal{C}_n/U \ ,
\end{equation}
and the incidence maps $\,E(\Gamma_{n}) \to V(\Gamma_{n})\,$ given by
the canonical projections $\, i: \mathcal{C}_n/U \to \CC_n/A \,$
and $\, t: \mathcal{C}_n/U \to \mathcal{C}_n/B  \,$. Since the elements of
$A$ and $B$ generate $ G $ and $ G $ acts transitively
on each $ \CC_n $, the graph $ \Gamma_n $ is connected.

Now, on each orbit in $ \mathcal{C}_n/A $ and
$\,\mathcal{C}_n/B\,$
we choose a basepoint and elements $\, \sigma_A \in G \,$
and $\,\sigma_B \in G\, $ moving these basepoints
to the basepoint $ (X_0,\,Y_0) $ of $ {\mathcal C}_n $.
Next, on each $U$-orbit $\,\O_U \in \mathcal{C}_n/U \,$ we also choose a
basepoint and an element $\,\sigma_U \in G \,$ moving this
basepoint to $ (X_0,\,Y_0) $ and such that $\,	\sigma_U \in
\sigma_A A \,\cap \,\sigma_B B\,$,
where $ \sigma_A $ and $ \sigma_B $ correspond to the
(unique) $A$- and $B$-orbits containing $ \O_U $.
Then, we assign to the vertices and edges of $ \Gamma_n $ the
stabilizers $\,A_\sigma = G_n \cap \sigma A \sigma^{-1} \,$,
$\,B_\sigma = G_n \cap \sigma B \sigma^{-1} \,$, $\,U_\sigma
= G_n \cap \sigma U \sigma^{-1} \,$  of the corresponding
elements $ \sigma $ in the graph of right cosets of $ G $
under the action of $ G_n $.
These data together with natural group homomorphisms $\,
\alpha_\sigma: U_\sigma \hookrightarrow A_\sigma  \,$
and $\, \beta_\sigma:\,U_\sigma \hookrightarrow B_\sigma \,$
define a graph of groups $ {\mathbf \Gamma}_n $ over
$ \Gamma_n $, and its fundamental group
$\, \pi_1({\mathbf \Gamma}_n,\,T)\,$
relative to a maximal tree $\,T \subseteq \Gamma_n \,$
has canonical presentation, cf. \eqref{fgr}:
\begin{equation}
\label{fgr2}
\pi_1({\mathbf \Gamma}_n,\,T) = \frac{(A_\sigma \ast_{\,U_\sigma}
B_\sigma \ast \, \ldots\,) \ast \langle\,
E(\Gamma_n \!\!\setminus T)\,\rangle}{(\,e^{-1} \alpha_\sigma(g)\, e
= \beta_{\sigma}(g)\, :\, \forall\,e \in
E(\Gamma_n \!\!\setminus T),\, \forall\, g \in U_\sigma\,)}\ .
\end{equation}
In \eqref{fgr2}, the amalgam
$\, (A_\sigma \ast_{\,U_\sigma} B_\sigma \ast \, \ldots) \,$ stands for
the tree product taken along the edges of $ T $,
while $\, \langle\, E(\Gamma_n\!\!\setminus T)\,\rangle\,$ denotes
the free group generated by the set of edges of $ \Gamma_n $ in
the complement of $T$.
The main result of this section is the following
\begin{theorem}
\label{T2}
For each $ n \ge 0 $, the group $ G_n $ is isomorphic to $\, \pi_1({\mathbf \Gamma}_n,\,T)\,$.
In particular, $ G_n $ has an explicit presentation of the form \eqref{fgr2}.
\end{theorem}
\begin{proof}
Let $\,\GG_n := \CC_n \rtimes G \,$ denote the (discrete) transformation groupoid
corresponding to the action of $ G $ on $ \CC_n $. The canonical projection
$\,p:\, \GG_n \to G \,$ is then a covering of groupoids\footnote{Recall that
if $ \mathscr{E} $ and $ \mathscr{B} $ are (small connected) groupoids,
a covering $ p: \mathscr{E} \to \mathscr{B} $ is a functor that is surjective on objects
and restricts to a bijection  $ p: x \backslash \mathscr{E} \stackrel{\sim}{\to}
p(x)\backslash \mathscr{B} $
for all $ x \in \mathrm{Ob}(\mathscr{E}) $, where $\,x \backslash \mathscr{E} \,$
is the set of arrows in $ \mathscr{E} $ with source at $x$ (see \cite{M}, Chap.~3).},
which maps identically the vertex group of $ \GG_n $ at $\, (X_0,\,Y_0) \in \CC_n \,$ to the
subgroup $ G_n \subseteq G $. Now, each of the subgroups $A$, $B$ and $U$ of $ G $ can be lifted
to $ \GG_n\,$: $\,p^{-1}(A) = \GG_n \times_{G} A \,$,  $\,p^{-1}(B) = \GG_n \times_{G} B \,$
and $\,p^{-1}(U) = \GG_n \times_{G} U \,$, and these fibred products
are naturally isomorphic to the subgroupoids $\, \AG_n := \CC_n \rtimes A $, $\, \BG_n := \CC_n \rtimes B \,$ and $\, \UG_n := \CC_n \rtimes U \,$ of $ \GG_n$, respectively. Since the coproducts of groups agree with coproducts in the category of groupoids and the latter can be lifted through
coverings (see \cite[Lemma~3.1.1]{O}), the decomposition \eqref{amal} implies
\begin{equation}
\la{vank}
\GG_n = \AG_n *_{\UG_n} \BG_n\ ,\quad \forall\,n \ge 0\ .
\end{equation}
Unlike $ \GG_n$, the groupoids $ \AG_n $, $ \BG_n $ and $ \UG_n $ are not transitive
(if $\, n \ge 1 $), so \eqref{vank} can be viewed as an analogue of the Seifert-Van Kampen
Theorem for nonconnected spaces. As in topological situation, computing
the fundamental (vertex) group from \eqref{vank} amounts to contracting the connected components (orbits) of $ \AG_n $ and $ \BG_n $ to points (vertices) and $ \UG_n $ to
edges (see, e.g., \cite[Chap. 6, Appendix]{Ge}). This defines a graph which is exactly
$ \Gamma_n $. Now, choosing basepoints in each of the contracted components and assigning the fundamental groups
at these basepoints to the corresponding vertices and edges defines a graph of groups
({\it cf.} \cite{HMM}, p.~46). By \cite[Theorem~3]{HMM}, this graph of groups is (conjugate) isomorphic to the graph $ {\mathbf \Gamma}_n $ described above, and our group $ G_n $ is isomorphic
to $\, \pi_1({\mathbf \Gamma}_n,\,T)\,$.
\end{proof}

\subsection{Examples}
\la{S4.2.1}
We now look at the graphs $ \Gamma_n $ and groups $G_n$ for small $n$.

\subsubsection{}
\la{n=0}
For $ n = 0 $, the space $ \CC_n $ is just a point, and so are a fortiori
its orbit spaces. The graph $ \Gamma_0 $ is thus a segment, and the corresponding graph of
groups $ {\mathbf \Gamma}_0 $ is given by $\,[\,A \stackrel{U}{\longrightarrow} B\,] \,$.
Formula \eqref{fgr2} then says that $\,G_0 = A \ast_U B \,$ which agrees with \eqref{amal}.

\subsubsection{}
\la{n=1}
For $ n = 1 $, we have $\, {\mathcal C}_1 \cong \c^2 $, with
$ (X_0, Y_0) = (0,0) $. Since each of the groups
$A$, $B$ and $U$ contains translations $\,(x+a, y+b) \,$,
$\,a,b \in \c $, they act transitively on $ {\mathcal C}_1 $. So again $ \Gamma_1 $ is just the segment, and
$ {\mathbf \Gamma}_1 $ is given by $\,[\,A_1 \stackrel{U_1}{\longrightarrow} B_1\,]\,$,
where $\, A_1 := G_1 \cap A \,$, $\, B_1 := G_1 \cap B \,$ and $\, U_1 := G_1 \cap U \,$.
Since, by definition, $ G_1 $ consists of all $\, \sigma \in G \,$
fixing origin, the groups $A_1$, $ B_1 $ and $U_1 $ are obvious:
\begin{eqnarray}
A_1\!\!\!\!\! &:& (ax + by,\,cx+dy)\ ,
\quad a,\,b,\,c,\,d \in \c \ ,\ ad-bc=1\ ,
\nonumber \\
B_1\!\!\!\!\! &:& (ax + q(y),\,a^{-1} y)\ ,
\quad a \in \c^*\ ,\
q \in \c[y]\ ,  \ q(0) = 0\ , \nonumber \\
U_1 \!\!\!\!\! &:&  (ax + by,\,a^{-1} y)\ ,
\quad a \in \c^*\ ,\
b\in \c\ .
\nonumber
\end{eqnarray}
It follows from \eqref{fgr2} that $\, G_1 = A_1 \ast_{U_1} B_1 \,$.
In particular, $ G_1 $ is generated by its subgroups $ G_{1,x}$ and $ G_{1,y} $.

\subsubsection{}
\la{n=2}
The group $ G_2 $ has a more interesting structure. To describe the corresponding
graph $ \Gamma_2 $ we decompose
\begin{equation}
\la{decreg}
\CC_2 = \CC_2^{\rm reg}\,\bigsqcup \, \CC_2^{\rm sing} \ ,
\end{equation}
where $ \CC_2^{\rm reg} $ is the subspace of $ \CC_2 $ with $ Y $ diagonalizable.
The following lemma is established by elementary calculations.
\blemma
\la{Uorb}
The action of $\, U $ on $ \CC_2 $ preserves the decomposition
\eqref{decreg}. Moreover,

$(a)$ $\ \CC_2^{\rm reg} $ is a single $U$-orbit $\, \O^{\rm reg}\,$
passing through $ (X_1,\,Y_1) \in \CC_2 $ with
\begin{equation*}\la{reg1}
X_1 =  \begin{pmatrix}
0 & -1 \\
1 & 0
\end{pmatrix} \quad , \quad
Y_1 =  \begin{pmatrix}
1 & 0 \\
0 & 0
\end{pmatrix}\ .
\end{equation*}

$(b)$ $\ \CC_2^{\rm sing} $ consists of two orbits $ \O(2,1) $ and
$ \O(2,2) $ passing through $\, (X(2,1), \,Y_2) \,$ and $\,
(X(2,2),\,Y_2) \,$ with
\begin{equation*}\la{sing1}
X(2,1) =  \begin{pmatrix}
0 & 0 \\
-1 & 0
\end{pmatrix} \quad , \quad X(2,2) =  \begin{pmatrix}
0 & 0 \\
1 & 0
\end{pmatrix}\quad , \quad Y_2 =  \begin{pmatrix}
0 & 1 \\
0 & 0
\end{pmatrix}\ .
\end{equation*}
\elemma
Note that the orbit $\, \O^{\rm reg}\,$ is open
(and therefore has dimension $4$); $\, \O(2,1) $ and $ \O(2,2) $
are closed orbits of dimension $ 3 $.

\blemma \la{Borb}
The $B$-orbits in $ \CC_2 $ coincide with the $U$-orbits.
\elemma
\bproof Note that $ B $ belongs
to the right coset $ U\,\Psi_q $ generated by $ \Psi_q = (x + q(y), y) $ with  $\, q = a
y^2 + b y^3 + \ldots \,$. Since $ Y_2 $ is nilpotent,
such $ \Psi_q $ acts trivially on $ (X(2,r),\,Y_2)$, so
$\, B \,(X(2,r),\,Y_2) = U \, (X(2,r),\,Y_2) = \O(2,r)\,$
for $\,r=1,\,2\,$.
It follows that $ \O(2,1) $ and $ \O(2,2) $ are distinct $B$-orbits.
Since there are only three $U$-orbits in $ \CC_2 $,
$\, \O^{\rm reg} $ must be a separate $B$-orbit.
\eproof
\blemma \la{Aorb}
The group $ A $ acts transitively on $ \CC_2 $.
\elemma
\bproof Assume that $A$ has more than one orbit in $ \CC_2$. Since there are only
three $U$-orbits, at least one of the $A$-orbits (say, $ \O_A $) consists
of a single $U$-orbit. But then, by Lemma~\ref{Borb}, $ \O_A $ is also a
$B$-orbit. Since $ A $ and $ B $ generate $ G $, this means that $
\O_A $ is a $G$-orbit and hence, by Theorem~\ref{T4}, coincides
with $\CC_2$. Contradiction. \eproof
Summing up, we have $\, \CC_2/A = \{\O_A\} \,$ and
$$
\CC_2/B = \{\O_B^{\rm reg},\, \O_B(2,1),\, \O_B(2,2)\}\ ,
\quad
\CC_2/U = \{\O_U^{\rm reg},\,\O_U(2,1),\, \O_U(2,2)\}\ ,
$$
where $\O_B $ and $ \O_U $ denote the same subspaces in $ \CC_2 $
but viewed as $B$- and $U$-orbits respectively. Thus the graph $
\Gamma_2 $ is a tree which looks as
$$
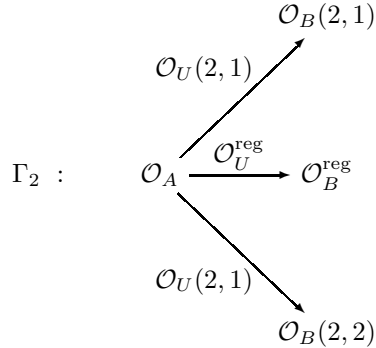
\Gamma_2\ :\
\begin{diagram}
&                        &                       &  \O_B(2,1)    \\
&                        & \ruTo^{\O_U(2,1)}     &               \\
&\O_A                     & \rTo^{\O_U^{\rm reg}} & \O_B^{\rm reg} \\
&                         & \rdTo_{\O_U(2,1)}     &               \\
&                         &                       &  \O_B(2,2)    \\
\end{diagram}
$$
Computing the stabilizers of basepoints for each of the orbits, we obtain the graph of groups
$$
{\mathbf \Gamma_2}\ :\
\begin{diagram}
&                        &                       &  G_{2,y} \rtimes T    \\
&                        & \ruTo^{T}     &               \\
&T                   & \rTo^{\Z_{2}\quad} & G^{(1)}_{2, y} \rtimes \Z_{2} \\
&                         & \rdTo_{T}     &               \\
&                         &                       & G_{2,x} \rtimes T    \\
\end{diagram}
$$
where $ T \subset G $ is the subgroup of scaling transformations
$\, (\lambda x, \lambda^{-1}y)\,$, $ \lambda \in \c^* $,
and the group $  G^{(1)}_{2, y}$ is defined in terms of generators \eqref{Phs} by
\begin{eqnarray*}
G^{(1)}_{2, y} &:=&
\{\,\Phi_{-x}\,\Psi_q \,\Phi_{ x} \in G\ :\ q \in \c[y]\ , \
q(\pm 1) = 0\,\}\ .
\end{eqnarray*}
Formula \eqref{fgr2} yields the presentation
\begin{equation}
\la{g2}
\,G_2 = (G_{2,x} \rtimes T) \ast_{T}
(G_{2,y} \rtimes T) \ast_{\Z_2} (G^{(1)}_{2, y}
\rtimes \Z_2) \ .
\end{equation}
In particular, $ G_2 $ is generated by its subgroups
$\, G_{2,x} $, $\, G_{2,y} $, $\, G^{(1)}_{2, y}$ and $ T $.

Using the above explicit presentations, it is easy to show that the groups
$\, G $, $ G_1 $ and $ G_2 $ are pairwise non-isomorphic (see \cite{BEE}).
In Section~\ref{S7}, we will give a general proof of this fact for all groups
$ G_n $. For $ n \ge 3 $, the amalgamated structure of $ G_n $ seems to be more complicated;
the corresponding graphs $ \Gamma_n $ are no longer trees (in fact, there are
infinitely many cycles).

\section{$ G_n $ as an algebraic group}
\la{S6}
In this section, we will equip $ G $ with the structure of an ind-algebraic group
that is compatible with the action of $ G $ on the varieties $ \CC_n $.
Each $ G_n \subseteq G $ will then become a closed subgroup and
hence will acquire an ind-algebraic structure as well.
We begin by recalling the definition and basic properties of
ind-algebraic varieties. Apart from the original papers of Shafarevich
\cite{Sh1, Sh2} a good reference for this material is Chapter~IV of \cite{Ku}.

\subsection{Ind-algebraic varieties and groups}
\la{S6.1}
An {\it ind-algebraic variety} (for short: an {\it ind-variety}) is a set  $\, X =
\bigcup_{k \ge 0} X^{(k)} \,$
given together with an increasing filtration
$$
X^{(0)} \subseteq X^{(1)} \subseteq X^{(2)} \subseteq \ldots
$$
such that each $ X^{(k)} $ has the structure of a finite-dimensional (quasi-projective)
variety over $ \c $, and each inclusion $  X^{(k)} \into  X^{(k+1)} $ is a closed
embedding of varieties. An ind-variety has a natural topology where a subset
$ S \subseteq X $ is open (resp, closed) iff $\, S^{(k)} := S \,\cap\, X^{(k)}\,$
is open (resp, closed) in the Zariski topology of $ X^{(k)}$ for all $ k $.
In this topology, a closed subset $ S $ acquires an ind-variety
structure defined by putting on $\, S^{(k)} \,$ the closed (reduced)
subvariety structure from $ X^{(k)} $. We call $S$ equipped with this
structure a {\it closed ind-subvariety} of $X$. More generally, any locally closed subset
$ S \subseteq X $ acquires from $X$ the structure of an ind-variety
since each $ S^{(k)} $ is a locally closed subset\footnote{The converse is not true:
a subset $ S \subset X $ may not be locally closed in $X$ even
though each of its components $ S^{(k)} $ is locally closed in $ X^{(k)} $.
A counterexample is given in \cite[Sect.~2.3]{FuM}.}
and hence a subvariety in $ X^{(k)} $.

An ind-variety $ X $ is said to be {\it affine} if each $ X^{(k)} $ is affine.
For an affine ind-variety $X$, we define its {\it coordinate ring} by $ \c[X] :=
\varprojlim_{\,k} \c[X^{(k)}] $, where $ \c[X^{(k)}] $ is the coordinate ring of $ X^{(k)} $. Naturally, $ \c[X] $ is a topological algebra equipped with the
inverse limit topology.

If $ X $ and $ Y $ are two ind-varieties with filtrations $ \{X^{(k)}\} $ and
$ \{Y^{(k)}\} $, a map $ f: X \to Y $ defines a {\it morphism} of ind-varieties
if for each $ k \ge 0 $, there is $ m \ge 0 $ (depending on $ k $)
such that $ f(X^{(k)}) \subseteq Y^{(m)} $ and the restriction of $f$
to $ X^{(k)}$, $\, f_k: X^{(k)} \to Y^{(m)} \,$, is a morphism of varieties.
A morphism of ind-varieties $ f: X \to Y $ is continuous
with respect to ind-topology and, in the case of affine ind-varieties,
induces a continuous algebra map $ f^*: \c[Y] \to \c[X] $.
\begin{example}
\la{prodvar}
For any ind-varieties $ X $ and $ Y $, the set $ X \times Y $ has
the canonical ind-variety structure defined by the filtration
$\, (X \times Y)^{(k)} := X^{(k)} \times Y^{(k)} \,$,
where $ X^{(k)} \times Y^{(k)} $ is a product in the category of varieties
({\it cf.} \cite[Example 4.1.3 (2)]{Ku}). The two natural projections
$ X \twoheadleftarrow X \times Y \onto Y $ are then morphisms of ind-varieties.
\end{example}

A morphism of ind-varieties $ f: X \to Y $ is called an {\it isomorphism} if $ f $
is bijective and $ f^{-1} $ is also a morphism. It is easy to see that a morphism
$ f: X \to Y $ of affine ind-varieties is an isomorphism iff the induced
map $ f^*: \c[Y] \to \c[X] $ is an isomorphism of topological algebras.
Two ind-variety structures on the same set $X$ are said to be {\it equivalent} if
the identity map $ \id: X \to X $ is an isomorphism.  It is natural not to
distinguish between equivalent structures on $X$.

\begin{example}
\la{indaff}
Any vector space $ V $ of countable dimension can be given the structure of
an (affine) ind-variety by choosing a filtration $ V^{(k)}$ by
finite-dimensional subspaces. It is easy to see that up to equivalence,
this structure is independent of the choice of filtration; hence
$ V $ has the canonical structure of an ind-variety which is denoted
$ {\c}^\infty $ ({\it cf.} \cite[Example 4.1.3 (4)]{Ku}).
\end{example}

A morphism of ind-varieties $f: X \to Y$ is called a {\it closed embedding}
if all the morphisms $ f_k: X^{(k)} \to Y^{(m)} $ are closed embeddings,
$ f(X) $ is closed in $ Y $ and $ f: X \to f(X) $ is a homeomorphism
under the subspace topology on $ f(X) $. The next lemma gives a useful
characterization of morphisms of ind-varieties in terms of closed
embeddings ({\it cf.} \cite[Lemma 4.1.2]{Ku}).
\begin{lemma}
\la{Lmor}
Let $X$, $Y$, $Z$ be ind-varieties. Let $f: X \to Y$ be a closed
embedding, and let $ g: Z \to X $ be a map of sets with the
property that for every $ k \ge 0 $ there is $ m \ge 0 $
such that $\, g(Z^{(k)}) \subseteq X^{(m)} $. Then $f$ is a morphism
(resp., closed embedding) iff $ f\circ g: Z \to Y $ is a morphism
(resp., closed embedding).
\end{lemma}
For example, if $ Z \subseteq Y $ is a closed ind-subvariety of $ Y $,
then the inclusion $ Z \into Y $ is a closed embedding.
Lemma~\ref{Lmor} shows that the converse is actually also true:
\begin{corollary}
\la{Lmor1}
If $ Z \subseteq Y $ is a closed subset of an ind-variety $ Y $,
there is a unique ind-variety structure on $ Z $ making
$ Z \into Y $ a closed embedding.
\end{corollary}
\bproof
Assume that $ Z $ has two ind-variety structures, say $Z'$
and $ Z''$, making $ Z \into Y $ into closed embeddings:
$\,i': Z' \to Y \,$ and $\,i'': Z'' \to Y \,$. To apply
Lemma~\ref{Lmor} we first take $ f := i'' $ and  $ g: Z' \to Z'' $
to be the identity map $ \id_Z $. Since $ f \circ g = i' $ and
$g = \id $ obviously satisfies the assumption of the lemma, we conclude that
$ \id: Z' \to Z'' $ is a morphism of ind-varieties. Reversing the roles
of $ Z' $ and $ Z'' $, we similarly conclude that $ \id: Z'' \to Z' $
is a morphism. Thus $ Z' \cong Z'' $.
\eproof
%
%
%

An {\it ind-algebraic group} (for short: an {\it ind-group}) is a group
$H$ equipped with the structure of an ind-variety such that the map
$\, H \times H \to H $, $\,(x,y) \mapsto x y^{-1} $, is a morphism of
ind-varieties. A {\it morphism of ind-groups}
is an abstract group homomorphism which is also a morphism of ind-varieties.
For example, any closed subgroup $K$ of $H$ is again an ind-group
under the closed ind-subvariety structure on $ H $, and the natural
inclusion $ K \into H $ is a morphism of ind-groups. Finally,
an action of an ind-group $ H $ on an ind-variety $ H $ is said to be
{\it algebraic} if the action map $ H \times X \to X $ is a morphism of ind-varieties.

\subsection{The ind-algebraic structure on $G$}
\la{S6.2}
Recall that $ R $ is the free associative
algebra on two generators $x$ and $y$. Letting
$ V $ be the vector space spanned by $x$ and $y$ we
identify $ R $ with the tensor algebra $ T_\c(V) :=
\bigoplus_{n \ge 0} V^{\otimes n}  $. Then, associated to the natural
tensor algebra grading is a filtration on $ R $ by vector subspaces:
\begin{equation}
\la{filtR}
R^{(0)} \subseteq  R^{(1)} \subseteq \ldots \subseteq R^{(k)} \subseteq R^{(k+1)}
\subseteq \ldots  \ ,
\end{equation}
where $\, R^{(k)} := \bigoplus_{n \leq k} V^{\otimes n} \,$.
Since each $ R^{(k)} $ has finite dimension, this filtration makes $ R $ an affine
ind-variety, which is isomorphic to $ \c^\infty $ (see Example~\ref{indaff}).
We write $\,\deg:\, R \to \Z_{\ge 0}\,\cup\,\{-\infty\}\,$ for the degree function
associated with \eqref{filtR}: explicitly, if $ p \in R $ is nonzero,
$\, \deg(p) := k \,\Leftrightarrow \, p \in R^{(k)}\! \setminus\! R^{(k-1)}\,$,
while $ \deg(0) := - \infty $ by convention. Thus
$\, R^{(k)} = \{p \in R \, : \, \deg(p) \leq k \}\,$.

Now, let $ E := \End(R) $ denote the set of all algebra endomorphisms of $ R $.
Each endomorphism is determined by its values on $x$ and $y$; hence we can identify
\begin{equation}
\la{idend}
E = R \times R\ ,\quad \sigma \mapsto (\sigma(x),\,\sigma(y))\ .
\end{equation}
This identification allows us to equip $ E $ with an ind-variety structure  by
taking the product on ind-variety structures on $ R $ (see Example~\ref{prodvar}):
\begin{equation*}
E^{(k)} := R^{(k)} \times R^{(k)} = \{(p,q) \in R \times R\, : \,
\deg(p) \leq k\,,\,\deg(q) \leq k\}
\end{equation*}
Clearly, $E$ is an affine ind-variety, which is actually isomorphic to $ \c^\infty $.
We define the degree function on $ E $ by $\, \deg(\sigma) :=
\max\{\deg(p),\, \deg(q)\} $, where $ \sigma = (p,q) \in E $.

Next, recall that we have defined $ G $ to be the subset of $E$
consisting of invertible endomorphisms that preserve
$ w = [x, y] \in R $. On the other hand, a well-known
theorem of Dicks, which is an analogue of the Jacobian conjecture for $R$,
implies that {\it every} endomorphism of $R$  that preserve $ w $ is
actually invertible (see \cite[Theorem~6.9.4]{C}). We will use this result to put an
ind-variety structure on $ G $.
\bprop
\la{indg}
There is a unique ind-variety structure on $ G $ making the inclusion $ G \into E $
a closed embedding.
\eprop
\bproof Consider the map
\begin{equation}
\la{commap}
c:\, E \to R\ ,\quad (p,q) \mapsto [p,q]\ ,
\end{equation}
where $\,[p,q] := p\,q - q\,p \,$ is the commutator in $ R $.
Since $c(E^{(k)}) \subseteq R^{(2k)} $ for all $k \ge 0 $ and the restrictions
$ c_k: E^{(k)} \to R^{(2k)} $ are given by polynomial equations,
\eqref{commap} is a morphism of ind-varieties (in particular,
a continuous map).
Under the identification \eqref{idend}, we have $ c(\sigma) = w $ for all elements
$ \sigma \in G $. Now, by Dicks' Theorem,
$ G $ actually coincides with the preimage of $ w $. Since
$c$ is a continuous map (and $ w \in R $ is a closed point),
$ G = c^{-1}(w) $ is a closed subset in $ E $. Letting $ G^{(k)} = G \cap E^{(k)} $
for $ k \ge 0 $ and putting on $ G^{(k)} $ the closed (reduced) subvariety structure
from $ E^{(k)} = R^{(k)} \times R^{(k)} $, we make $ G $ a closed ind-subvariety
of $ E $. Then $ G \into E $ is a closed embedding, and the uniqueness
follows from Corollary~\ref{Lmor1}.
\eproof

The group $ G $ equipped with the ind-variety structure of Proposition~\ref{indg}
is actually an affine ind-group. Indeed, by construction, $ G $ is an affine
ind-variety. We need only to show that
$ \mu: G \times G \to G $, $\,(\sigma, \tau) \mapsto \sigma \circ \tau^{-1} \,$,
is a morphism of ind-varieties. For this, it suffices to show that for
each $ k \ge 0 $, there is $ m = m(k) \ge 0 $ such that
$\, \mu((G \times G)^{(k)}) \subseteq G^{(m)} $. But since
$ G $ admits an amalgamated decomposition, a standard inductive
argument (see, e.g, \cite[Lemma~4.1]{K1}) shows that
$\,\deg(\tau^{-1}) \leq \deg(\tau) \,$ for all $ \tau \in G $.
This implies that $ \mu((G \times G)^{(k)}) \subseteq G^{(k^2)} $ for all $ k $.

\vspace{1ex}

\begin{remark}
The full automorphism group $ \tilde{G} := \Aut(R) $ of the algebra
$R$ has also a natural structure of an ind-group. In fact, by Dicks' Theorem,
$\, \tilde{G} $ is the preimage of the subset $\{\lambda w \in R\,:\, \lambda \in \c^*\} $
which is locally closed in the ind-topology of $ R $. By \cite[Lemma~2]{FuM},
$  \tilde{G} $ is then locally closed in $ E $ and hence has the
structure of an ind-subvariety of $ E $ with induced filtration
$ \tilde{G}^{(k)} = \tilde{G} \cap E^{(k)} $.
\end{remark}

\vspace{1ex}

We record two basic properties of the ind-group $ G $ which are similar to
the properties of the Shafarevich ind-group $ \Aut \,\c[x,y] $ (see \cite{Sh1, Sh2}).
First, recall ({\it cf.} \cite{Sh2, Bl}) that an ind-variety $ X $ is
{\it path connected} if for any $\,x_0,\, x_1 \in X\,$, there is an open set
$ U \subset {\mathbb A}^1_{\c} $ containing $0$ and $1$ and a morphism $
f: U \to X $ such that $ f(0) = x_0 $ and $ f(1) = x_1 $.
\blemma
\la{lincon}
A path connected ind-variety is connected.
\elemma
\bproof
Indeed, a morphism of ind-varieties $ f: U \to X $ is continuous. Hence, if
$ X $ is the disjoint union of two proper closed subsets which intersect with $ \im(f)$,
the preimages of these subsets under $f$ must be non-empty, closed and disjoint in $ U $.
This contradicts the fact that $ U $ is connected in the Zariski topology.
\eproof
\bthm
\la{Gconn}
The group $ G $ is connected and hence irreducible.
\ethm
\bproof
By Lemma~\ref{lincon}, it suffices to show any element of $ G $ can be joined
to the identity element $ e \in G $ by a morphism $ f: U \to G $. By Theorem~\ref{TCML},
any $\, \sigma \in G \,$ can be written as a composition $ \sigma = \Phi_{p_1} \Psi_{q_1} \,\ldots \,
\Phi_{p_n} \Psi_{q_n} $ of transformations \eqref{Phs}. Rescaling the polynomials
$ p_i$ and $ q_i$, we define
\begin{equation}
\la{family}
f: {\mathbb A}_\c^1 \to G\ ,\quad
t \mapsto \Phi_{tp_1} \Psi_{tq_1} \,\ldots \, \Phi_{tp_n} \Psi_{tq_n} \ ,
\end{equation}
which is obviously a morphism of ind-varieties such that $ f(1) = \sigma $ and
$ f(0) = e $. Hence $ G $ is connected.
On the other hand, it is known that any connected ind-group is actually
irreducible (see \cite[Prop.~3]{Sh2} and also \cite[Lemma~4.2.5]{Ku}).
\eproof

The next theorem is the analogue of \cite[Theorem~8]{Sh1}.
\bthm
\la{algsub}
Every finite-dimensional algebraic subgroup of $ G $ is conjugate to
either a subgroup of $ A $ or a subgroup of $ B $.
\ethm
\bproof
The proof is essentially the same as in the classical case; we
recall it for reader's convenience. By definition,
an algebraic subgroup $ H $ of $ G $ is a closed subgroup which is again
an ind-group with respect to the closed subvariety structure on
$ H^{(k)} = H \cap G^{(k)} $.
In particular, each $ H^{(k)} $ is closed in $ H $ and hence, when $ H $ is finite-dimensional,
there is a $ k\ge 0 $ such that $\, H^{(k)} = H^{(k+1)} = \ldots = H $.
It follows that $ H \subset G^{(k)} $ for some $k$, which means that
\begin{equation}
\la{bdeg}
\deg(\sigma) \leq k \ \ \mbox{for all}\ \ \sigma \in H\ .
\end{equation}
On the other hand, by Theorem~\ref{TCML}, every element $ \sigma \in G $ has a
reduced decomposition of the form
$$
\sigma = a_1 \, b_1 \, \ldots \, a_l \, b_l \, a_{l+1}
$$
where the $ b_i \in B \setminus A $ for all $ 1 \leq i \leq l $ and
$ a_j \in A \setminus B $ for $ 2 \leq j \leq l $. The number
$ l $ is independent of the choice of a decomposition and called the
length of $ \sigma $. As in the case of polynomial automorphisms (see \cite{Wr, FM}),
the length and the degree of $ \sigma $ are related by the formula
$\,
\deg(\sigma) = \deg(b_1) \deg(b_2) \ldots \deg(b_l) \,$,
which shows that any subset of $ G $, bounded in degree, is also
bounded in length. Thus \eqref{bdeg} implies that $H$ is a subgroup
of $ G $ bounded in length. A theorem of Serre (see \cite[Thm~4.3.8]{Se})
then implies that $H$ is contained in a conjugate of $A$ or $B$.
\eproof

\subsection{The ind-algebraic structure on $ G_n \,$}
\la{S6.3}
We have seen in Section~\ref{maincor} that each $ G_n $ is a maximal subgroup
of $ G $. It therefore natural to expect that $ G_n $ is an algebraic subgroup
of $ G $. This follows formally from the next theorem.
\bthm
\la{GnAlg}
The ind-group $ G $ acts algebraically on each space $ \CC_n $.
\ethm
\bproof
Recall that the action of $ G $ on $ \CC_n $ is defined by \eqref{gact}.
To see that this action is algebraic we first consider
\begin{equation}
\la{algact}
G \times \M_n(\c)^{\times 2} \xrightarrow{s \times \id} G \times \M_n(\c)^{\times 2}
\xrightarrow{\iota \times \id} E \times \M_n(\c)^{\times 2}
\xrightarrow{{\rm ev}} \M_n(\c)^{\times 2}\ ,
\end{equation}
where $ s: G \to G $ is the inverse map on $G$, $\, \iota: G \into E $
is the natural inclusion and $\, {\rm ev} \,$ is the evaluation map defined by
$\, [(p, q),\, (X,Y)] \mapsto (p(X,Y),\,q(X,Y)) \,$.
Each arrow in \eqref{algact} is a morphism of ind-varieties with
respect to the product ind-variety structure on
$ E \times \M_n(\c)^{\times 2} $; hence \eqref{algact} defines an
algebraic action of $ G $ on $ \M_n(\c)^{\times 2} $. This restricts to
an action $ G \times \tilde{\CC}_n \to \tilde{\CC}_n $, which is also
algebraic since $ \tilde{\CC}_n $ is a closed subvariety of
$ \M_n(\c)^{\times 2} $. Finally, as $ \PGL_n $ acts freely on $ \tilde{\CC}_n $
and $ G $ commutes with $ \PGL_n $, the quotient map
$ G \times \tilde{\CC}_n \to \tilde{\CC}_n \onto \CC_n $ is algebraic and
it induces an algebraic action $ G \times \CC_n \to \CC_n $, which is precisely
\eqref{gact}.
\eproof
By definition, $ G_n $ is the fibre of the action map
$\, p:\, G \to \CC_n $, $\,\sigma \mapsto
\sigma(X_0, Y_0)\,$, over the basepoint $ (X_0, Y_0) \in \CC_n $.
By Theorem~\ref{GnAlg}, this map is a morphism of ind-varieties: hence,
$ G_n $ is a closed subgroup of $ G $. We believe that the following is true.

\vspace{1ex}

\begin{conjecture}
\la{conjGn}
{\it The groups $ G_n $ are connected and hence irreducible for all $ n \ge 0 $.}
\end{conjecture}

\vspace{1ex}

In Section~\ref{S4.2.1}, we have explicitly described the structure of $ G_n $ as
a discrete group for small $n$. Using this explicit description, we can easily prove

\begin{proposition}
\la{congn}
Conjecture~\ref{conjGn} is true for $\, n=0,\,1,\, 2\,$.
\end{proposition}
\bproof
For $ n = 0 $, this is just Theorem~\ref{Gconn}. The argument of
Theorem~\ref{Gconn} can be also extended to $ G_1 $ and $ G_2 $,
since we know explicit generating sets for these groups.
Precisely, as shown in Section~\ref{n=1}, $ G_1 $ is generated by $ \Phi_p $ and $ \Psi_q $
with  $ p \in \c[x] $  and $ q \in \c[y] $ satisfying $ p(0) = q(0) = 0 $.
Hence, for any $ \sigma \in G_1 $, the morphism \eqref{family} constructed
in the proof of Theorem~\ref{Gconn} has its image in $ G_1 $, which, by Lemma~\ref{lincon},
implies that $ G_1 $ is connected.
Similarly, by \eqref{g2}, $ G_2$ is generated by its subgroups
$G_{2,x} \, , \, G_{2,y} \, , \, G^{(1)}_{2,y}$ and $T$, each of which
is path connected. For example,  every $ \sigma \in G^{(1)}_{2,y} $
has the form $\, \Phi_{-x} \Psi_{q} \Phi_x \,$, with $ q \in \c[y] $
satisfying $ q(\pm 1) = 0 $. The last condition is
preserved under rescaling $ t \mapsto t q $. Hence, we can join $ \sigma $
to $ e $ within $ G^{(1)}_{2,y} $ by the algebraic curve $\, f: t \mapsto  \Phi_{-x} \Psi_{tq}
\Phi_x $. This shows that $ G_2 $ is connected.
\eproof

Unfortunately, for $ n \ge 3 $, the Bass-Serre decomposition of $ G_n $ is
too complicated, and its factors (generating subgroups of $ G_n $) are much
harder to analyze. In general, one might try a different approach using basic topology\footnote{We thank P.~Etingof for suggesting us this idea.}.
First, observe that, since $G$ acts transitively on $ \CC_n $ and
$ \dim\, \CC_n < \infty $, the action map $ p:\, G \to \CC_n $
restricts to a surjective morphism of affine varieties
$\, p_k:\, G^{(k)} \to \CC_n $ for $ k \gg 0 $. Thus there is a fibration
\begin{equation}
\la{fibr}
\begin{diagram}[small, tight]
G_n^{(k)} & \rInto  & G^{(k)} \\
\dTo      &         & \dOnto_{p_k} \\
\ast      & \rTo & \CC_n \\
\end{diagram}
\end{equation}
Since $ G^{(k)} $ and $ \CC_n $ are algebraic varieties over $ \c $,
we can equip them with classical topology: we write $ G^{(k)}(\c) $
and $ \CC_n(\c) $ for the corresponding complex analytic varieties.
The key question then is

\vspace{1ex}

\noindent
{\bf Question.}\ Is \eqref{fibr} locally trivial in classical topology for $ k \gg 0 \,$?

\vspace{1ex}

Assume (for a moment) that the answer is `yes'. Then,
for $ k \gg 0$, there is an exact sequence associated to \eqref{fibr}:
\begin{equation}
\la{homseq}
\ldots \to  \pi_1 \,(G^{(k)}(\c),\, \tilde{x}_0)
\xrightarrow{(p_k)_*} \pi_1\,(\CC_n(\c),\, x_0) \to \pi_0\,[G_n^{(k)}(\c)] \to 0\ ,
\end{equation}
where $ \pi_0\,[G_n^{(k)}(\c)] $ is the set of connected
components of $ G_n^{(k)}(\c) $.
Now, it is known that
$ \CC_n(\c) $ is homeomorphic to the Hilbert scheme
$ {\rm Hilb}_n(\c^2) $ and hence is simply connected.
It follows from \eqref{homseq} that $ G_n^{(k)}(\c) $
is connected and hence $ G_n^{(k)} $ is connected in
Zariski topology. If this holds for {\it all}
$\, k \gg 0 \,$, then  by
\cite[Prop.~2]{Sh2} (see also \cite[Prop.~2.4]{K2}),
the ind-variety $ G_n $ must be connected.

\subsection{A nonreduced ind-scheme structure on $G$}
\la{S6.2.2}
In this section, we define another ind-algebraic structure
on the group $ G $. Although, strictly speaking,
this structure does not obey Shafarevich's definition,
in some respects, it is more natural than the one introduced
in Section~\ref{S6.2}.

Recall that, set-theoretically, $ G $ can be identified
with the fibre over $ w = [x,y] $ of the commutator map $\, c: E \to R \,$, see \eqref{commap}. As shown in Section~\ref{S6.2}, $\, c $ is a morphism of two affine ind-varieties equipped with canonical filtrations. Let $ c_k:\, E^{(k)} \to R^{(2k)} \,$ denote the restriction of $c$ to the corresponding filtration components;
by construction, $\, E^{(k)} $ and $ R^{(2k)} $ are
finite-dimensional vector spaces and $\, c_k \,$ is a polynomial map. Now, for each $ k \ge 0 $, we define $ {\mathcal G}^{(k)} $ to be the
{\it scheme-theoretic} fibre of $ c_k $ over the closed point
$ w \in R^{(2k)} \,$: that is,
\begin{equation*}
\la{indsc}
\G^{(k)} := \Spec\ \c[E^{(k)}]/c_k^*({\mathfrak m}_w) 
\end{equation*}
where $ {\mathfrak m}_w \subset \c[R^{(2k)}] $ is the maximal ideal
corresponding to $w$. Clearly, for all $ k $, we have closed embeddings
of affine schemes
\begin{equation}
\la{indfil}
\ldots \into \G^{(k-1)} \into \G^{(k)} \into \G^{(k+1)} \into \ldots
\end{equation}
induced by the natural inclusions $ E^{(k-1)} \subset E^{(k)} \subset E^{(k+1)} $.
Moreover, if we identify $ G^{(k)} \,$ (the set-theoretic fibre of $ c_k $)
with $\, \Spec\, \c[\G^{(k)}]_{\rm red} \,$, we get the commutative diagram of
affine schemes
\begin{equation}
\la{inddia}
\begin{diagram}[small, tight]
\ldots\,& \rInto & G^{(k-1)} & \rInto &  G^{(k)} & \rInto &  G^{(k+1)} & \rInto & \ldots \\
        &        &     \dTo   &        & \dTo     &        &    \dTo    &         &   \\
\ldots\,& \rInto & \G^{(k-1)} & \rInto &  \G^{(k)} & \rInto &  \G^{(k+1)} & \rInto & \ldots
\end{diagram}
\end{equation}
with vertical arrows corresponding to the algebra projections $ \c[\G^{(k)}] \onto
\c[\G^{(k)}]_{\rm red} $.

The filtration \eqref{indfil} defines on $ G $ the structure of an
affine ind-scheme\footnote{By an affine ind-scheme we mean a
countable inductive limit of closed embeddings
in the category of affine $\c$-schemes  ({\it cf.} \cite{K2, K3}). Clearly,
any affine ind-variety in the sense of Section~\ref{S6.1} is an example of
an affine ind-scheme.}, which we denote by $ \G $. The diagram
\eqref{inddia} gives a canonical morphism of affine ind-schemes
\begin{equation}
\la{tautmap}
i:\, G \to \G
\end{equation}
that reduces to the identity map on the underlying sets. We will prove
\begin{proposition}
\la{noniso}
The map $ i $ is\, {\rm not} an isomorphism of affine ind-schemes.
\end{proposition}
%


We begin with general remarks on tangent spaces.
Recall that, if $X$ is an affine $\c$-scheme and $ x \in X $ a closed  point, the (Zariski) tangent space to $ X $ at $x$ is defined by
$\, T_x X := ({\mathfrak m}_x/{\mathfrak m}^2_x)^* $, or equivalently
$$
T_x X = \Der(\c[X], \c_x) \ ,
$$
where $ \c_x = \c $ is viewed as a $ \c[X] $-module via the algebra map $\,\c[X] \to \c \,$ corresponding to $x$. A morphism $ f: X \to Y $ of
affine schemes defines an algebra map $ f^*: \c[Y] \to \c[X] $, which
in turn, induces the linear map
$$
\,d f_x:\, T_x X \to T_y Y \ ,
\quad \partial \mapsto \partial \circ f^* \ ,
$$
called the differential of $f$ at $x$. The kernel of $ df_x $ is canonically isomorphic to the tangent space of $\, {\mathcal Z} :=
\Spec \,\c[X]/f^*({\mathfrak m}_y) \,$,
the scheme-theoretic fibre of $f$ over $ y = f(x) $, and we identify
$$
T_x {\mathcal Z} = \Ker(df_x) \ .
$$
If $\,X,\,Y\,$ are varieties and
$ Z :=  {\mathcal Z}_{\rm red} = f^{-1}(y) $ is the (set-theoretic) fibre
of $ f $, then there is a canonical map $ i:\,Z \into {\mathcal Z} $, which,
for all $ x \in Z $, induces an inclusion
\begin{equation}
\la{zartan}
di_x:\, T_x Z \into T_x {\mathcal Z} = \Ker(df_x)\ .
\end{equation}
In particular, if $ {\mathcal Z} $ is
reduced (i.e., $\,i\,$ is an isomorphism), then $ T_x Z = \Ker(df_x) $.

Now, let $ X $ be an affine ind-variety with filtration $ \{X^{(k)}\}$.
For any $ x \in X $, there is $ k_0 \ge 0 $ such that $ x \in X^{(k)}$
for all $ k \ge k_0 $. Hence, the filtration embeddings
$ X^{(k)} \into X^{(k+1)} $ induce the sequential
direct system of vector spaces
$$
T_x X^{(k)} \to T_x X^{(k+1)} \to T_x X^{(k+2)} \to \ldots \ ,
$$
and the corresponding direct limit $\, T_x X := \dlim_{k} T_x X^{(k)} \,$ is
called the tangent space to $ X $ at $ x $ ({\it cf.} \cite[4.1.4]{Ku}).
If $ X = V $ is an ind-vector space filtered by finite-dimensional subspaces
$ V^{(k)} $ (see Example~\ref{indaff}), then, for each $ k \ge 0 $, we can
identify $ T_x V^{(k)} = V^{(k)} $, using the canonical isomorphism
\begin{equation}
\la{dirder}
V^{(k)} \stackrel{\sim}{\to} T_x V^{(k)}\ , \quad v \mapsto \partial_{v,x}\ ,
\end{equation}
where $\,\partial_{v,x} \in \Der(\c[V], \c) \,$ is defined by
$\, \partial_{v,x} F = \partial_{v}F(x) := (d/dt)[F(x+vt)]_{t=0}\,$.
With these identifications, the natural inclusions $ T_x V^{(k)} = V^{(k)} \into V $
induce an injective linear map $ T_x V = \dlim_{k} T_x V^{(k)} \into V $, which is
actually an isomorphism, since $ \dlim_{k} V^{(k)} = \bigcup_{k \gg 0} V^{(k)} = V $.
Thus, just as in the finite-dimensional case, we can identify $\, T_x V = V \,$
using \eqref{dirder}.

Next, let $ f: V \to W $ be a morphism of ind-varieties, each of which is an ind-vector
space of countable dimension. Fix $ w \in W $ and let $ Z := f^{-1}(w) $ be the
(set-theoretic) fibre of $ f $ over $w$. Assume that $ Z \ne \varnothing $ and put the
induced topology on $ Z $, i.e. $ Z^{(k)} = Z \cap V^{(k)} $. Then, there is $ k_0 \ge 0 $
such that $ w \in \im(f_k) $ for all $ k \ge k_0 $, and we obviously have
$ Z^{(k)} = f_k^{-1}(w) $, where $ f_k: V^{(k)} \to W^{(m)} $ is the restriction of
$ f $ to $ V^{(k)} $. Let
$ {\mathcal Z}^{(k)} := \Spec\, \c[V^{(k)}]/f_k^*({\mathfrak m}_w) $ be the
scheme-theoretic fiber of $ f_k $ over $ w \in W^{(k)} $, and let
$\,i_k :\, Z^{(k)} \into {\mathcal Z}^{(k)} $ be the canonical maps.
Then, by \eqref{zartan},  we have
$$
(di_k)_x:\, T_x Z^{(k)} \into T_x {\mathcal Z}^{(k)} = \Ker\,(d f_k)_x
$$
for all $ k \ge k_0 $. Hence
$$
T_x Z := \dlim T_x Z^{(k)} \, \into \, \dlim \Ker\,(d f_k)_x = T_x {\mathcal Z}\ .
$$
On the other hand, since $\, \dlim \,$ preserves exact sequences, the
exactness of
$$
0 \to \Ker\,(d f_k)_x \to V^{(k)} \xrightarrow{(df_k)_x} W^{(m)}
$$
implies
$\,
\dlim \Ker(d f_k)_x = \Ker\,(d f)_x = \{v \in V:\, \partial_v f(x) = 0 \}\,$.
Thus,
\begin{equation}
\la{incl}
T_x {\mathcal Z} = \{v \in V:\, \partial_v f(x) = 0 \}\ ,
\end{equation}
and the differential of $\, i: Z \to Z \,$ induces an
inclusion $\, di_x:\, T_x Z \into T_x {\mathcal Z} \,$, which
is an isomorphism whenever the fibres of $ f_k $ are reduced for all $ k \gg 0 $.

To prove Proposition~\ref{noniso} we apply the above remarks to
the morphism $ c: E \to R $. Under \eqref{idend}, the
identity element of $ G $ corresponds to $ e = (x,y) \in E $, and
for any $ \boldsymbol{v} = (u, v) \in R^2 $, we have
$$
\partial_{\boldsymbol{v}} c(e) = (d/dt)[c(e+t\boldsymbol{v})]_{t=0} = [x, v] + [u, y]\ .
$$
Hence, by  \eqref{incl}, we can identify
\begin{equation}
\la{TeG}
T_e \G = \{(u, v) \in R^2\ :\ [x, v] + [u, y] = 0\}\ .
\end{equation}
and the differential of \eqref{tautmap} gives an embedding
\begin{equation}\la{diffeg}
di_e:\, T_e G \into T_e \G \ .
\end{equation}
\begin{proof}[Proof of Proposition~\ref{noniso}]
It suffices to show that \eqref{diffeg} is not surjective.
We identify $ T_e G $ with the image of $di_e$ and show that $ T_e G \ne T_e \G $.
To this end we will use the canonical anti-involution on
the algebra $ R $ defined by $ x^{\dagger} = x $ and
$ y^{\dagger} = y $ and $ (ab)^{\dagger} = b^{\dagger} a^{\dagger} $ for all $ a,b \in R $. The
elements of $R$ invariant under $\,\dagger\,$ are called {\it palindromic}:
we write $ R^\dagger := \{a \in R \,:\, a^\dagger = a\} $.
Extending $\,\dagger \,$ to $ E = R^2 $ by
$ (p,q)^{\dagger} := (p^{\dagger}, q^{\dagger}) $, we get an (auto)morphism
of the ind-variety $ E $. Now, it is known (see \cite[Corollary 1.5]{SY}) that
the map $\, \dagger: E \to E $ restricts to the identity map on $ G $ and hence induces
the identity map on $ T_e G $. This means that $ T_e G \subseteq (T_e \G)^\dagger $,
where $ (T_e \G)^\dagger := \{(u, v) \in T_e \G  \, :\, u^\dagger = u\ ,\ v^\dagger = v\} $.
Thus it suffices to show that $ (T_e \G)^\dagger \ne T_e \G  $.
This can be verified directly: for example, take in $R$ the following elements
$$
a = x y^2 x^2 + x^2 yxy + yx^2 yx \ ,\quad b = y x^2 y^2 + y^2 xyx + xy^2 xy
$$
and define $ u := a - a^{\dagger} $ and $ v := b - b^{\dagger} $, so that $ u^\dagger = - u $ and $ v^\dagger = - v  $. Then obviously
$ (u,v) \not\in (T_e \G)^{\dagger} $, but a trivial calculation shows that
$\, (u,v) \in T_e \G  $.
\end{proof}
%


\noindent
{\bf Remarks.}
1. The proof of Proposition~\ref{noniso} shows that the schemes $ \G^{(k)} $ are actually
nonreduced for $ k \ge 5 $.

2. There is an intrinsic Lie algebra structure on $ T_e \G $ coming from the fact that
$ \G $ is an ind-group ({\it cf.} \cite{Sh2}, \cite[Sect.~4.2]{Ku}). With
identification \eqref{TeG}, the Lie bracket on $ T_e \G $ can be described
as follows. Let $ \Der(R) $ be the Lie algebra of linear derivations of $ R $,
and let $ \Der_w(R) $ be the subalgebra of $ \Der(R) $ consisting of derivations that vanish at $ w $. Identify $ \Der(R) = R^2 $ via the evaluation map $\,\delta \mapsto (\delta(x), \delta(y))\,$.
Then $ T_e \G \subset R^2 \,$ corresponds precisely to $ \Der_w(R) $ and the Lie bracket on
$ T_e \G $ corresponds to the commutator bracket on $ \Der_w(R) $. Thus, as was originally suggested in \cite{BW}, there is an isomorphism of Lie algebras
$\, {\rm Lie}(\G) \cong  \Der_w(R) \,$.

3. Proposition~\ref{noniso} shows that the Lie
algebra $ {\rm Lie}(G) $ of the ind-group $ G $ is a
{\it proper} subalgebra of $ {\rm Lie}(\G) $.
In fact, $ {\rm Lie}(G) \subseteq {\rm Lie}(\G)^{\dagger}
\subsetneqq {\rm Lie}(\G) $. We expect that
$ {\rm Lie}(G) $ is generated by the two abelian Lie
subalgebras $ {\rm Lie}(G_x) $ and $ {\rm Lie}(G_y) $, which are
spanned by the derivations $ \{(x^k, 0)\}_{k \in \N} $ and
$  \{(0,y^m)\}_{m \in \N} $ ({\it cf.} \cite[Question~17.10]{EG}).

4. There is an appealing description of the Lie algebra $ \Der_w(R) $, due to Kontsevich \cite{Ko} (see also \cite{G} and \cite{BL}).
Specifically, $ \Der_w(R) $ can be identified with $ \bar{\L} := R/([R,R]+\c) \,$, the space of cyclic words in the variables $x$ and $y$. The Lie bracket on $ \Der_w(R) $ corresponds to a Poisson bracket on $ \bar{\L} $ defined in terms of cyclic derivatives (see \cite[Sect.~6]{Ko}). Note that the
Lie algebra $ \bar{\L} $ has an obvious one-dimensional central extension: $ \L := R/[R,R] $. By a theorem of Ginzburg \cite{G},  the varieties $ \CC_n $ can be naturally embedded in $ \L^* $ as coadjoint orbits, and thus can be identified with coadjoint orbits of a central extension of $ \G $. Conjectures~\ref{Conj1} and~\ref{Conj2} in Section~\ref{Conj} suggest that Ginzburg's theorem extends to all configuration spaces $ \CC_n^{[k]}$ and their products $ \CC^{[k_1]}_{n_1} \times \CC^{[k_2]}_{n_2} \times \ldots \times \CC^{[k_m]}_{n_m}\,$
(with $ n_i \ne n_j $).

\section{Borel Subgroups}
\la{S7}
In this section, we will study the Borel subgroups of $ G_n $ and prove our main
classification theorems stated in the Introduction.
Recall that the natural action of the group $ G $ on $ \c^2 $ is faithful
(see Proposition~\ref{PCML} and remark thereafter). Using this action, we will identify
$ G $ as a {\it discrete} group with a subgroup of polynomial
automorphisms in $ \Aut(\c^2) \,$: in other words, we will think of the elements
of $G$ (and hence $G_n$) as automorphisms of $ \c^2 $.
This will allow us to apply the results of \cite{FM} and \cite{L}.
On the other hand, to define Borel subgroups we will regard $G$ and $G_n$ as
{\it topological} groups with (reduced) ind-Zariski topology introduced in
Section~\ref{S6}.

\subsection{Friedland-Milnor-Lamy classification}
\la{S7.1}
By \cite{FM}, the elements of $ G  $ can be divided into two separate classes
according to their dynamical properties as automorphisms of $ \c^2 $: every
$ g \in G $ is conjugate to
either an element of $B$ or a composition of generalized H\'enon automorphisms of the form:
$$
\sigma\, g \, \sigma^{-1} = g_1 \, g_2\, \ldots \, g_m\ ,
$$
where $\, g_i = (y, x  + q_i(y)) \,$ with polynomials $\, q_i(y) \in \c[y] \,$ of degree $\, \ge  2 \,$. We say  that $ g $ is of \textit{elementary} or \textit{H\'enon} type, respectively. A subgroup $ H \subseteq G $ is called {\it elementary} if each element of $ H $ is of elementary type.

It is convenient to reformulate this classification in terms of the action of $ G $ on the standard
tree $ \mathcal{T} $ associated to the amalgam $ G = A*_{U} B $. By definition,
the vertices $  V(\mathcal{T}) $ of $\,\mathcal{T} $ are the left cosets $ G/A \,\sqcup\, G/B\,$, while the set of edges is $ E(\mathcal{T}) = G/ U $. The group $ G $ acts on $ \mathcal{T}  $ by left translations. Notice that if  $ g \in G $ fixes two vertices in $ \mathcal{T} $, then it also fixes all the vertices linking these two vertices; thus, for each $ g \in G $, we may define a subtree $ \mathtt{Fix}(g) \subseteq \mathcal{T} $ fixed by $g$. More generally, if $H $ is a subgroup of $G$, following \cite{L}, we put $ \mathtt{Fix}(H) :=  \cap_{g \in H}\,\mathtt{Fix}(g) $.
It is easy to see that $ \mathtt{Fix}(g) $ is non-empty iff $ g $ is elementary, and
$\,\mathtt{Fix}(g) = \varnothing\,$ iff  $ g $ is of H\'enon type. In the latter case, following \cite{Se}, we may define the {\it geodesic} of $g$ to be the set of vertices of $ \mathcal{T} $ that realizes the infimum $\, \inf_{p \in V(\mathcal{T})}\,\mathrm{dist}(p, g(p))\,$, where $\, \mathrm{dist}(p, q)\,$ is the number of edges of the shortest path joining the vertices $p$ and $q$ in $ \mathcal{T} $.

The following theorem is a consequence of the main result of S.~Lamy.
\begin{theorem}[\cite{L}]
\la{TLamy}
Let $ H $ be a subgroup of $G$. Then, one and only one of the following possibilities occurs:
\begin{enumerate}
\item [(I)] $H$ is an elementary subgroup conjugate to a subgroup of $A$ or $B$.

\item [(II)] $H$ is an elementary subgroup which is not conjugate to a subgroup of $A$ or $B$.
Then $ H $ is countable and abelian.

\item [(III)] $H$ contains elements of H\'enon type, and all such elements in $H$ share
the same geodesic. Then $H$ is solvable and contains a subgroup of finite index
isomorphic to $ \Z $.
\item [(IV)] $H$ contains two elements of H\'enon type with distinct geodesics. Then $H$ contains a free subgroup on two generators.
\end{enumerate}
\end{theorem}

\vspace{1.3ex}

\noindent
{\bf Remarks.}
1. Theorem~\ref{TLamy} is essentially Th\'eor\`eme~2.4 of \cite{L}, except that this last
paper is concerned with subgroups of the full automorphism group $ \Aut(\c^2) $.
As a subgroup of $ \Aut(\c^2) $, $\,G$ coincides with the kernel of the Jacobian map
$ {\rm Jac}: \Aut(\c^2) \to \c^* $, which splits and gives an identification
$\,\Aut(\c^2) \cong G \rtimes \c^* \,$. Using this we can easily
deduce Theorem~\ref{TLamy} from \cite[Th\'eor\`eme~2.4]{L}. Indeed,
if $ H $ is an elementary subgroup of $ G $, then (by definition) it is
elementary in $ \Aut(\c^2) $ and hence is either of type I or type II in that group.
If $ H $ is of type II in $ \Aut(\c^2) $ then it is automatically of type II in $G$.
If $ H $ is of type I in $ \Aut(\c^2) $, then, by \cite[Th\'eor\`eme~2.4]{L},
it can be conjugate to a subgroup $ \tilde{H} $ of $ \Aut(\c^2) $, which is
either in $ A \rtimes \c^* $ or $ B \rtimes \c^* $. But $ {\rm Jac}(\tilde{H}) =  {\rm Jac}(H) = 1 $,
hence $ \tilde{H} \subset G $, and since $\,\Aut(\c^2) \cong G \rtimes \c^* \,$. we can
conjugate $ H $ to $\tilde{H}$ within $ G $. For types III and IV, the implication
Th\'eor\`eme~2.4  $\,\Rightarrow\,$ Theorem~\ref{TLamy} is automatic.

2. We have added to \cite[Th\'eor\`eme~2.4]{L} that any subgroup $H$
of type III contains a finite index subgroup
isomorphic to $\Z$. Indeed, by \cite[Prop.~4.10]{L}, all subgroups of $G$ satisfying
property (III) for a {\it fixed} geodesic generate a unique largest subgroup, which contains a copy of $ \Z $ as a subgroup of finite index. The last condition means that $\,H\,\cap\,\Z\,$ is a subgroup of finite index in $ H $; hence $H$ either contains $ H\,\cap\,\Z \cong \Z $ as a subgroup of finite index or is finite (if $\, H\,\cap\,\Z = \{1\}$).
It remains to note that the last possibility does not occur, since, by \cite[Theorem~8, Sect.~I.4.3]{Se}, any finite subgroup of $G$ is conjugate to a subgroup of $A$ or $B$ and hence is of type I.

\vspace{1.3ex}

As a consequence of Theorem~\ref{TLamy}, the following {\it Tits alternative} holds for $ G\,$
({\it cf.} \cite{L}, Corollary~2.5): every subgroup of $ G $ contains either a solvable subgroup of finite index or a non-abelian free group.

\subsection{Solvable subgroups of $G$}
\la{S7.2.3}
We begin with the following observation which may be of independent interest.
\begin{lemma}
\la{clos}
For $g \in G$, one and only one of the following possibilities occurs:
\begin{enumerate}
\item[(a)] $g$ is elementary,  and $\langle g \rangle \cong \mathbb{Z}\,$,
\item[(b)] $g$ is elementary and $\langle g \rangle \cong \mathbb{Z}_n$ for some $n\geq 1\,$,
\item[(c)] $g$ is H\'enon type and $\langle g \rangle \cong \mathbb{Z}\,$.
\end{enumerate}
Moreover, $\langle g \rangle $ is a closed subgroup of $G$ if and only if it is as in {\rm (b)} or {\rm (c)}.
\end{lemma}
\begin{proof}
By the Friedland-Milnor classification, any element of $G$ is either elementary
(i.e. conjugate to an element of $B$) or of H\'enon type.

Suppose that $g$ is elementary. Then we may assume that $g$ is contained in $B$.
If $g$ has finite order, $\,\langle g \rangle \cong \mathbb{Z}_n$ for some $n\geq 1$.
Since $ \langle g \rangle $ is finite, it is a closed subgroup of $G$.  If $g$ has
infinite order then $\langle g \rangle \cong \mathbb{Z}$. Moreover
$\langle g \rangle \subset G^{(k)}$, for some $k$, where $G^{(k)}$ is $k$-th filtration component of $G$. Since $ \|\langle g \rangle\|$ is countable, it cannot be closed in $G^{(k)}$.

Suppose $g$ is of H\'enon type. For a H\'enon automorphism, the sequence $\{\deg(g^k)\}^{\infty}_{k=1}$ is strictly
increasing, and we have $\lim_{k \rightarrow \infty } \deg (g^k)= \infty$. For any $n>0$,
$G^{(n)} \cap \langle g \rangle$ is finite and hence closed. Thus  $\langle g \rangle$ is equipped
with an increasing filtration of closed sets therefore it is an ind-subgroup of $G$.
\end{proof}
\begin{proposition}
\la{typeIII}
Let $H$ be a subgroup of $G$ with either of the following properties:

{\rm (S1)} $H$ is a solvable group without a proper subgroup of finite index,

{\rm (S2)} $H$ is a connected solvable group.

Then $H$ cannot be of type III (in the nomenclature of Theorem~\ref{TLamy}).
\end{proposition}
\begin{proof}
Suppose $H$ is a type III subgroup. Then $H$ is a subgroup of the group $K$ explicitly
described in \cite[Proposition~4.10]{L}.
By the proof of this proposition, $K$ has a finite index subgroup generated by a
H\'enon type element. We denote this subgroup by $ K_1 $.  Since $ H \subseteq K $,
we have $ H/(H \cap K_1) \subseteq K /K_1 $ and $ H/(H \cap K_1) $ is finite.
This shows that $H$ with property (S1) cannot be of type III, since
$ H \cap K_1 $ is a subgroup of finite index in $H$.

Since $K_1$ is closed in $G$, by Lemma~\ref{clos}(c), $ H \cap K_1$ is closed in  $ H$.
Therefore $ H = \bigcup^n_ {i=1} g_i ( H \cap K_1)$  is the disjoint union of closed subsets.
Since $ H$ is connected, we must have $ H \cap K_1 =  H $,
hence $ H \subseteq K_1 $. It follows that either $ H = \langle g \rangle$ for some $g \in K_1$
or  $ H = 1$. One can easily see that  $ H  = \langle g \rangle$ cannot be connected:
$ H_1 = \langle g^2 \rangle $ is its closed subgroup of index 2, hence $ H= g  H_1 \cup  H_1$ is the disjoint union of closed subsets. Therefore $H=1$. This proves (S2).
\end{proof}
\subsection{Borel subgroups of $G$}
\la{S7.3}
Recall that a subgroup of a topological group is called {\it Borel}\,
if it is connected, solvable and maximal among all connected solvable subgroups.
For basic properties of Borel subgroups we refer to \cite[\S~11]{Bo1}. We only
note that any Borel subgroup is necessarily a closed subgroup.

We begin with the following proposition which establishes the main properties
of the subgroup of triangular automorphisms in $G$.
%
%
%
%
%
%
%
%

\begin{proposition}
\la{B-msolv}
Let $B$ be the subgroup of $G$ defined by \eqref{B}. Then
\begin{enumerate}
\item[(a)] $B$ is a solvable group of derived length 3.
\item[(b)] $\mathtt{Fix}(B)=\{1 \cdot B\}$ consists of a single vertex.
\item[(c)] $N_G (B) = B$.
\item[(d)] $B$ is a connected subgroup of $G$.
\item[(e)] $B$ is a maximal solvable subgroup of $G$.
\end{enumerate}
In particular, $B$ is a Borel subgroup of $G$.
\end{proposition}

\begin{proof}
(a)  One can easily compute the derived series of $B\,$, which is given by
$$
B^{(1)} = \{(x+p(y),y+f) \, | \, f \in \c \, , \, p(y) \, \in \c[y] \} \, ,  \quad  B/B^{(1)} \cong \c^*
$$
$$
B^{(2)} = \{(x+p(y),y) \, |  \, p(y) \, \in \c[y] \} \, ,  \quad  B^{(1)} /  B^{(2)} \cong \c
$$

(b) It is clear that $\mathtt{Fix}(B)$ contains $ \{  1 \cdot B\}$. Now,
by \cite[Proposition 3.3]{L}, there is an element $f \in B$ such that
$\mathtt{Fix}(f) =\{1 \cdot B\}$. Hence $\mathtt{Fix}(B)=\{1 \cdot B\}$.

(c) Let $g \in N_G (B)$, i.e. $g^{-1} B g \subseteq B$. Then $B \subseteq g B g^{-1}$.
Hence $B$ must also fix the vertex $g \cdot B$. By part (b), $g \cdot B = 1 \cdot B$ and $g \in B$.

(d)  By Lemma~\ref{lincon}, it suffices to show that $B$ is path connected.
Let $b=(t x + p(y), t^{-1} y+f )$ be an arbitrary element in $B$. Consider $b_{s}=
(t x + s \, p(y), t^{-1} y+s \, f ) \in B $ for $s \in \c $. We have
$b_0 =(t x , t^{-1} y)$ and $b_1 = b$. Thus, every element of $B$ is connected to the subgroup
$ T = \{ (t x , t^{-1} y) \, | \, t \in \c^* \}$. On the other hand, $T$ is path connected, hence
$B$ is path connected as well.

(e) Suppose $B$ is contained in a solvable subgroup $ H \subset G $. Then, $H$ is a solvable group of length at least 3. Then, by Theorem~\ref{TLamy}, it is either of type I or type III.
By Proposition~\ref{typeIII}, it can be only of type I: {i.e.}, it is conjugate to a subgroup of either $A$ or $B$.
Suppose that there is $g \in G$ such that $g^{-1} B g \subseteq
g^{-1} H g \subset A$. Then $B \subset g A g^{-1}$. This implies that $B$ fixes the vertex
$g \cdot A$, which contradicts part (b).
Suppose that there is $g \in G$ such that $g^{-1} B g \subseteq
g^{-1} H g \subset B$. Once again, we can conclude that $B$ fixes
$g \cdot B$ and hence $g \in B$ by (b). It follows that $B  = H $ and
hence $B$ is maximal solvable.
\end{proof}

Now, we can prove Theorem~\ref{G-Borel} from the Introduction.
%
%
\begin{proof}[Proof of Theorem~\ref{G-Borel}]
Let $H$ be a Borel subgroup of $G$. Then, by classification of Theorem~\ref{TLamy},
$ H $ can only be a subgroup of type I. Indeed, it is obvious that $ H $ cannot be
of type IV (since it is solvable); it cannot be of type III (by Proposition \ref{typeIII}),
and it cannot be of type II, since, by \cite[Proposition~3.12]{L},
any type II subgroup of $G$ is given by a countable union of finite cyclic groups
and hence is totally disconnected in the ind-topology of $G$
({\it cf.} \cite[4.1.3(5)]{Ku}).
Thus $ H $ is conjugate to either a subgroup of $A$ or a subgroup of $B$.
In the first case, it must be a Borel subgroup of $A$. Since $A$ is a connected
algebraic subgroup of $ G $, by the classical Borel Theorem,
all Borel subgroups of $A$ are conjugate to each other. Since $U$ is a Borel in $A$,
$H$ must be conjugate to $U$. This obviously contradicts the maximality of $H$
since $U$ is properly contained in $B$. Hence $H$ is conjugate to a subgroup
of $B$; by maximality, it must then be conjugate to $B$.
\end{proof}

The next lemma is elementary: we recall it for reader's convenience
(the proof can be found, for example, in \cite{H}).

\begin{lemma}
\la{prop}
Let $ G $ be an abstract group.
\begin{enumerate}
\item[(a)] If $ G $ has a proper subgroup of finite index then $ G $ has a proper
normal subgroup of finite index.
\item[(b)] If $ G $ has no proper subgroup of finite index then any homomorphic
image of $ G $ has no proper subgroup of finite index.
\item[(c)] If $ G $ is solvable and has no proper finite index subgroup, then it is infinitely generated.
\end{enumerate}
\end{lemma}
%
%
%
%
%

Using Lemma~\ref{prop}, we can now prove
\begin{lemma}
\la{nonfinB}
The group $B$ contains no proper subgroups of finite index.
\end{lemma}
\begin{proof}
Suppose $H$ is a proper finite index subgroup of $B$.  By Lemma~\ref{prop}{\rm (a)},
we may assume that $H$ is normal. Consider the quotient map
$ p_1: B \twoheadrightarrow B/B^{(1)} \cong \c^*$. Then, the image of $ H $ under $ p_1 $
is a finite index subgroup of $\c^*$.  But $\c^*$ has no proper finite index subgroups.
Hence  $ p_1(H) = B/B^{(1)}$ and therefore $ B=B^{(1)}H $.
Now, let $ H_1 := B^{(1)} \cap H$. Then
\begin{equation}
\la{grisom1}
   \frac{B}{H} \, = \,   \frac{B^{(1)}H}{H} \, \cong \, \frac{B^{(1)}}{H_1} \, .
 \end{equation}
This implies that $H_1$ is a finite index subgroup of $B^{(1)}$. Next, we consider
$ p_2: B^{(1)} \onto B^{(1)} /  B^{(2)} \cong \c $. Again, $ p_2(H_1)$ is a finite index subgroup
of $\c $. Since $\c$ has no proper finite index subgroups, we conclude
$p_2(H_1)= B^{(1)} /  B^{(2)} $. Hence $B^{(1)}=B^{(2)}H_1$, and we have
\begin{equation}
\la{grisom2}
 \frac{B^{(1)}}{H_1} \, \cong  \,    \frac{ B^{(2)}}{B^{(2)} \cap H_1} \, .
 \end{equation}
Thus $B^{(2)} \cap H_1$ is a finite index subgroup $B^{(2)}$. On the other hand $B^{(2)} \cong \c[y]$
which has no proper finite index subgroups. Hence $B^{(2)} \cap H_1 =  B^{(2)}$, which implies
that $B^{(2)}$ is a subgroup of $H_1$. Next, since $B^{(1)}=B^{(2)}H_1$,
we have $B^{(1)} = H_1= B^{(1)} \cap H$. From this last equality we see that $B^{(1)} \subseteq H$. Finally, from  $B=B^{(1)}H$ we get $H=B$. This contradicts the properness of $H$.
\end{proof}

Before characterizing the Borel subgroups of $ G $, we recall a classical
characterization of solvable subgroups of $\GL_n (\c) $ due to A.~I.~Maltsev.
Maltsev's theorem can be viewed as a generalization of the Lie-Kolchin Theorem
({\it cf.} \cite[6.3.1]{Sp}): for its proof we refer to \cite[Theorem~3.1.6]{LR}.
\begin{theorem}[Maltsev]
\la{malcev}
Let $\Gamma $ be any solvable subgroup of $\GL_n (\c) $.
Then $ \Gamma $ has a finite index normal subgroup which is
conjugate to a subgroup of upper triangular matrices.
\end{theorem}

We are now in position to prove Steinberg's Theorem for the group $G$.
\begin{theorem}
\la{abs-char}
Let $H$ be a non-abelian subgroup of $G$. Then $H$ is Borel iff
\begin{enumerate}
\item[(B1)] $H$ is a maximal solvable subgroup of $G\,$,
\item[(B2)] $H$ contains no proper subgroups of finite index.
\end{enumerate}
\end{theorem}

\begin{proof}
$(\Rightarrow)$ Suppose $H$ is Borel. Then, by Theorem \ref{G-Borel}, $\,H$ is conjugate to $B$. Hence, by Proposition \ref{B-msolv} and Lemma \ref{nonfinB}, $H$ satisfies $(B1)$ and $(B2)$ respectively.

$(\Leftarrow)$
Let $H$ be a non-abelian subgroup of $G$ satisfying (B1) and (B2).
Then, by Theorem~\ref{TLamy}, it is either of type I or type III.
By Proposition \ref{typeIII},  it cannot be of type III.
Therefore, it is conjugate to a subgroup of $A$ or $B$. Suppose that $H$ is conjugate to
a subgroup of $A$. The image of composition $ g^{-1} H g \rightarrow A \rightarrow \SL_2 (\c)$
is then a solvable subgroup of $\SL_2 (\c)$. We denote this group by $S$.
By Theorem~\ref{malcev}, $S$ has a finite index normal
subgroup $T$, which is a subgroup of upper triangular matrices in $\SL_2 (\c)$. By Lemma
\ref{prop}(b), the group $S$, being the image of $H$, contains no proper subgroups of finite index. Thus, $\,S=T\,$ and $H$ is conjugate to a subgroup of $U$, which is a proper
solvable subgroup of $B$. This contradicts the assumption that $H$ is a maximal solvable
subgroup of $G$. Hence, $H$ can be only conjugate to a subgroup of $B$. Since $H$ is maximal
solvable, it must be conjugate to $B$ itself.
\end{proof}
Theorem~\ref{abs-char} is the special case of Theorem~\ref{ThB1} corresponding to $ n = 0 $. We now turn to the general case.

\subsection{Borel subgroups of $G_n$}
\la{S7.4}
We begin with some technical lemmas. First, recall that, for any element $ g = (P,Q)
\in G $, we defined its degree in $G$ by
$$
\deg(g):= \max\{\deg(P),\,\deg(Q)\} \ ,
$$
where $ \deg(P)$ and $ \deg(Q) $ are the degrees of $ P = P(x,y) $ and $ Q = Q(x,y) $
in the free algebra $ \c\langle x,y\rangle $ ({\it cf.} Section~\ref{S6.2}). It is easy
to see that $ \deg(g) $ thus defined coincides with the degree of $ g $ viewed as an
automorphism of $ \c^2 $.
\begin{lemma}
\la{B-conj}
Let $H$ be a subgroup of $B$ with the property that for any $\,N > 0\,$, there is $\,h \in H\,$
such that $ \deg(h)> N $. If $\,H \subseteq g^{-1} B \,g \cap B \,$ for some $ g \in G $,
then $\, g^{-1} B\,g \cap B = B\,$ and $\,g\in B$.
\end{lemma}
\begin{proof}
If $g \in B$, then $g^{-1}  B \, g = B $ and therefore $g^{-1} \, B \, g \,\cap\, B = B$.
Assume now that $ g \in G \setminus B $. Then we can
write $ g= w_0 w_1 \, \ldots \, w_l $, where $w_0 \in U$ and $\{w_1, \ldots, w_l\} $ are representatives of some cosets in $ A/U $ or $ B/U $. Without loss of generality,
we may assume that $ w_0 =1 $ and $ w_1 $ is a coset representative from $ A/U $. Then
\begin{equation}
\label{wrd}
g^{-1} B \, g \, \cap \, B =
(w_l^{-1} \ldots w_2^{-1} \, w_1^{-1} \, U \, w_1  \, w_2 \ldots w_l)\, \cap \, B\ ,
\end{equation}
since $g^{-1} (B \backslash U) \, g $ consists of words of length $\,2l+1\,$ and
$g^{-1} (B \backslash U) \, g \cap B = \varnothing$. Let $\deg(g)= n $. Then,
by \cite[Lemma~4.1]{K1}, $\,\deg(g^{-1}) \leq n \,$,  and the degrees of all
elements in \eqref{wrd} are at most $ n^2 $. This contradicts the assumption that \eqref{wrd} contains  $H$ whose elements have arbitrary large degrees.
\end{proof}

Now, for $ g \in G $, we define $ B_g :=  g^{-1} B \, g \, \cap \, G_n $. Clearly,
$B_g$ is a subgroup of $ G_n $ that depends only on the right coset of $ g \in G $
(mod $B$). We write $\, V_n (B) := \{B_g\}_{g \in B}\,$ for the set of all such subgroups
of $ G_n $ and note that $ G_n $ acts on $ V_n(B) $ by conjugation.
\begin{lemma}
\la{mapet}
The assignment $\, g \mapsto B_g \,$ induces a bijection
\begin{equation*}
\eta: \, B\backslash G \, \stackrel{\sim}{\rightarrow} \, V_n (B)\ ,\
\end{equation*}
which is equivariant under the $($right$)$ action of $ G_n $.
\end{lemma}
\begin{proof}
It is clear that the map $\eta$ is well defined and surjective. We need only to prove
that $ \eta $ is injective. Suppose that
$ g_1^{-1} B\, g_1\, \cap\, G_n =  g^{-1}_2 B \,g_2 \, \cap \, G_n \,$ for some
$\,g_1,\, g_2, \, \in \, G$. Then
$$
g_2 \,  g_1^{-1} \, B \, g_1 \, g_2^{-1} \cap\,  g_2 \, G_n \, g_2^{-1} =
B \, \cap\, g_2 \, G_n \, g_2^{-1}\ ,
$$
which implies $ B \, \cap\, g_2 \, G_n \, g_2^{-1} \subseteq
g_2\,  g_1^{-1} \, B \, g_1 \, g_2^{-1} \cap \,B \, $. Now, observe that $ B \, \cap\, g_2 \, G_n \, g_2^{-1} = \Stab_B[g_2\cdot (X_0, Y_0)]\,$. Hence $ H := B \, \cap\, g_2 \, G_n \, g_2^{-1} $ satisfies the assumptions of Lemma~\ref{B-conj}, and we conclude:
$ g_2\,  g_1^{-1} \, B \, g_1 \, g_2^{-1} \, = \, B $ and $g_1 \, g_2^{-1} \in B$.
It follows that $ B g_1 = B g_2 $. To see the equivariance of $\eta $, for $ h \in G_n $,
we compute
$$
B_{gh} := (gh)^{-1} B \, (gh)\, \cap \, G_n = h^{-1} (g^{-1} B \, g \,\cap\, G_n)\,h
= h^{-1} B_g \, h \ .
$$
\end{proof}
Dividing the map $ \eta $ of Lemma~\ref{mapet} by the action of $ G_n $, we get
\begin{equation}
\la{etab}
\CC_n/B \, \stackrel{\sim}{\rightarrow} \, V_n(B)/{\rm Ad}\,G_n \ ,
\end{equation}
where we have identified $\, B \backslash G/G_n \,=\, \CC_n/B \,$ via
$\, B \,g\, G_n \leftrightarrow  B\,g(X_0, Y_0)\,$.

Notice that $ V_n(B) $ is the set of $B$-vertex groups
of the graph $ \Gamma_n $ constructed in Section~\ref{S4.2}.
The next lemma gives a simple description of all vertex groups
of $ \Gamma_n$.
\begin{lemma}
\la{vert}
If $\, n \ge 1 $, then, for any $ g \in G $, there is
\begin{enumerate}
\item $\, \tilde g \in A g \,$ such that $\, g^{-1} \, A \, g\, \cap\, G_n = \tilde g^{-1} \, \SL_2 (\c) \, \tilde g \, \cap \, G_n\,$,

\item $\, \tilde g \in B g \,$ such that $\, g^{-1} \, B \, g\, \cap \, G_n  = \tilde g^{-1} \, (T \ltimes G_y) \, \tilde g\, \cap\, G_n\,$.
\end{enumerate}
In particular, every $B$-vertex group of $ \Gamma_n $ is a solvable subgroup of $ G_n $ of derived length $ \leq $ 2.
\end{lemma}
\begin{proof}
This follows from the fact that each $A$- and $B$-orbit in $\CC_n$ contains a point $(X,Y)$ with $\,\Tr(X)=\Tr(Y)=0\,$. Indeed, both
$A$ and $ B $ contain translations, so we can move  $ (X,Y) $ to
$ (X - \frac{1}{n}\,\Tr(X)\,I, \ Y -\frac{1}{n}\,\Tr(Y)\,I) $ along the orbits.
\end{proof}

\begin{proposition}
\la{B_n-msolv}
Let $ B_g  \in V_n(B) $. Then
\begin{enumerate}
\item[(a)] $B_{g}$ is a solvable group of derived length $ \leq 2\,$.
\item[(b)] $N_{G_n} (B_{g}) = B_{g}\,$.
\item[(c)] $B_{g}$ is a maximal solvable subgroup of $G_n\,$.
\end{enumerate}
\end{proposition}
\begin{proof}
(a) By Lemma~\ref{vert}, $\, B_g$ is isomorphic to a subgroup of  $T \ltimes G_y$. Since
$T \ltimes G_y$ is solvable of derived length 2,  $\, B_g$ is solvable of derived length at most 2.

(b) follows from Lemma~\ref{mapet} and the (obvious) fact that
$ \Stab_{G_n} (Bg) = B_g\,$.

(c)   Let $ H $ be a solvable subgroup of $ G_n $ containing $B_g$. Since $H$ is uncountable, by Theorem \ref{TLamy}, it can only be of type I: i.e, conjugate either to a subgroup of  $A$ or a subgroup of $B$.
In the the first case, $B_g \subseteq H \subseteq h^{-1}  A \, h\, \cap \, G_n$ for some $h \in G_n$. This is impossible, since $B_g$ contains elements of arbitarary large degree. Hence, $H$ can only be conjugate to a subgroup of $B$, i.e. $B_g \subseteq H \subseteq h^{-1} B\,h\,  \cap\, G_{n} $, and therefore $\,h\, B_g\, h^{-1} \subseteq (gh^{-1})^{-1} B \, gh^{-1}\, \cap \, B\,$. By Lemma~\ref{B-conj}, this implies $\,gh^{-1} \in B \,$, whence the equality $ B_g = B_h $.
\end{proof}

\begin{theorem}
\la{G_n-borel}
Any Borel subgroup of $G_n$ equals $B_g$ for some $ g \in G $.
\end{theorem}
\begin{proof}
Suppose $H$ is a Borel subgroup of $G_n$. Then $ H $ is a connected solvable subgroup of $G$,
and hence, by Theorem~\ref{TLamy}, it must be of type I or type III. By Proposition \ref{typeIII}, it can only be of type I: i.e. conjugate to a subgroup of $A$ or $B$. In the first
case, by Lemma~\ref{vert}$(1)$, it can be conjugated to a subgroup of $\SL_2 (\c)$.
In fact, since $H$ is connected and solvable, it can be conjugated
to a subgroup of upper
triangular matrices: $U_0 = U \,\cap \, \SL_2 (\c) \subseteq U $. In particular, there is $g \in G$ such that
$H \subseteq g^{-1} U \, g  \,\cap\, G_n$ which is always a proper subgroup of $g^{-1}  B \, g\,  \cap\, G_n$
This contradicts the maximality of $H$. Thus $H$ can only be conjugated to a subgroup of $B$, i.e.
there is $ g \in G$ such that  $H \subseteq g^{-1} B \, g \,\cap\, G_n$. Since $H$ is a maximal solvable subgroup of $G_n$, we must have $H = g^{-1} B \, g \,\cap \,G_n = B_g $.
\end{proof}
As a consequence of Theorem~\ref{G_n-borel} and Proposition~\ref{B_n-msolv}(b),
we get the following infinite-dimensional generalization of a well-known theorem
of Borel \cite{Bo1}.
\begin{corollary}
\la{Borth}
Any Borel subgroup of $ G_n $ equals its normalizer.
\end{corollary}

Now, let $\mathfrak{B}_n$ denote the set of all Borel subgroups of $G_n$.
By Theorem~\ref{G_n-borel}, we have a natural inclusion
$\, \iota: \,  \mathfrak{B}_n  \into  V_n(B)\,$,
which is obviously equivariant with respect to the adjoint action of $ G_n$.
Taking quotients by this action and combining the induced map of $ \iota $
with the inverse of \eqref{etab}, we get
\begin{equation}
\la{BAd}
\mathfrak{B}_n/{\rm Ad}\,G_n \into \CC_n /B\ ,
\end{equation}
which is precisely the embedding \eqref{inB} mentioned in the Introduction.
Our aim now is to prove Theorem~\ref{ThB0}. We begin by recalling the following
important fact proved by G.~Wilson in \cite[Sect.~6]{W}.
\begin{theorem}[\cite{W}]
\la{lemmaW0}
For each $ n $, the variety $ \CC_n $ has exactly $ p(n) $ torus-fixed
points $ (X,Y) $ which are in bijection with the partitions of $n$. These
points  are characterized by the property that both
$X$ and $Y$ are nilpotent matrices.
\end{theorem}
We will refer to points $ (X,Y) \in \CC_n $, with $X$ and $Y$ being nilpotent matrices,
as `nilpotent points.' The next observation is an easy consequence of Theorem~\ref{lemmaW0}.

\begin{corollary}
\la{lemmaW}
Let $\Gamma $ be a subgroup of $ T $ containing a cyclic group of order
$\, > n \,$ $($possibly infinite$)$. If a point $ (X,Y) \in \CC_n $ is fixed by  $\Gamma $
then it is also fixed by $T$.
\end{corollary}
\begin{proof} If $ (X,Y)$ is fixed by $\Gamma$, then $ \Tr(X^k) $ and $ \Tr(Y^k)$ vanish for all
$ k \leq n $. Hence $ \Tr(X^k) = \Tr(Y^k) = 0 $ for all $k > 0 $. This means that
$X$ and $Y$ are both nilpotent matrices, and the claim follows from Theorem~\ref{lemmaW0}.
\end{proof}

Next, for each $(X,Y) \in \CC_n $, we define the following canonical map
\begin{equation}
\la{charmap}
\chi_{(X,Y)}:\  \Stab_B (X,Y) \into  B \onto B/[B,B]\ .
\end{equation}
Note that the image of \eqref{charmap} depends only on the $B$-orbit of $ (X,Y) $
in $ \CC_n $ (not on the specific representative). The target of \eqref{charmap} plays the role
of an `abstract' Cartan subgroup of $ G $, which (just as in the finite-dimensional case,
{\it cf.} \cite[Sect.~3.1]{CG}) can be identified with a maximal torus:
\begin{equation}
\la{projgr}
B/[B,B] \cong T\ ,\quad
[(tx +p(y), \, t^{-1} y+f)] \, \leftrightarrow (tx, \, t^{-1} y)\ .
\end{equation}
In terms of \eqref{charmap}, we can give the following useful characterization of $B$-orbits
with $T$-fixed points.
\begin{lemma}
\la{nilpB}
A $B$-orbit of $ (X,Y) \in \CC_n $ contains a $T$-fixed point if and only if the map
$ \chi_{(X,Y)} $ is surjective. For this, it suffices that the image of
$ \chi_{(X,Y)} $ contains an element of order $\, > n\,$.
\end{lemma}
\begin{proof}
First, in view of \eqref{projgr}, it is obvious that $ \chi_{(X,Y)} $ is surjective
if $ T \subseteq \Stab_B (X,Y) $. For the converse, we will prove
the existence of a $T$-fixed point under the assumption that $ \chi_{(X,Y)} $
contains an element of order $\, > n\,$.
By this assumption, there is an element $ h = (t x + p(y), t^{-1}y+f) \in \Stab_B \,(X,Y)$
such that $ t $ has order $ \geq n+1 \,$ in $ \c^* $. By the Cayley-Hamilton Theorem,
we may assume that $\, \deg\,p(y) \leq n-1 \,$. Applying the automorphism
$ b_1 := (x-\frac{1}{n}\Tr(X), y-\frac{1}{n}\Tr(Y)) \in B $ to $ (X,Y) $, we get
a point $ (X_1 , Y_1) $ with $\Tr(X_1) = \Tr(Y_1)=0$. Hence $ b_1 \, [\Stab_B \, (X,Y)] \, b^{-1}_1  \subseteq T \ltimes G_y$ and $\,h_1 = b_1 \, h \, b^{-1}_1 = (t x + p_1(y), t^{-1}y)\,$ with
$\, \deg\,p_1(y) \leq n-1 \,$. We now show that $ h_1 $ can be conjugated to
$(t x , t^{-1}y)$. Indeed, write $ p_1(y)= \sum^{n-1}_{i=0} a_i y^i $ and conjugate
$$
b_2 \,h_1 \, b_2^{-1} = (t x + t q(y) - q(t^{-1}y) + p_1(y), \,t ^{-1} y)\ ,
$$
where $ b_2 := (x+q(y),y) $ with $ q(y) = \sum^{n-1}_{i=0} c_i y^i  \in \c[y]\,$. Setting
$$
t q(y)-q(t ^{-1}y) + p_1(y) = 0 \ ,
$$
we get a linear system for the coefficients of $ q(y) $ of the form
$$
c_i (t - t^{-i}) = a_i \quad  i=0, \ldots ,  n-1\ .
$$
Hence, if we take $ c_i = a_i\,(t - t^{-i})^{-1} $ for $ b_2 $ and set
$ b := b_2 b_1 $, then $ b \, h \, b^{-1} = (t x , t^{-1}y)$. Thus $
b\, [\Stab_B (X,Y)] \, b^{-1} $ contains $(t x , t^{-1}y)$. By Corollary~\ref{lemmaW},
we now conclude that $ b\cdot (X,Y)$ is a nilpotent point, and hence,
by Theorem~\ref{lemmaW0}, it is $T$-fixed.
\end{proof}

Now, for $(X,Y) \in \CC_n $, let $ G_y(X,Y) $ denote the stabilizer of $(X,Y) $ in $ G_y $.
Note that $ G_y(X,Y) \subseteq \Stab_B(X,Y) $ for any $ (X,Y) $, since $ G_y \subset B $.
The next lemma is a direct consequence of Proposition~\ref{lstab}, which is proved in
Section~\ref{Sadelic}; it shows that all groups $ G_y(X,Y) $ are path connected
(and hence connected).
\begin{lemma}
\la{lstab1} For any $(X,Y) \in \CC_n $, if $\,(x + q(y), y) \in G_y(X,Y) $, then
$\,(x + \lambda q(y), y) \in G_y(X,Y) $ for any $ \lambda \in \c $.
\end{lemma}

We now give a classification of $B$-orbits in $ \CC_n $ and their isotropy groups.
\begin{proposition}
\la{classBstab}
For a $B$-orbit $\mathcal{O}_B $ in $ \CC_n $, one and only one of the following
possibilities occurs:
\begin{enumerate}
\item[(A)] $\,T$ acts freely on $ \mathcal{O}_B $,  $\, \Stab_B (X,Y) = G_y (X,Y)\,$
and the map $\chi_{(X,Y)}$ is trivial $($i.e., its image is $1$$)$ for every $ (X,Y) \in \O_B $.

\item[(B)] $ \O_B $ contains a $T$-fixed $\,(X,Y)\,$, $\, \Stab_B (X,Y) =
T \ltimes G_y (X,Y)$,  and the map of $\chi_{(X,Y)}$ is surjective.

\item[(C)] $ \O_B $ contains a point $\,(X,Y)\,$ such that
$ \Stab_B (X,Y) = \mathbb{Z}_k \ltimes G_y (X,Y)$ for some $0 < k \leq n$, and
the image of $\chi_{(X,Y)}$ is isomorphic to $\mathbb{Z}_k$.
\end{enumerate}
\end{proposition}
\begin{proof}
Let $ \O_B $ be a fixed $B$-orbit. For any $ (X,Y) \in \O_B $, the character map
\eqref{charmap} combined with \eqref{projgr} gives the short exact sequence
\begin{equation}
\la{seq-stab}
1 \, \rightarrow \, G_y (X, Y) \, \rightarrow \, \Stab_B (X,Y) \, \rightarrow \, K \,
\rightarrow \, 1
\end{equation}
where $K$ is the image of $\chi_{(X, Y)}$ in $T$.

If $ K = 1 $ for some point in $ \O_B $, then $ K = 1 $ for all
$ (X,Y) \in \O_B $ and hence $ \Stab_B (X,Y) = G_y (X,Y) $ for all
$ (X,Y) \in \O_B $, which means that $T$ acts freely on $\mathcal{O}_B$.
This is case (A).

If $K$ contains an element of order $ \geq n+1 $ (possibly $\infty$)
for some point in $ \O_B $, then, by  Lemma~\ref{nilpB}, $ \O_B$
contains a $T$-fixed point $ (X,Y) $ and $ K = T $. Then $\Stab_B (X,Y)$
contains $T$, the above short exact sequence splits, and we have
$ \Stab_B (X,Y) = T \rtimes G_y (X,Y)$. This is case (B).

Finally, assume that neither (A) nor (B) holds. Then, by Lemma~\ref{vert},
there is still a point $ (X,Y) \in \mathcal{O}_B $ such that
$\Stab_B(X,Y) \subseteq T \ltimes G_y$. By our assumption, the corresponding
$ K \subset T $ must be a cyclic group of order $k$ for $ 0 < k \leq n$.
Let $(\lambda , \lambda^{-1}) \in K $ be the generator of $K$. Write
$\phi = (\lambda x + p(y), \lambda^{-1} y)$ for the preimage of
$ (\lambda , \lambda^{-1}) $ in $\Stab_B (X,Y)$. Iterating $ \phi\, $, we get
\begin{equation*}
\la{phik}
\phi^k = (x+\, \sum^k_{j=1} \, \lambda^{k-j} p(\lambda^{1-j} y), y)
\end{equation*}
Explicitly, if $ p(y)= \sum^m_{i=0} \, a_i y^i\,$, then the coefficient under $y^i$
in the first component of $ \phi^k $ is equal to
$$
a_i \, \sum^{k}_{j=1} \,\lambda^{-k-j+(1-j)i} = a_i \lambda^i \,
\sum^k_{j=1} \, \lambda^{-j(i+1)}
$$
Since $\sum^k_{j=1} \, \lambda^{k-j}=0$, all these coefficients vanish except
those with $\, i \equiv -1 \,({\rm mod} \, k)\,$. Thus $\, \phi^k = (x+ k \, p_1(y) ,y)\,$,
where $\, p_1(y) = a_{k-1} y^{k-1} + a_{2k-1} y^{2k-1} +\ldots\,$ is a polynomial  obtained
from $p(y)$ by removing all coefficients except those with $i\equiv -1 \,({\rm mod} \, k)\,$.
Since $ \phi^k \in \Stab_B(X, Y)$, by Lemma~\ref{lstab1},$\ (x - \lambda^{-1} p_1(y) ,y) \in \Stab_B (X,Y)$. Hence
$ \phi_1 := (\lambda x + p(y)- p_1(y), \lambda^{-1} y) \in \Stab_B (X, Y)$ and
$ \phi_1^k = 1 $. Now, the mapping $ (\lambda , \lambda^{-1}) \mapsto \phi_1 $
splits \eqref{seq-stab}. Hence $ \Stab_B (X, Y) = \Z_k \ltimes G_y (X, Y)$, where
$ \Z_k $ is generated by $ \phi_1 $. This is case (C).
\end{proof}
We are now ready to prove Theorem~\ref{ThB0} from the Introduction.
\begin{proof}[Proof of Theorem \ref{ThB0}]
By Theorem~\ref{G_n-borel}, any Borel subgroup of $ G_n $ has the form
$ B_g := g^{-1} B \, g \, \cap \, G_n $, while
$\, B_g = g^{-1} [\Stab_B \,(X,Y)]\, g \,$, where $ (X,Y)= g \cdot (X_0,Y_0) \in \CC_n $.
Now, by classification of Proposition~\ref{classBstab}, the group $ \Stab_B(X,Y) $ is
connected if and only if the corresponding $B$-orbit is of type (A) or type (B).
Indeed, in case (A), we have $ \Stab_B (X,Y)= G_y(X,Y) $. Hence, by Lemma~\ref{lstab1},
$ \Stab_B (X,Y) $  is path connected and therefore connected. Note also that
$ \Stab_B (X,Y) $ is abelian, since so is $ G_y(X,Y) $.

In case (B), we may assume that $ \Stab_B (X,Y) = T \ltimes G_y (X,Y) $. Then
any element of $ \Stab_B(X,Y) $ can be written in the form $ b = (ax + q(y), a^{-1}y) $,
where $ q(y) \in \c[y] $. By Lemma~\ref{lstab1}, if $ b \in \Stab_B (X,Y) $ then
$ b_{t} := (a x + t \,q(y), a^{-1} y) \in \Stab_B (X,Y) $ for all $ t \in \c $, hence
we can join $ b = b_1 $ to $ b_0 = (ax, a^{-1}y) \in T $ within $ \Stab_B (X,Y) $. It follows
that $ \Stab_B (X,Y) $ is connected since so is $ T $.
Note that in this case, $ \Stab_B (X,Y) $ is a solvable but non-abelian subgroup of $G$.

In case (C), the group $ \Stab_B (X,Y) $ is obviously disconnected. Hence
the corresponding $ B_g $ cannot be a Borel subgroup of $ G_n $.
\end{proof}

\subsection{Conjugacy classes of non-abelian Borel subgroups}
\la{S7.4.1}
Following \cite{W}, we denote the $T$-fixed points
of $ \CC_n $ by $ (X_{\mu}, Y_{\mu}) $, where
${\mu} = (n_1, n_2, \ldots, n_k )$ is a partition of $n$ with $n_1 \leq n_2 \leq \ldots \leq n_k$.  We consider the $B$-orbits of these points in $ \CC_n $ as vertices of the graph $ \Gamma_n $ defined in Section~\ref{S4.2}.
For a fixed collection of elements $ g_{\mu} \in G $ such that $ g_{\mu} (X_0,Y_0) = (X_{\mu},Y_{\mu}) $,
we define the subgroups $ B_{\mu} \subset G_n $ by
\begin{equation}
\la{Bmu}
B_{\mu} \, := \, g_{\mu} ^{-1}\, B \, g_{\mu} \, \cap \, G_n \, .
\end{equation}
These are $B$-vertex groups attached to the $B$-orbits $ B(X_{\mu}, Y_{\mu}) $ in $ \Gamma_n $.
Geometrically, $B_{\mu} $ are the conjugates of subgroups of
$B$ fixing the points $ (X_{\mu},Y_{\mu})$ in $\CC_n$. More explicitly
$ B_{\mu} = \, g_{\mu}^{-1}\, B({\mu}) \, g_{\mu}$, where
$ B({\mu}) := \Stab_B \,  (X_{\mu},Y_{\mu}) \,$.

As an immediate consequence of Theorem \ref{ThB0}, we have
\begin{corollary}
$B_{\mu}$ is a Borel subgroup of $G_n$.
\end{corollary}
Next, we prove
\begin{theorem}
\la{nonab-G_n-borel}
Any non-abelian Borel subgroup of $G_n$ is conjugate to some $B_{\mu}$.
\end{theorem}
\begin{proof}
Suppose $H$ is a non-abelian Borel subgroup of $G_n$. By Theorem~\ref{G_n-borel}, any Borel group is equal to $H=B_g$ for some $g \in G$. Then, by Theorem \ref{ThB0}, $H$ is Borel if
either (A) $T$ acts freely on corresponding $B$-orbit or (B) $T$ has a fixed point on the
corresponding $B$-orbit. In the first case, $H$ must be abelian, which contradicts our assumption. In the second case, $ H $ is conjugate to  $ T \ltimes G_y (X,Y) \,$,  where $(X,Y)$ is a nilpotent point.  Hence $H$ is conjugate to $ B_{\mu} $ for some $ \mu $.
\end{proof}
\begin{lemma}
\la{nonfinBni}
$B_{\mu}$ contains no proper subgroup of finite index.
\end{lemma}
\begin{proof}
Similar to the proof of Lemma~\ref{nonfinB}.
\end{proof}
Now we are ready to prove Steinberg's Theorem in full generality.
\begin{proof}[Proof of Theorem~\ref{ThB1}]
$(\Rightarrow)$  Let $H$ be a Borel subgroup of $G_{n}$.
Then, by Theorem \ref{nonab-G_n-borel}, $H$ is conjugate to $B_{\mu}$.
Hence, by Proposition \ref{B_n-msolv}(c) and Lemma \ref{nonfinBni}, $H$ satisfies
properties (B1) and (B2) respectively.

$(\Leftarrow)$ Let $H$ be a subgroup of $G_n$ satisfying (B1) and (B2).
By Theorem~\ref{TLamy}, it is then either of type I or type III.
By Proposition~\ref{typeIII},  it cannot be of type III.
Therefore, it is conjugate to either a subgroup of $A$ or a subgroup of $B$. Suppose that it is conjugate to a subgroup of $A$. The image of
$g^{-1} H g \rightarrow A \rightarrow \SL_2 (\c)$ is then a solvable subgroup of $\SL_2 (\c)$. We denote this group by $S$. By Theorem \ref{malcev}, $S$ has a finite index normal
subgroup $T$, which is a subgroup of upper triangular matrices in $\SL_2 (\c)$. By Lemma \ref{prop}(b), the group $S$, being a homomorphic image of $H$, contains no proper subgroup of finite index. Thus $S=T$ and $H$ is conjugate to a subgroup of upper
triangular matrices: $U_0 = U \cap \SL_2 (\c) \subseteq U$. In particular, there is $g \in G$ such that $H \subseteq g^{-1} \, U \, g  \cap G_n$ which is always a proper subgroup of $g^{-1} \, B \, g  \cap G_n$. This contradicts property (B1). Hence $H$ can only be conjugate to a subgroup of $B$. Thus $ H \leq  g^{-1} \, B \, g \cap G_{n}$ for some $g \in G$. Since $H$ is maximal solvable,
we have $H =  g^{-1} \, B \, g \cap G_{n}$, thus $H=g^{-1} \, \Stab_B (X,Y) \, g$. Since $H$ is non-abelian, by Proposition \ref{classBstab}, the group $ \Stab_B (X,Y)$ is either (1) $T \ltimes G_y (X,Y)$ or (2) $\mathbb{Z}_k \ltimes G_y (X,Y)$. By assumption (B2), $H$ does not contain a subgroup of finite index, hence (2) is impossible. Therefore
we must have $\Stab_B (X,Y)= T \ltimes G_y (X,Y)$ and hence $H$ is conjugate to some $B_{\mu}$.
\end{proof}

We will prove that the subgroups $ B_{\mu} $ are pairwise non-conjugate in $ G_n$. We begin with the following lemma, the proof of which is essentially contained in \cite{W}. For reader's
convenience, we provide full details.
\begin{lemma}
\la{orb}
The nilpotent points $ (X_{\mu}, Y_{\mu}) $ in $\CC_n$
belong to distinct $B$-orbits.
\end{lemma}
\begin{proof}
Consider the subgroup $ B_0 $ consisting of the automorphisms $(x + p(y),y) \in G  $
with $p(0)=0$. It is easy to see that any two nilpotent points are in the same $B$-orbit
iff they are in the same $B_0$-orbit. Indeed, $ T $ fixes each of the nilpotent points,
hence does not contribute to the $B$-orbit. On the other hand, applying an automorphism with nonzero constant terms to a nilpotent point moves it to a point with a nonzero trace, which is not nilpotent.
Therefore we will only consider orbits of $B_0$.
By \cite[Proposition 6.11]{W}, the points $ (X_{\mu}, Y_{\mu}) $ are exactly the centers of distinct $n$-dimensional cells in $\CC_n$ which have pairwise empty intersection.
Now, if we show that these cells contain the $B_0$-orbits of $(X_{\mu}, Y_{\mu})$, the result will follow. We start by looking at the simplest case the point corresponding with partition: $\mu=\mu(n,r)$ where  $\mu(n,r)=(1,\ldots ,1, n-r+1)$. In this case $(X_{\mu}, Y_{\mu})$ is given by
\begin{equation*}
\label{base1}
X_{\mu} = \begin{pmatrix}
0& 0 &0& \ldots &0 \\*[1ex]
a_1&0& 0& \ldots &0\\
0&a_2&0& \ddots & \vdots\\
\vdots& \vdots & \ddots& \ddots & 0\\
0&0& \ldots & a_{n-1} &0
\end{pmatrix}\quad , \quad
Y_{\mu} = \begin{pmatrix}
0&1&0& \ldots &0 \\*[1ex]
0&0&1& \ldots &0\\
0&0&0& \ddots & \vdots\\
\vdots& \vdots & \ddots& \ddots & 1\\
0&0& \ldots & 0 &0
\end{pmatrix}
\end{equation*}
where $ (a_1 , \ldots , a_{n-1}) = (1,2, \, \ldots \, , r-1; \, -(n-r), \, \ldots \, , -2,-1)$.
Then, the $B_0$-orbit of $(X_{\mu},Y_{\mu})$ consists of the points $(X,Y_{\mu})$, where
$ X=X_{\mu}+\sum^{n-1}_{k=1} X^{(k)}$ with matrices $X^{(k)}$  having nonzero terms only
on the $k$-th diagonal. Applying a transformation $Q_{t}$, which is essentially a scaling transformation followed by
conjugation by $\diag (1, t, \ldots, t^{n-1})$ (see \cite[(6.5)]{W}), we obtain
\begin{equation}\label{eq2}
Q_{t} (X,Y_{\mu}) = (X_{\mu}+\sum^{n-1}_{k=1} t^{-k-1} X^{(k)},Y_{\mu})
\end{equation}
As $t \rightarrow \infty $, we see that $Q_{t} ( X,Y_{\mu}) \rightarrow (X_{\mu},Y_{\mu})$, hence
$( X,Y_{\mu})$ is still in a cell with the center $(X_{\mu},Y_{\mu})$.

More generally, consider the partition $\mu = \mu(n_1,r_1, \ldots n_k,r_k)$ which corresponds to the Young diagram
with one-hook partitions $(1,\ldots ,1, n_i-r_i+1)$ placed inside each other; such that
 neither the arm nor the leg of any hook is allowed to poke out beyond the preceding one.
In this case, $\, Y_{\mu} = \oplus^k_{i=1} J(n_i)\,$ as a sum of several nilpotent Jordan blocks of dimensions $n_k$;
and $X_{\mu}$ is a block matrix consisting of the
diagonal blocks $X_{ii} = X_{(1,\ldots , 1, n_i-r_i+1)}$ described as in the previous paragraph and certain
(unique) matrices $X_{ij}$ with non-zero entries only on the
$(r_j- r_i-1)$-th diagonal. The $B_0$-orbit of $(X_{\mu},Y_{\mu})$ then consists of the points
$(\tilde X ,Y_{\mu})$, where $ \tilde X_{ij} = X_{ij}$ for $i\neq j$
and $\tilde X_{ii} =  X_{ii}+\sum^{n_i-1}_{k=1} X^{(i,k)}$ is the sum of matrices $X^{(i,k)}$ with only nonzero terms on the $k$th diagonal of the corresponding block matrix.
Once again, looking at $Q_{t} (\tilde X,Y_{\mu}) $
one can easily show that the diagonal blocks $\tilde X_{ii}$ flow to $ X_{ii}$ as $t \rightarrow \infty$. On the other hand, the only non-zero diagonal of $\tilde X_{ij}$ is
the $(r_j-r_i -1)$-th diagonal, counting within the $(i,j)$-block; or, if we count diagonals inside the big matrix $\tilde X$, it is the one with number  $q_j-q_i -1$, where
$$
q_i := n_1 +  \ldots  + n_{i-1} + r_i\ .
$$
Thus, the map $Q_{t}$ multiplies the non-zero diagonal of $\tilde X_{ij}$ by $t^{q_i-q_j}$. If we now conjugate  by the block-scalar matrix $\,\oplus\, t^{-q_i} I_{n_i}$ , then the
$(i,j)$-block gets multiplies by $t^{q_j-q_i}$, so we get $X_{ij}$.
Thus, summing up, we obtain that $Q_{t} (\tilde X,Y_{\mu}) \rightarrow (X_{\mu},Y_{\mu})$ as $t \rightarrow \infty$, hence the corresponding $B_0$-orbit is in the cell.
\end{proof}
\begin{theorem}
\la{fund}
The subgroups $B_{\mu}$ are pairwise non-conjugate in $G_n$, i.e. there is no $g \in G_n$
such that $\, g^{-1} B_{\mu} g = B_{\lambda} \,$ unless $\mu = \lambda$.
\end{theorem}
\begin{proof}
This is a consequence of Lemma~\ref{mapet} (see \eqref{etab}) and
Lemma~\ref{orb}.
\end{proof}

Now, we can prove Theorem~\ref{ThB2} and Corollary~\ref{niscor} stated in
the Introduction.
\begin{proof}[Proof of Theorem~\ref{ThB2}]
Combine Theorem~\ref{nonab-G_n-borel} and Theorem~\ref{fund}.
\end{proof}
\begin{proof}[Proof of Corollary~\ref{niscor}]
Suppose that there exists an (abstract) group isomorphism $\, G_k \cong G_n $ for some
$k$ and $n$. Then,
by Theorem~\ref{ThB1}, it must induce a bijection between the sets of conjugacy classes
of non-abelian Borel subgroups in $ G_k $ and $ G_n $. By Theorem~\ref{ThB2}, these sets
are finite sets consisting of $ p(k) $ and $ p(n) $ elements. Hence $\,p(k)=p(n)\,$
and therefore $k=n$.
\end{proof}
\subsection{Adelic construction of Borel subgroups}
\la{Sadelic}
We conclude this section by giving an explicit description of the special subgroups $ B(\mu)$.
To this end we will use an infinite-dimensional {\it adelic Grassmannian} $ \Gr $
introduced in \cite{W1}. We recall that $ \Gr $ is the space parametrizing all primary decomposable
subspaces of $ \c[z] $ modulo rational equivalence. To be precise, a subspace
$ W \subseteq \c[z] $ is called primary decomposable if there is a finite collection
of points $\,\{\lambda_1,
\lambda_2, \ldots, \lambda_N\}
\subset \c\,$ such that $\, W = \bigcap_{i=1}^N W_{\lambda_i} \,$, where
$ W_{\lambda} $ is a $\lambda $-primary (i.e., containing a power
of the maximal ideal $ \m_\lambda $) subspace of $ \c[z] $.
Two such subspaces, say $W$ and $W'$, are (rationally) equivalent if $ p W = q W' $
for some polynomials $p$ and $q$. Every equivalence class $ [W] \in \Gr $ contains a unique {\it irreducible} subspace, which is characterized by the property that it is not contained in a proper ideal of $ \c[z] $. We may therefore identify $ \Gr $ with the set of irreducible primary decomposable subspaces in $ \c[z] $.

Now, by \cite{W} and \cite{BW}, there is a natural bijection $\,\beta: \bigsqcup_{n \geq 0} {\mathcal C}_n  \stackrel{\sim}{\to} \Gr \,$, which is equivariant under $G$. It is not easy to
describe the action of the full group $ G $ on $ \Gr $; however, for our purposes, it will suffice to know the action of its subgroup $ G_y $, which is not difficult to describe. We will use
the construction of the action of $ G_y $ on $ \Gr $ given in \cite{BW} (where $ G_y $ is denoted by $ \Gamma $).

Let $ {\mathcal H} $ denote the space
of entire analytic functions on $ \c $ equipped with its usual topology (uniform convergence on compact subsets). Given a subspace $ W \subseteq  \c[z] $ we write
$ \overline{W} \subseteq {\mathcal H} $ for its completion in $ {\mathcal H} $, and
conversely, given a closed subspace $ {\mathcal W} \subseteq {\mathcal H} $ we set
$\,{\mathcal W}^{\rm alg} := {\mathcal W}\,\cap\,\c[z]\,$. Then, for any $ q \in \c[z] $, we define
$$
e^q \cdot W := (e^{q} \,\overline{W})^{\rm alg}
$$
The action of $ G_y $ under $ \beta $ transfers to $ \Gr $ as follows (see \cite[Sect.~10]{BW}): if $ W = \beta(X,Y) \in \Gr $ then
$$
e^q \cdot W = \beta(X + q'(Y), Y) \ ,\quad \forall\, q \in \c[z]\ .
$$
Now, for any $ W \in \Gr $, put
\begin{equation*}
\la{aw}
A_W := \{q \in \c[z] \,:\, q W \subseteq W \} \,.
\end{equation*}
Clearly $A_W $ is a commutative algebra, $W$ being a finite module over $A_W$.
Geometrically, $ A_W $ is the coordinate ring of a rational curve
$\, X = \Spec(A_W) \,$, on which $\, W \,$ defines a (maximal) rank 1 torsion-free
coherent sheaf $ \L $. The inclusion $\, A_W \into \c[z] \,$ gives
normalization $\, \pi : \A^1 \to X \,$ (which is set-theoretically a bijective map).
In this way, $\, \Gr $ parametrizes the isomorphism classes of triples
$\, (\pi, X, \L) \,$ (see \cite{W1}).

\begin{proposition}
\la{lstab}
For any $ W \in \Gr $,
$\,\Stab_{G_y}(W) = \{(x + q'(y), y) \in G \, : \, q \in A_W\}$.
\end{proposition}
\begin{proof}
By \cite[Lemma~2.1]{BW} and the above discussion, the claim is equivalent to
$$
A_W \,=\,\{q \in \c[z]\ :\ e^{q}\,\overline{W} = \overline{W}\}\ .
$$
The inclusion `$\,\subset\,$' is easy: if $ q \in A_W $ then
$ q^n \,W \subset W $ for all $ n \in \N $, hence
$ e^{q} \, W \subset \overline{W} $ and therefore
$\,e^{q}\,\overline{W} = \overline{W}\,$.

To prove the other inclusion it is convenient to use the `dual' description
of $ \Gr $ in terms of algebraic distributions  (see \cite{W1}).
To this end assume that $ W $ is supported on $\,\{\lambda_1, \lambda_2, \ldots, \lambda_N\} \subset \c\,$. Then,
for each $\, \lambda_i \in \supp(W) \,$, there is a finite-dimensional subspace
$ W^*_{\lambda_i} $ of linear functionals on $ {\mathcal H} $
supported at $ \lambda_i $ such that\footnote{Note that the elements of
$ W^*_{\lambda_i} $ can be written as
$\,\varphi_i = \sum_{k} c_{ik} \, \delta^{(k)}(z - \lambda_i)\,$,
where $ \delta^{(k)}(z - \lambda_i) $ are the derivatives of the
$\delta$-function with support at $ \lambda_i $.}
$$
W = \{f \in \c[z]\ :\ \langle \varphi_i,\, f \rangle = 0\ \mbox{for all}\ \varphi_i \in
W^*_{\lambda_i} \ \mbox{and for all}\ i=1,2,\ldots, N\}\ .
$$
By \cite[Lemma~2.1]{BW}, we then also have
$$
\overline{W} = \{f \in {\mathcal H}\ :\ \langle \varphi_i,\, f
\rangle = 0\ \mbox{for all}\ \varphi_i \in
W^*_{\lambda_i}\ \mbox{and for all}\ i=1,2,\ldots, N\}\ .
$$

Now, suppose that $  e^{q}\,\overline{W} = \overline{W} $ for some $ q \in \c[z] $.
Then  $\,  e^{tq}\,\overline{W} = \overline{W} $ for all $ t \in \c $.
Indeed, for fixed $ \varphi_i \in W^*_{\lambda_i} $ and $ f \in \overline{W} $,
the function $\, P(t) := \langle \varphi_i,\, e^{tq} f \rangle \,$
is obviously a quasi-polynomial in $t$ of the form $ P(t) = p(t) \,e^{q(\lambda_i) t} $,
where $ p(t) \in \c[t] $.
Since $\,  e^{q}\,\overline{W} = \overline{W} \,$ implies $\,  e^{k q}\,\overline{W} = \overline{W} \,$ for all $ k \in \Z $, we have $ P(k) = 0 $ and hence $ p(k) = 0 $
for all $ k \in \Z $. This implies $ P(t) \equiv 0 $. In particular, we have
$\, P'(0) = \langle \varphi_i,\,q f \rangle = 0 $. Since this equality holds for all
$ \varphi \in W_{\lambda_i}^* $, for all $i$ and for all $ f \in W $, we conclude
$ q W \subseteq W $. Thus $ q \in A_W $.
\end{proof}

\vspace{1ex}

Now, let $ (X_\mu, Y_\mu) $ be the $T$-fixed point of $ \CC_n $ corresponding to
a partition $ \mu = \{n_1 \leq n_2 \leq \ldots \leq n_k \} $. Then, the corresponding
(irreducible) primary decomposable subspace of $ \Gr $ is given by
$$
W_{\mu} = \Span \{1, \,x^{r_1},\, x^{r_2},\, x^{r_3},\, \ldots \}\ ,
$$
where $\, r_i = i + n_k - n_{k-i}  \,$ (with convention $ n_j = 0 $ for $ j < 0 $).
Write $ R_\mu := \{r_0=1,\, r_1,\, r_2,\,\ldots \} $ for the set of exponents
of monomials occurring in $ W_\mu $, and denote by $ S_\mu := \{k \in \N\,:\,
k + R_\mu  \subset R_\mu\} $ the subsemigroup of $ \N $ preserving $ R_\mu $.
Then $ A_{W_\mu} = \Span\{x^s\,:\, s \in S_\mu\} $, and as a consequence
of Proposition~\ref{lstab}, we get
\begin{corollary}
\la{corbor}
For any partition $ \mu $, $\, B(\mu) = T \ltimes G_{\mu, y}\,$, where
$ G_{\mu, y} $ is the subgroup of $ G_y $ generated by the transformations
$ \{(x + \lambda y^{s-1}, y)\,:\, s \in S_\mu \, ,\, \lambda \in \c \} \,$.
\end{corollary}
To illustrate Corollary~\ref{corbor}, we list below all special Borel
subgroups of $ G_n $ for $ n = 1,\,2,\,3,\,4 \,$.

\subsubsection{Examples}
For $n=1$, there is only one $T$-fixed point $(0,0) \in \mathcal C_1 $ and the corresponding Borel subgroup is
$$B_{(1)} =  T \ltimes \{ \Psi_{cy^k}    \, | \,   c \in \c, k\ge 1\} = \{(ax + cy^k, a^{-1}y) \ | \ a\in \c^*, c \in \c, k\ge 1\}  \, .$$

\vspace{1ex}

For $n=2$, the fixed points are $ (X_{(2)}, Y_{(2)}) $ and
$ (X_{(1,1)}, Y_{(1,1)}) $, where
$$
X_{(2)} :=  \begin{pmatrix}
0 & 0 \\
-1 & 0
\end{pmatrix} \quad , \quad X_{(1,1)} :=  \begin{pmatrix}
0 & 0 \\
1 & 0
\end{pmatrix}\quad , \quad Y_{(2)} =Y_{(1,1)}:=  \begin{pmatrix}
0 & 1 \\
0 & 0
\end{pmatrix}\ .
$$
The corresponding Borel subgroups are given by
\begin{eqnarray*}
B_{(2)} &=&  T \ltimes \{ \Psi_{cy^k}    \, | \,   c \in \c, k\ge 2\}  \, ,\\*[1ex]
B_{(1,1)} &=&  T \ltimes \{ \Phi_{cx^k} \, | \,   c \in \c, k\ge 2\}  \,
\end{eqnarray*}

\vspace{1ex}

For $n=3$, the fixed points are
$$ X_{(3)} =  \begin{pmatrix}
0  & 0   &  0 \\
-2 & 0   & 0  \\
0 &  -1  &  0
\end{pmatrix}\ , \quad
X_{(1,1,1)} = \begin{pmatrix}
0 & 0 & 0 \\
1 & 0 & 0\\
0&  2 & 0
\end{pmatrix}\ , \quad
X_{(1,2)} = \begin{pmatrix}
0  & 0   &  0 \\
1 & 0   & 0  \\
0 &  -1  &  0
\end{pmatrix}\
$$
and
$$
Y_{(3)}=Y_{(1,1,1)}=Y_{(1,2)} =
\begin{pmatrix}
0 & 1 & 0 \\
0 & 0 & 1\\
0&  0 & 0
\end{pmatrix}
$$
The corresponding Borel subgroups are given by
\begin{eqnarray*}
B_{(3)} &=&
T \ltimes \{ \Psi_{cy^k}    \, | \,   c \in \c, k\ge 3\}  \, , \\*[1ex]
B_{(1,1,1)} &=& T \ltimes \{ \Phi_{cx^k} \, | \,   c \in \c, k\ge 3\}\ , \\*[1ex]
B_{(1,2)} &=&  \Psi_{-y^2}\ \Phi_{-\frac{x^2}{2}} \, \Psi_{-2y^2} \, B(1,2)\, \Psi_{2y^2}\,
\Phi_{\frac{x^2}{2}} \, \Psi_{y^2} \ ,
\end{eqnarray*}
where
$$ B(1,2)\, := \, T \ltimes \{ \Psi_{q(y)} \, | \,  q(y) \in  \c y + y^3\c[y] \} $$

\vspace{1ex}

For $n=4$, there are five fixed points:
$$
X_{(4)} =  \begin{pmatrix}
0  & 0   &  0 & 0 \\
-3 & 0   & 0 &0   \\
0 &  -2  &  0 & 0\\
0& 0&-1&0
\end{pmatrix}\ , \quad
X_{(1,3)} = \begin{pmatrix}
0 & 0 & 0& 0 \\
1 & 0 & 0&0\\
0&  -2 & 0&0\\
0&0&-1&0
\end{pmatrix}\ , \quad
X_{(1,1,2)} =  \begin{pmatrix}
0 & 0 & 0& 0 \\
1 & 0 & 0&0\\
0&  2 & 0&0\\
0&0&-1&0
\end{pmatrix}\
$$
$$  X_{(2,2)} =  \begin{pmatrix}
0 & 0 & 0& 0 \\
1 & 0 & 0&0\\
0&  -1 & 0&1\\
-3&  0  & 0&0
\end{pmatrix}\ , \quad
X_{(1,1,1,1)} =  \begin{pmatrix}
0 & 0 & 0& 0 \\
1 & 0 & 0&0\\
0&  2 & 0&0\\
0&0&3&0
\end{pmatrix}\
$$
$$
Y_{(2,2)}\, = \begin{pmatrix}
0 & 1 & 0&0 \\
0 & 0 & 1&0\\
0&  0 & 0&0\\
0&0&0&0
\end{pmatrix}
\ , \quad
Y_{\mu}\, = \begin{pmatrix}
0 & 1 & 0&0 \\
0 & 0 & 1&0\\
0&  0 & 0&1\\
0&0&0&0
\end{pmatrix} \ ,
$$
where $ \mu=\{(4), \, (1,3), \, (1,1,2), \, (1,1,1,1) \} $.
The corresponding Borel subgroups are

\begin{eqnarray*}
B_{(4)} &=&
B(4) \, , \quad B_{(1,3)} =  \Psi_{-y^3} \, \Phi_{\frac{x^3}{6}} \, \Psi_{-3y^3} \, B(1,3) \,
\Psi_ {3y^3} \, \Phi_ {-\frac{x^3}{6}} \, \Psi_{y^3}\ , \\*[1ex]
B_{(1,1,2)} &=& \Psi_{-y^3} \, \Phi_{\frac{x^3}{3}} \, \Psi_{3y^3} \,  B(1,1,2)\, \Psi_ {-3y^3} \, \Phi_ {-\frac{x^3}{3}} \, \Psi_{y^3}\\*[1ex]
B_{(1,1,1,1)} &=& \Psi_{-y^3} \, \Phi_{\frac{x^3}{2}} \, \Psi_{-y^3} \,  B(1,1,1,1)\, \Psi_ {y^3} \, \Phi_ {-\frac{x^3}{2}} \, \Psi_{y^3} \\*[1ex]
B_{(2,2)} &=& \Psi_{-y^2} \, \Phi_{\frac{-x^2}{4}} \, \Psi_{-2y^2} \,  B(2,2)\, \Psi_ {2y^2} \, \Phi_ {\frac{x^2}{4}} \, \Psi_{y^2} \, ,
\end{eqnarray*}
where
\begin{eqnarray*}
B(4) &=&
T \ltimes \{ \Psi_{q(y)}    \, | \,   q(y) \in y^4\c[y] \}  \, , \\*[1ex]
B(1,3) &=&
T \ltimes \{ \Psi_{q(y)} \, | \,   q(y) \in \c y^2+ y^4 \c[y] \}\ ,\\*[1ex]
B(1,1,2) &=&
T \ltimes \{ \Psi_{q(y)}    \, | \,   q(y) \in \c y^2+ y^4 \c[y] \} \ ,\\*[1ex]
B(1,1,1,1) & = &
T \ltimes \{\Psi_{q(y)}    \, | \,   q(y) \in y^4\c[y] \}\ , \\*[1ex]
B(2,2) & = &
T \ltimes \{ \Psi_{q(y)}    \, | \,   q(y) \in y^3\c[y]\}\  .
\end{eqnarray*}
\bibliographystyle{amsalpha}

\end{document}